\documentclass[11pt]{amsart}
\usepackage{amsmath,amscd,amssymb}
\usepackage[mathscr]{eucal}
\usepackage{amsmath}
\usepackage{enumerate}
\usepackage{color}
\usepackage{mathtools}

\DeclareMathOperator{\arccot}{arccot}

 \newtheorem{thm}{Theorem}[section]
 \newtheorem{cor}[thm]{Corollary}
 \newtheorem{lemma}[thm]{Lemma}
 \newtheorem{prop}[thm]{Proposition}

 \theoremstyle{definition}

  \newtheorem{exas}[thm]{Examples}
	\newtheorem{exa}[thm]{Example}
 \theoremstyle{remark}
 \newtheorem{rem}[thm]{Remark}
 \newtheorem{rems}[thm]{Remarks}

\newtheorem{remark}[thm]{Remark}

\numberwithin{equation}{section}

\def\R{\mathbb{R}}
\def\C{\mathbb{C}}
\def\N{\mathbb{N}}

\def\a{\alpha}
\def\b{\beta}
\def\ep{\epsilon}
\def\l{\lambda}
\def\B{L}
\def\Im#1{\operatorname{Im}#1}
\def\Re#1{\operatorname{Re}#1}
\def\Sect#1{\operatorname{Sect}#1}
\def\wt{\widetilde}

\def\Bes{\mathcal{B}}

\def\Bq{{\Bes_0}}
\def\D{\mathcal{D}}

\def\LT{\mathcal{LM}}
\def\LL{\mathcal{L}L^1}
\def\lt{\mathcal{L}}

\def\H{\mathcal{H}}

\def\Hol{\operatorname{Hol}}
\def\embedr{\overset{r}{\hookrightarrow}}
\def\embedi{\overset{i}{\hookrightarrow}}

\begin{document}

\title[Functional calculi for sectorial operators]{Functional calculi for sectorial operators and related function theory}

\author{Charles Batty}
\address{St. John's College\\
University of Oxford\\
Oxford OX1 3JP, UK
}

\email{charles.batty@sjc.ox.ac.uk}

\author{Alexander Gomilko}
\address{Faculty of Mathematics and Computer Science\\
Nicolas Copernicus University\\
Chopin Street 12/18\\
87-100 Toru\'n, Poland
}

\email{alex@gomilko.com}

\author{Yuri Tomilov}
\address{
Institute of Mathematics\\
Polish Academy of Sciences\\
\'Sniadeckich 8\\
00-656 Warsaw, Poland\\
and Faculty of Mathematics and Computer Science\\
Nicolas Copernicus University\\
Chopin Street 12/18\\
87-100 Toru\'n, Poland
}

\email{ytomilov@impan.pl}

\begin{abstract}
We construct two bounded functional calculi for sectorial operators on Banach spaces, which enhance the
functional calculus for analytic Besov functions, by extending the class of functions, generalizing and sharpening estimates, and adapting the calculus to the angle of sectoriality.     
The calculi are based on appropriate reproducing formulas, 
they are compatible with standard functional calculi and they admit appropriate 
convergence lemmas and spectral mapping theorems.
To achieve this, we develop the theory of associated function spaces in ways which are interesting and significant.
As consequences of our calculi, we derive several well-known operator norm-estimates and provide generalizations of some of them.
\end{abstract}



\thanks{This work was partially supported financially by a Leverhulme Trust Visiting Research Professorship and an NCN grant UMO-2017/27/B/ST1/00078, and inspirationally by the ambience of the Lamb \& Flag, Oxford.}

\date \today

\maketitle
\begin{center}
Dedicated to Vladimir M\"uller on the occasion of his $70$th birthday
\end{center}

\tableofcontents
\section{Introduction}

The theory of functional calculi forms a basis for the
study of sectorial operators and semigroup generators.
In particular, 
there were two functional calculi used extensively in the
research on operator semigroups and sectorial operators through the past
fifty years.
One of them, the Hille-Phillips (HP) functional calculus for semigroup generators, 
 probably stemmed from the foundational monograph \cite{HP},
and it became an indispensable part of semigroup theory. 
The systematic approach to the other one, the holomorphic functional calculus
for sectoral operators,
was initiated by McIntosh and his collaborators in the 1980s.
While the two calculi appeared to be very useful in applications,
the operator norm-estimates within them are often problematic. 
The estimates within the HP-calculus are direct but rather crude,
and the task of getting bounds within the holomorphic functional calculus
is a priori cumbersome since the calculus is not, in general, a bounded Banach algebra homomorphism.

  To circumvent those problems,  a number of additional tools and methods appeared in the literature.
In particular, the advanced notions and techniques related to bounded $H^\infty$-calculus, $R$-boundedness, Fourier multipliers and transference were developed in depth, and one may consult \cite{Hytonen}, \cite{KW} and \cite{HaaseB} 
for many of these function-theoretical developments.
Moreover, various implications of positivity of functions and their derivatives (completely monotone, Bernstein, $\mathcal{NP}_+$-functions) were adjusted to the operator-theoretical set-up.
For clarification of the role of positivity, see  \cite{Schill}, \cite{GT15} and \cite{BGTad}, for example.

Recently, in \cite{BGT}, a new functional calculus was constructed, 
the so-called $\mathcal B$-calculus. First of all, the $\mathcal B$-calculus  offers 
a simple and efficient route to operator norm-estimates for functions of semigroup generators, thus 
unifying a number of estimates in the literature and leading to new ones.
No supplementary arguments are required and the estimates underline the strength of the calculi.
Moreover, the $\mathcal B$-calculus possesses all attributes of classical functional calculi, see \cite{BGT}.
When combined properly they lead to new spectral mapping theorems
and generalizations of fundamentals of semigroup theory, see \cite{BGT2}.
To put our results into a proper context and to use some of the $\mathcal B$-calculus properties
in the sequel, we  briefly recall the set-up for the $\mathcal B$-calculus, see \cite{BGT} for more details.

Let $\mathcal B$ be the algebra of holomorphic functions on the right half-plane $\C_+$ such that
\begin{equation}  \label{bdef0}
 \|f\|_\Bq:=\int_{0}^{\infty} \sup_{\b \in \R }|f'(\a+i\b)| \, d\a <\infty.
\end{equation}
These functions
have been considered in some detail in \cite{BGT} (see also \cite{V1}). 
In particular, every $f \in \mathcal B$ belongs to $H^\infty(\mathbb C_+)\cap C(\overline{\mathbb C}_+)$, and $\mathcal B$ is a Banach algebra with the norm
\begin{equation}\label{b-norm}
\|f\|_{\Bes}:=\|f\|_\infty+ \|f\|_\Bq.
\end{equation}
Moreover, the algebra $\mathcal B$ modulo constants is isomorphic to the holomorphic Besov space $B^{+}_{\infty,1}(\mathbb R),$ see \cite[Proposition 6.2]{BGT}.
In the setting of power bounded operators on Hilbert spaces, the unit disc counterpart of $\mathcal B$ was employed for the study of functional calculi in \cite{Pel}.

Let $A$ be a densely defined closed operator on a Banach space $X$ such that
$\sigma(A) \subseteq \overline{\C}_+$, and
\begin{equation} \label{8.1}
\sup_{\alpha >0} \alpha \int_{\mathbb R} |\langle (\alpha +i\beta + A)^{-2}x, x^* \rangle| \, d\beta <\infty
\end{equation}
for all $x \in X$ and $x^* \in X^*$.  
This class of operators includes two substantial subclasses, namely the negative generators of bounded $C_0$-semigroups on Hilbert spaces $X$  and the negative generators of (sectorially) bounded holomorphic $C_0$-semigroups on Banach spaces $X$.  On the other hand, every operator in the class is the negative generator of a bounded $C_0$-semigroup.

The study of functional calculus based on the algebra $\mathcal B$ was initiated in \cite{White}  for generators of bounded semigroups on Hilbert spaces and in \cite{V1} for generators of holomorphic semigroups.   These works adapted and extended the approach from
\cite{Pel} to a more demanding and involved setting of unbounded operators.  Most researchers were unaware of \cite{White} until it became accessible a few years ago.  Meanwhile the line of research put forward in \cite{White} and  \cite{V1} was continued in  \cite{Haase} and \cite{Sch1} proceeding in two different directions (additional related references can be found in \cite{BGT}).
In  \cite{Haase}, by means of a new transference technique,
 counterparts of the results from \cite{V1} were proved in the framework
of bounded semigroups on Hilbert space and certain substantial subclasses of $\mathcal  B$, while \cite{Sch1} offered a number of generalisations and improvements of estimates from \cite{V1}. 
Only \cite{Sch1} and \cite{V1} considered all functions in
$\mathcal B$, applied to generators of bounded holomorphic semigroups in both papers. In \cite{BGT}, we
introduced a bounded $\mathcal B$ functional calculus for all operators satisfying \eqref{8.1},
and we extended the theory in \cite{BGT2}.

For $f \in \Bes$, set
\begin{align}  \label{fcdef}
\lefteqn{\langle f(A)x, x^* \rangle} \quad\\
&=  f(\infty) \langle x, x^* \rangle - \frac{2}{\pi}  \int_0^\infty  \alpha \int_{\R}  {f'(\alpha +i\beta)} \langle (\alpha -i\beta +A)^{-2}x, x^* \rangle  \, d\beta\,d\alpha  \nonumber 
\end{align}
for all $x \in X$ and $x^* \in X^*$, where $f(\infty) = \lim_{\Re z\to\infty} f(z)$. Using \eqref{8.1} and the definition of $\mathcal B$ (and the Closed Graph Theorem), it is easy to show that  $f(A)$ is a bounded linear mapping from $X$ to $X^{**}$, and that the linear mapping
\[
\wt\Phi_A : \Bes \to \mathcal L(X,X^{**}), \qquad f \mapsto f(A),
\]
is bounded.

It was discovered in \cite{BGT} that much more is true.  If $A$ belongs to any of the classes of semigroup generators mentioned above, then the formula \eqref{fcdef} defines a bounded algebra homomorphism
\[
\Phi_A: \Bes \to  L (X),\qquad \Phi_A (f):=f(A).
\]
It is natural to call the homomorphism $\Phi_A$ the ($\Bes$-)calculus of $A$.
It was proved in \cite{BGT} that $\Phi_A$ possesses  a number of useful properties. 
In particular, it admits the spectral inclusion (spectral mapping, in the case of bounded holomorphic semigroups) theorem and a Convergence Lemma of appropriate form.  The utility of the $\mathcal B$-calculus  depends on the facts that it (strictly) extends the Hille-Phillips (HP-) calculus and it is compatible with the holomorphic functional calculi for sectorial and half-plane type operators.

Moreover, the $\Bes$-calculus $\Phi_A$ is the only functional calculus that one can define for $A$ satisfying \eqref{8.1} and for functions in $\Bes$.  Indeed, let $A$ be an operator on $X$ with dense domain, and assume that $\sigma(A) \subseteq \overline{\C}_+$.   A \emph{(bounded) $\Bes$-calculus} for $A$ is, by definition, a bounded algebra homomorphism  $\Phi : \Bes \to L(X)$ such that  $\Phi((z+\cdot)^{-1}) = (z+A)^{-1}$ for all $z \in \C_+$.   As shown in \cite{BGT2}, if $A$ admits a $\mathcal B$-calculus, then the resolvent assumption \eqref{8.1} holds, and the calculus is $\Phi_A$.

While the $\mathcal B$-calculus is optimal and unique for generators of Hilbert space semigroups,
the situation is far from being so for generators of bounded holomorphic semigroups on Banach spaces
(as this paper will, in particular, show).
Thus using the $\mathcal B$-calculus ideology as a guiding principle, it is natural to try to extend
it beyond the Besov algebra $\mathcal B$ keeping all of its useful properties such as 
availability of good norm-estimates,
spectral mapping theorems, Convergence Lemmas, compatibility with the other calculi, etc.
Moreover, it is desirable to cover all sectorial operators regardless of their sectoriality angle.

In this paper we will construct some functional calculi encompassing wider classes of functions (including some with singularities on $i\mathbb R$) and providing finer estimates for all negative generators of (sectorially) bounded holomorphic semigroups, and eventually for all sectorial operators.  
Functional calculi for generators of some classes of bounded holomorphic semigroups were constructed in \cite{GaPy97}, \cite{GaMiPy02}, \cite{GaMi06},  \cite{GaMi08}, and \cite{KrieglerWeis}. 
However, most of the results in those papers concern sectorial operators of angle zero, and the approaches there are based on fine estimates for the corresponding semigroups rather than fine analytic properties of resolvents. 

Let $A$ be a densely defined sectorial operator of sectorial angle $\theta_A \in [0,\pi)$ on a Banach space $X$.   It is well known that $-A$ is the generator of a (sectorially) bounded holomorphic $C_0$-semigroup on $X$ if and only if $A$ is sectorial and $\theta_A < \pi/2$ (we may write $A \in \Sect(\pi/2-)$ for this class).
In this paper we address the question whether the $\Bes$-calculus for $A$ can be extended to more functions.
Since the resolvent of $A$ satisfies the estimate
\begin{equation*} \label{Aa0_intro}
M_{\psi}(A) := \sup_{z\in\Sigma_{\pi-\psi}} \|z(z+A)^{-1}\| < \infty, \qquad z \in \Sigma_{\pi-\psi}.
\end{equation*}
for all $\psi\in(\theta_A,\pi)$,
a direct way to define an appropriate function algebra would be to introduce a Banach space
of functions $f$ which are holomorphic on sectors $\Sigma_{\psi}:=\{z\in \mathbb C: |\arg(z)| < \psi \}$ such that 
\begin{equation}\label{assump_f}
\|f\|_{\psi}:=\int_{\partial \Sigma_{\psi}} \frac{|f(z)|}{|z|}|dz| <\infty.
\end{equation}
In order to apply this to all $A \in \Sect(\pi/2-)$, $f$ should be holomorphic on $\C_+$ and 
the assumption \eqref{assump_f} should hold for all $\psi \in (0,\pi/2)$,
and in order to provide an estimate which is uniform in $\theta$ it is desirable to have
$\sup_{\psi\in  (0,\pi/2)} \|f\|_{\psi}<\infty$.  To our knowledge, no spaces of this type have been studied systematically in the literature, although they appear naturally in \cite[Appendix H2 and Chapter 10.2]{Hytonen},  \cite[Section 6]{Haase_Sp}  or \cite[Appendix C]{Haak_H}.   This class of functions is strictly included in each of the spaces $\mathcal D_s, s>0$ (see Proposition \ref{weight1} and a discussion following it), which we now define.   

To define a functional calculus for all $A \in \Sect(\pi/2-)$, we let $\mathcal{D}_s, s >-1$,  be the linear space of all holomorphic functions $f$ on $\C_{+}$ such that
\begin{equation}\label{vintro}
\|f\|_{\mathcal D_{s,0}}:= \int_0^\infty \alpha^s\int_{-\infty}^\infty\frac{|f'(\alpha +i\beta)|}{(\alpha^2+\beta^2)^{(s+1)/2}}\,d\beta\,d\alpha <\infty.
\end{equation}
If $f \in \mathcal D_s$,  then  there exists
a finite limit $f(\infty):=\lim_{|z|\to\infty,\,z\in {\Sigma}_\psi}\,f(z)$ for all $\psi\in(0.\pi/2)$.  

For every $s>-1$ the space $\mathcal D_s$ equipped with the norm
\[
\|f\|_{\mathcal D_s}:= |f(\infty)|+\|f\|_{\mathcal D_{s,0}}, \qquad f \in\mathcal D_s,
\]
is a Banach space, but not an algebra.  However, the spaces $\D_s$ increase with $s$, and we prove in Lemma \ref{Alg1} that
\[
\mathcal D_\infty:=\bigcup_{s>-1} \mathcal D_s
\]
 is an algebra. 

Let $f\in \mathcal{D}_\infty$, so $f \in \mathcal{D}_s$ for some $s>-1$, and let $A$ be sectorial with $\theta_A < \pi/2$.
Define
\begin{equation}\label{formulaD_intro}
f_{\mathcal{D}_s}(A):=f(\infty)-
\frac{2^s}{\pi}\int_0^\infty \alpha^s\int_{-\infty}^\infty f'(\alpha+i\beta)(A+\alpha-i\beta)^{-(s+1)}\,d\beta\,d\alpha.
\end{equation}
Then $f_{\D_\sigma}(A) = f_{\D_s}(A)$ whenever $\sigma>s$.   The following result sets out other properties of this functional calculus.   The proof will be given in Section \ref{defD}.

\begin{thm} \label{dcalculus_intro}
Let A be a densely defined, closed operator on a Banach space $X$ such that $\sigma(A) \subset \overline{\C}_+$.   The following are equivalent:
\begin{enumerate}[\rm(i)]
\item $A \in {\rm Sect}(\pi/2-);$
\item There is an algebra homomorphism $\Psi_A : \D_\infty \to L(X)$ such that
\begin{equation*}
\Psi_A((z+\cdot)^{-1}) = (z+A)^{-1}, \qquad z \in \C_+, 
\end{equation*}
and $\Psi_A$ is bounded in the sense that there exist constants $C_s, \, s>-1$, such that, for every $f \in \mathcal D_s$, 
\begin{equation}\label{bounded_intro}
\|\Psi_A(f)\|\le C_s \|f\|_{\mathcal D_s}.
\end{equation}
\end{enumerate}
When these properties hold, $\Psi_A$ is unique, and it is defined by the formula \eqref{formulaD_intro}:
\begin{eqnarray*}
\Psi_A : \mathcal D_{\infty} \mapsto L(X), \qquad \Psi_A(f)=f_{\mathcal{D}_s}(A), \quad f \in \mathcal{D}_s.
\end{eqnarray*}
\end{thm}

The homomorphism $\Psi_A$ will be called the {\it $\mathcal D$-calculus} for $A$.  It will be shown in Section \ref{defD} that the $\D$-calculus is compatible with the Hille-Phillips calculus and the holomorphic calculus for sectorial operators, and a spectral mapping theorem is given in Theorem \ref{SMT}.  Corollary \ref{Compat} provides a version of this functional calculus based on the Banach algebra $H^\infty(\mathbb C_+)\cap \mathcal D_s$ for a fixed value of $s$.  

The $\mathcal D$-calculus defined as above does not take into account the  sectoriality angle of $A \in {\rm Sect}(\pi/2-)$. 
However, it can be used to construct a functional calculus which does not have this drawback.  
To achieve this aim we introduce the Hardy-Sobolev spaces $\H_\psi$, on sectors $\Sigma_\psi$. 
First, for any $\psi\in (0,\pi)$, we define the Hardy space $H^1(\Sigma_\psi)$ as the linear space of functions $f\in \Hol(\Sigma_\psi)$ such that
\begin{equation}\label{CC2_intro}
\|f\|_{H^1(\Sigma_{\psi})}:=\sup_{|\varphi|< \psi}\,
\int_0^\infty \bigl(|f(te^{i\varphi})|+
|f(te^{-i\varphi})| \bigr) \,dt <\infty.
\end{equation}
Note that $H^1(\Sigma_{\pi/2})$ coincides with the classical Hardy space $H^1(\mathbb C_+)$ in the right half-plane $\mathbb C_+$.   It is well-known that $(H^1(\Sigma_\psi), \|\cdot\|_{H^1(\Sigma_\psi)})$ is a Banach space, and every $f \in H^1(\Sigma_\psi)$ has a boundary function  on $\partial \Sigma_\psi$.   The boundary function exists as the limit of $f$ in both an $L^1$-sense and a pointwise (a.e.) sense.  Moreover, the norm of $f$ in $H^1(\Sigma_\psi)$ is attained by the $L^1$-norm of its boundary function.   See Section \ref{hardy_rays} for a succinct approach to the Hardy spaces on sectors.

The space $H^1(\Sigma_\psi)$ induces the corresponding Hardy-Sobolev space $\H_\psi$ on $\Sigma_\psi$ as
\[
\H_\psi:=\left\{f \in \Hol(\Sigma_\psi): f' \in H^1(\Sigma_\psi)\right\}.
\]
Any function $f \in \H_\psi$ has a finite limit $f(\infty):=\lim_{t \to \infty}f(t)$,
and moreover $f \in H^\infty(\Sigma_\psi)$.  Then $\H_\psi$ becomes a Banach algebra in the norm
\[
\|f\|_{\mathcal \H_\psi}:=\|f\|_{H^\infty(\Sigma_\psi)} +\|f'\|_{H^1(\Sigma_\psi)}, \quad f \in \H_\psi.
\]
The relationship between these spaces and the spaces $\mathcal D_s$ for all $s>-1$ is set out in Corollary \ref{HHH} and Lemma \ref{Dh1}; in particular, for each $s>-1$, $\H_{\pi/2}$ is contained in $\D_s$, and $\D_s$ is embedded in $\H_\psi$ for $\psi<\pi/2$.

To make use of the angle of sectoriality of $A$, we can adjust the $\mathcal D$-calculus to sectors as follows.
If $f\in \H_\psi$ where $\psi  \in (\theta_A,\pi)$, $\gamma = \pi/(2\psi)$, and $f_{1/\gamma}(z):=f(z^{1/\gamma})$, then 
$f'_{1/\gamma}\in H^1(\C_{+})$ and $f_{1/\gamma}(\infty)=f(\infty)$, and hence $f_{1/\gamma}\in \mathcal{D}_0$.   This observation allows us to extend the $\mathcal D$-calculus to the class of all sectorial operators, and makes the next definition (based on the $\mathcal{D}$-calculus) natural and plausible.

If $A$ is sectorial and $\psi \in (\theta_A,\pi)$, define
\begin{equation}\label{sigma_def_intro}
f_{\mathcal H}(A):= f(\infty)-\frac{1}{\pi}\int_0^\infty\int_{-\infty}^\infty f'_{1/\gamma}(\a+i\b)
(A^{\gamma}+\a-i\b)^{-1}\, d\b\,d\a.
\end{equation}
One can prove (see \eqref{A2Dop}) that
\begin{equation}\label{A2Dop_intro}
\|f_{\mathcal H}(A)\|\le |f(\infty)|+\frac{M_{\pi/2}{(A^\gamma)}}{\pi}\|f_{1/\gamma}\|_{\mathcal{D}_0}
\le  |f(\infty)|+ M_{\pi/2}({A^\gamma}) \|f\|_{\H_\psi}.
\end{equation}

Then \eqref{sigma_def_intro} and \eqref{A2Dop_intro} hold for any $\gamma \in (1,\pi/(2\theta))$, and the definition of $f_\H(A)$ does not depend on the choice of $\psi$.

Now we are able to formalize our extension of the $\mathcal D$-calculus as follows.

\begin{thm}\label{sigma_cal_intro}
Let $A$ be a densely defined operator on a Banach space $X$ such that $\sigma(A) \subset \overline{\Sigma}_\theta$, where $\theta \in [0,\pi)$.   The following are equivalent:
\begin{enumerate}[\rm(i)]
\item $A \in \Sect(\theta);$
\item For each $\psi \in (\theta,\pi)$, there is a bounded Banach algebra homomorphism $\Upsilon_A : \mathcal \H_\psi \mapsto L(X)$ such that 
$\Upsilon_A((z+\cdot)^{-1}) = (z+A)^{-1}, \quad z \in \Sigma_{\pi-\psi}$.
\end{enumerate}
When these properties hold, the homomorphism $\Upsilon_A$ is unique for each value of $\psi$, and it is defined by the formula \eqref{sigma_def_intro}:
\begin{equation*}
\Upsilon_A: \H_\psi \to L(X), \qquad \Upsilon_A(f) =
f_\H(A), \quad f \in \H_\psi.
\end{equation*}
\end{thm}

The homomorphism $\Upsilon_A$ will be called the {\it $\mathcal H$-calculus} for $A$.  

The $\D$-calculus can be given a more succinct form,
by replacing \eqref{sigma_def_intro} with the somewhat more transparent
formula \eqref{formula_a_Intro} below, inspired by results in \cite{Boy}.

\begin{thm}\label{Sovp_Intro}
Let $A \in \Sect(\theta)$, $\theta < \psi < \pi$, and $\gamma=\pi/(2\psi)$. For $f\in \H_\psi$,  
let
\[
f_\psi(s):=\frac{f(e^{i\psi}t)+f(e^{-i\psi}t)}{2}, \quad t \ge 0.
\]
Then
\begin{equation}\label{formula_a_Intro}
f_{\mathcal H}(A)= f(\infty)-\frac{2}{\pi}\int_0^\infty
f_\psi'(t){\arccot}(A^\gamma/t^{\gamma})\,dt
\end{equation}
where the integral converges in the uniform operator topology, and
\begin{equation*}
\|f_{\mathcal H}(A)\|\le |f(\infty)|+ M_\psi(A)\|f_\psi'\|_{L^1(\R_+)} \le M_\psi(A) \|f\|_{\H_\psi}.
\end{equation*}
Thus $\|\Upsilon_A\| \le M_\psi(A)$.
\end{thm}

See Section \ref{alternative} for details.   The $\mathcal D$-calculus and the $\mathcal H$-calculus possess natural properties of functional calculi such as spectral mapping theorems and Convergence Lemmas. These properties are studied in Section \ref{SMT_sec}.

The strength of the constructed calculi is illustrated by several examples showing
that they lead to sharper estimates than those offered by other calculi (see Section \ref{VS} for one example). 
Moreover, the theory developed in this paper is successfully tested by deriving several significant estimates for functions of sectorial operators from the literature.   In particular, in Section \ref{norm_estimates}, we provide a proof of permanence of the class of sectorial operators under subordination and we revisit a few basic results from semigroup theory.

In developing the $\mathcal D$- and $\mathcal H$-calculi we prove a number of results of independent interest in function theory. Apart from the theory of the spaces $\mathcal D_s$ and $\H_\psi$, their reproducing formulas, and boundedness of the associated operators elaborated in this paper, we emphasize the property \eqref{n_equality} in Corollary \ref{HHH} yielding isometric coincidence of spaces of Hardy type, Theorem \ref{ACM} on Laplace representability of Hardy-Sobolev functions, and Theorem \ref{DM122} on the density of rational functions in Hardy-Sobolev spaces.

{\it Added Note:}  During the preparation of this paper, we have become aware of a paper by Arnold and Le Merdy \cite{ALM} who consider negative generators of bounded $C_0$-semigroups on Hilbert space.  Inspired by ideas in \cite{Pel} for the discrete case, they extend the $\Bes$-calculus for those operators, to a strictly larger Banach algebra $\mathcal{A}$ in which $\Bes$ is continuously embedded.   Their extension is complementary to our extensions to the $\mathcal{D}$- and $\mathcal{H}$-calculi for negative generators of bounded holomorphic $C_0$-semigroups on Banach spaces.   We are grateful to Loris Arnold for pointing out several defects in the original version of this paper.

\section{Preliminaries} \label{prelims}

\subsection*{Notation} 
Throughout the paper, we will use the following notation:
\begin{enumerate}[\null\hskip1pt]
\item $\R_+ :=[0,\infty)$,
\item $\C_+ := \{z \in\C: \Re z>0\}$, $\overline{\C}_+ = \{z \in\C: \Re z\ge0\}$,
\item $\Sigma_\theta := \{z\in\C: z \ne 0, |\arg z|<\theta\}$ for $\theta \in (0,\pi)$.
\item
\end{enumerate}
For $f : \C_+ \to \C$, we say that $f$ has a \emph{sectorial limit at infinity} if
\[
\lim_{|z|\to\infty, z \in \Sigma_\psi} f(z)
\]
exists for every $\psi \in (0,\pi/2)$.   Similarly, $f$ has a \emph{sectorial limit at $0$} if
\[
\lim_{|z|\to0, z \in \Sigma_\psi} f(z)
\]
exists for every $\psi \in (0,\pi/2)$.  We say that $f$ has a \emph{half-plane limit at infinity} if 
\[
\lim_{\Re z \to \infty} f(z)
\]
exists in $\C$.  We say that $f$ has a \emph{full limit at infinity} or \emph{at zero} if 
\[
\lim_{|z| \to \infty, z\in\C_+} f(z) \quad \text{or} \quad  \lim_{|z| \to 0, z\in\C_+} f(z)
\]
exists in $\C$. The notation $f(\infty)$ and $f(0)$ may denote a sectorial limit, a half-plane limit, or a full limit, according to context.  

\noindent
For $a \in \overline{\C}_+$, we define functions on $\C$ by
\[
e_a(z) = e^{-az}; \qquad r_a(z) = (z+a)^{-1}, \, z\ne-a.
\]

\noindent
We use the following notation for spaces of functions or measures, and transforms, on $\R$ or $\R_+$:
\begin{enumerate}[\null\hskip1pt]
\item  $\operatorname{Hol}(\Omega)$ denotes the space of holomorphic functions on an open subset $\Omega$ of $\C$, $H^\infty(\Omega)$ is the space of bounded holomorphic functions on $\Omega$, and $\|f\|_{H^\infty(\Omega)} = \sup_{\Omega} |f(z)| $.
\item $H^p(\C_+), \, 1 \le p \le \infty$,  are the standard Hardy spaces on the (right) half-plane.
\item $M(\R_+)$ denotes the Banach algebra of all bounded Borel measures on $\R_+$ under convolution.   We identify $L^1(\R_+)$ with a subalgebra of $M(\R_+)$ in the usual way.  We write $\mathcal L\mu$ for the Laplace transform of $\mu \in M(\R_+)$.
\item $\LT$ is the Hille-Phillips algebra, $\LT := \{\mathcal{L}\mu : \mu \in M(\R_+)\}$, with the norm $\|\mathcal{L}\mu\|_{\mathrm{HP}}:= |\mu|(\R_+)$, and $\LL := \{\lt f: f \in L^1(\R_+)\}$.
\item $dS$ denotes area measure on $\C_+$.
\end{enumerate}

\noindent
For a Banach space $X$, $X^*$ denotes the dual space of $X$ and $L(X)$ denotes the space of all bounded linear operators on $X$. The domain, spectrum and resolvent set of an (unbounded) operator $A$ on $X$ are denoted by $D(A)$, $\sigma(A)$ and $\rho(A)$, respectively.

If $(\mathcal{X},\|\cdot\|_{\mathcal{X}})$ and $(\mathcal{Y},\|\cdot\|_{\mathcal{Y}})$ are normed spaces of holomorphic functions on domains $\Omega_{\mathcal{X}}$ and $\Omega_{\mathcal{Y}}$, we will use notation as follows:
\begin{enumerate}[\null\hskip1pt]
\item  $\mathcal{Y} \embedi \mathcal{X}$ if $\Omega_{\mathcal{Y}} = \Omega_{\mathcal{X}}$, $\mathcal{Y}$ is a subset of $\mathcal{X}$ and the inclusion map is continuous;
\item $\mathcal{Y} \subset \mathcal{X}$ if $\Omega_{\mathcal{Y}} = \Omega_{\mathcal{X}}$, $\mathcal{Y}$ is a subset of $\mathcal{X}$ and $\mathcal{Y}$ inherits the norm from $\mathcal{X}$;
\item  $\mathcal{Y} \embedr \mathcal{X}$ if $\Omega_{\mathcal{Y}} \supset \Omega_{\mathcal{X}}$, and the restriction map $f \mapsto f|_{\Omega_{\mathcal{X}}}$ is a continuous map from $\mathcal{Y} \to \mathcal{X}$.
\end{enumerate}
Boundaries of all of the sectors appearing in this paper will be oriented from top to bottom.
\subsection*{Elementary inequalities}
We will need the following elementary lemma which gives lower bounds for $|z+\l|$ in terms of $|z|$ and $|\l|$, for $z$, $\lambda \in \C$.

\begin{lemma}  \label{trig}
\begin{enumerate}[\rm(i)]
\item
Let $z=|z|e^{i\psi}$ and $\lambda=|\lambda|e^{i\varphi} \in \C$, where $|\psi-\varphi|\le\pi$.  Then
\begin{equation} \label{Fc}
|z+\lambda|\ge \cos\left(\frac{\psi-\varphi}{2}\right) (|z|+|\lambda|).
\end{equation}
\item Let $z \in \overline{\Sigma}_{\psi}$ and $\l \in \overline{\Sigma}_\varphi$, where $\psi,\varphi>0$ and $\varphi+\psi<\pi$.  Then
\begin{equation} \label{Fd}
|z+\l| \ge \cos\left(\frac{\psi+\varphi}{2}\right) (|z|+|\lambda|).
\end{equation}
\item Let $z=|z|e^{i\psi}$ and $\lambda=|\lambda|e^{i\varphi} \in \C$, where $|\psi|<\pi/2$ and $|\varphi|\le \pi/2$.  Then
\begin{equation}\label{Fa}
|z+\lambda|\ge \cos\psi \, |\lambda|,
\end{equation}
and
\begin{equation}\label{Fb}
|z+\lambda|\ge \cos\psi\, |z|.
\end{equation}
\end{enumerate}
\end{lemma}

\begin{proof}
For \eqref{Fc}, we may assume that $\varphi\ge\psi$ and let $\theta := (\pi-\varphi+\psi)/2 \in[0,\pi/2]$.  By applying a rotation of $\C$ we may further assume that $\varphi=\pi-\theta$ and $\psi = \theta$.   Then
\[
|z+\l| \ge \Im z + \Im\l = \sin\theta (|z|+|\l|) = \cos \left(\frac{\varphi-\psi}{2}\right) (|z|+|\l|).
\]
The inequality \eqref{Fd} follows from \eqref{Fc}, since $\psi+\varphi$ is the maximum value of $|\psi'-\varphi'|$ for $\psi'\in[-\psi,\psi]$ and $\varphi' \in [-\varphi,\varphi]$.

The inequality \eqref{Fb} is obtained by considering $\Re(z+\l)$.  For the inequality \eqref{Fa}, we assume without loss of generality that $\sin\varphi \ge 0$.   Note that
\begin{align} \label{zl}
|\l+z|^2 -|\l|^2\cos^2\psi &= \left(|z|+|\l|\cos(\varphi-\psi)\right)^2 + |\l|^2 \left(\sin^2\varphi - \cos^2(\varphi-\psi)\right) \\
&= |z|^2 + 2|z|\,|\l| \cos(\varphi-\psi) + |\l|^2 \sin^2 \varphi.  \nonumber
\end{align}
 If $\cos(\varphi-\psi)<0$, we have
\[
\sin\varphi - |\cos(\varphi-\psi)| = \sin\varphi \, (1+\sin\psi) + \cos \varphi\cos\psi \ge 0.
\]
Then the expression on the right-hand side of the first line of \eqref{zl} is clearly non-negative.    If $\cos(\varphi-\psi)\ge0$, then the expression in the second line is clearly non-negative.   This completes the proof.
\end{proof}

\subsection*{Beta function}
The Beta function appears in many places in the paper. It is defined for $s,t>0$ by
\[
B(s,t) = B(t,s) := \int_0^1 \tau^{s-1} (1-\tau)^{t-1} \, d\tau = 2 \int_0^{\pi/2} \cos^{2s-1} \psi \sin^{2t-1} \psi \, d\psi.
\]
In particular, for $s>-1$ we will use the relations
\[
B\left(\frac{s+1}{2},\frac{1}{2}\right) = \int_{-\pi/2}^{\pi/2}\cos^s\psi\,d\psi = \int_{-\infty}^\infty \frac{dt}{(1+t^2)^{(s+1)/2}}
=\frac{\sqrt{\pi}\Gamma((s+1)/2)}{\Gamma(s/2+1)},
\] 
see \cite[p.\ 386]{Prud}.  We note also the following limit properties:
\[
\lim_{s\to -1}\,(s+1)\,B\left(\frac{s+1}{2},\frac{1}{2}\right)=2,\quad
\lim_{s\to\infty}\,\sqrt{s}B\left(\frac{s+1}{2},\frac{1}{2}\right)=\sqrt{2\pi}.
\]

\subsection* {Proof conventions}

We will make extensive use of the dominated convergence theorem, often for vector-valued functions.   With a few exceptions, we will not give details of the relevant dominating functions, as they are usually easily identified.

We will also use the following elementary lemma on several occasions. See \cite[p.21, Lemma 1]{Duren1} for a proof.

\begin{lemma} \label{duren21}
Let $(\Omega,\mu)$ be a $\sigma$-finite measure space, and $(f_n)_{n\ge1} \subset L^p(\Omega,\mu)$, where $p\in[1,\infty)$.  If $f_n \to f_0$ a.e., and $\|f_n\|_{L^p(\Omega,\mu)} \to \|f_0\|_{L^p(\Omega,\mu)}$, then $\|f_n-f_0\|_{L^p(\Omega,\mu)} \to 0$ as $n\to\infty$.
\end{lemma}

We will use Vitali's theorem several times, usually for holomorphic vector-valued functions.   We refer to the version given in \cite[Theorem A.5]{ABHN}.

Let $\mathcal{X}$ be a Banach space of holomorphic functions on a domain $\Omega_{\mathcal{X}}$ such that the point evaluations $\delta_z : f \mapsto f(z), \, z \in \Omega_{\mathcal{X}}$, are continuous on $\mathcal{X}$.   
Let $(\Omega,\mu)$ be either an interval in $\R$ with length measure or an open set in $\C$ with area measure, and $F : \Omega \to \mathcal{X}$ be a continuous function such that $\int_\Omega \|F(t)\|_{\mathcal{X}} \, d\mu(t) < \infty$.  
Then the integral
\[
G := \int_\Omega F(t) \, d\mu(t)
\] 
exists as a Bochner integral in $\mathcal{X}$ and it can be approximated by Riemann sums. It follows that $G$ belongs to the closed linear span of $\{F(t) : t\in \Omega\}$ in $\mathcal{X}$. 

Now assume that $F : \Omega \to \mathcal{X}$ is locally bounded, where $\Omega$ is an open set in $\C$, and that $\l \mapsto F(\l)(z)$ is holomorphic on $\Omega$ for all $z \in \Omega_{\mathcal{X}}$.  
We will use the fact that $F: \Omega \to \mathcal{X}$ is holomorphic in the vector-valued sense, without further comment.  
The result at this level of generality can be seen from \cite[Corollary A.7]{ABHN}, using the point evaluations as separating functionals.  
An alternative is to show that $F$ is continuous, and then apply Morera's theorem.   If the definition of $F$ is by an integral formula, it may also be possible to apply a standard corollary of the dominated convergence theorem which leads to an integral formula for the derivative $F'$.

\section{The Banach spaces $\mathcal{D}_s$ and their reproducing formulas}

In this section we introduce some spaces of holomorphic functions to which we will extend the $\Bes$-calculus of operators in Section \ref{defD} onwards.

\subsection{The spaces $\mathcal{V}_s$}

Let $s>-1$, $z=\alpha +i\beta$, and let $\mathcal{V}_s$ be the  Banach space of (equivalence classes of) measurable functions $g: \C_{+}\to \mathbb C$ such that the norm
\begin{align}\label{norm}
\|g\|_{\mathcal{V}_s}:&=  \int_{\C_+} \frac{(\Re z)^s |g(z)|}{|z|^{s+1}} \, dS(z) \\
&= \int_0^\infty \alpha^s\int_{-\infty}^\infty\frac{|g(\alpha +i\beta)|}{(\alpha^2+\beta^2)^{(s+1)/2}}\,d\beta\,d\alpha \notag\\
&=\int_{-\pi/2}^{\pi/2}\cos^s\varphi\int_0^\infty  |g(\rho e^{i\varphi})|\,d\rho\,d\varphi,\notag
\end{align}
is finite, where $S$ is the area measure on $\C_+$.
Note that
\begin{equation} \label{vst}
\mathcal{V}_s\subset \mathcal{V}_{\sigma} \qquad \text{and} \qquad  \|g\|_{\mathcal{V}_{\sigma}}\le \|g \|_{\mathcal{V}_s},
\qquad g\in \mathcal{V}_s,\; s < \sigma,
\end{equation}
and
\begin{equation} \label{vest}
\int_{\Sigma_\psi} \frac{|g(z)|}{|z|} \, dS(z) \le \max\left\{1,\frac{1}{\cos^s\psi}\right\} \|g\|_{\mathcal{V}_s}, \qquad g \in \mathcal{V}_s, \, \psi \in (0,\pi/2).
\end{equation}

The following property of functions from $\mathcal V_s$ is an essential element in the arguments which lead to the representations for functions in $\mathcal{V}_s$ in Proposition \ref{prim} and for $\D_s$ in Corollary \ref{Repr}, and eventually to the definition of a functional calculus for operators in \eqref{formulaD}.

\begin{lemma}\label{Derin}
Let $g\in \mathcal{V}_s$ be holomorphic, where $s>-1$.  For every $k\ge 1$ and every $\psi \in (0,\pi/2)$,
\begin{equation}\label{der1}
\lim_{|z|\to\infty ,\;z\in \Sigma_{\psi}}\,z^k g^{(k-1)}(z)=0.
\end{equation}
\end{lemma}

\begin{proof}
Let $g \in\mathcal V_s$ be holomorphic, $\psi\in (0,\pi/2)$, $\psi' = (\pi/2 + \psi)/2$ and $b_\psi={\sin((\pi/2-\psi)/2)} = \cos\psi'$.  If $z \in \Sigma_\psi$, then
\[
\{\l\in\C: |\l-z|\le b_\psi|z|\}\subset \{\l\in\Sigma_{\psi'}: |\l|\ge(1-b_\psi)|z|\}.
\]
Let $r \in (0, b_\psi|z|)$.  By Cauchy's integral formula for derivatives,
\[
g^{(k-1)}(z)=\frac{(k-1)!}{2\pi i}\int_{\{\l:\,|\l-z|=r\}}
\frac{g(\l)}{(\l-z)^{k}}\,d\l.
\]
Multiplying by $r^k$ and integrating with respect to $r$ over $(0, b_\psi|z|)$ gives
\[
\frac{(b_\psi|z|)^{k+1}}{k+1} |g^{(k-1)}(z)| \le \frac{(k-1)!}{2\pi} \int_{\{\l: |\l-z|\le b_\psi|z|\}} |g(\l)| \, dS(\l),
\]
and then
\[
|z|^k |g^{(k-1)}(z)| \le \frac{(k+1)(k-1)!(1+b_\psi)}{2\pi b_\psi^{k+1}} \int_{\{\l\in\Sigma_{\psi'}: |\l|\ge (1-b_\psi)|z|\}} \frac{|g(\l)|}{|\l|} \, dS(\l).
\]
It now follows from \eqref{vest} that $|z^k g^{(k-1)}(z)| \to 0$ as $|z|\to\infty, \, z \in \Sigma_\psi$.
\end{proof}

\subsection{The spaces $\D_s$ and the operators $Q_s$}  \label{3.02}

We now define a linear operator  $Q_s$ on $\mathcal{V}_s, \,  s>-1$.    It will play a similar role to the operator $Q$ on $\mathcal{W}$ considered in \cite[Section 3]{BGT2}, where $\mathcal{W}$ is the Banach space of all (equivalence classes of) measurable functions $g : \C_+ \to \C$ such that
\begin{equation} \label{3.2}
\|g\|_{\mathcal{W}} := \int_0^\infty \sup_{\beta\in \R} |g(\a+i\b)| \, d\a < \infty.
\end{equation}
Indeed the definition of $Q_1$ is formally the same as the definition of $Q$ in \cite{BGT2}, but the domain $\mathcal{V}_1$ of $Q_1$ is larger than $\mathcal{W}$.

For $g \in \mathcal{V}_s$, let
\begin{equation} \label{qdef}
(Q_s g)(z):= - \frac{2^s}{\pi}\int_0^\infty \alpha^s\int_{-\infty}^\infty\frac{g(\alpha+i\beta)}{(z+\alpha-i\beta)^{s+1}} \, d\b\,d\alpha,
\quad z\in \C_{+} \cup \{0\}.
\end{equation}
By \eqref{Fa}, the integral is absolutely convergent, and
\begin{equation}\label{angle}
|(Q_s g)(z)|\le \frac{2^s\|g\|_{\mathcal{V}_s}}{\pi \cos^{s+1}\psi}, \qquad z \in \Sigma_\psi, \, \psi \in (0,\pi/2).
\end{equation}
The dominated convergence theorem implies that $Q_s g$ is continuous on $\C_+$, with sectorial limits at infinity and $0$:
\begin{equation} \label{qinf}
\lim_{|z|\to\infty,\,z\in {\Sigma}_\psi}\,(Q_s g)(z)=0, \qquad \psi \in (0,\pi/2),
\end{equation}
and
\begin{equation} \label{qzero}
\lim_{|z|\to 0,\,z\in {\Sigma}_\psi}\,(Q_s g)(z)=(Q_s g)(0), \qquad \psi \in (0,\pi/2).
\end{equation}
Thus $Q_s g$ is bounded and continuous on $\overline{\Sigma}_\psi$ for every $\psi\in (0,\pi/2)$.
Moreover, $Q_s g$ is holomorphic on $\C_+$ and
\begin{equation} \label{qder}
(Q_s g)'(z) = (s+1)\frac{2^s}{\pi}\int_0^\infty \alpha^s\int_{-\infty}^\infty\frac{g(\alpha + i\beta)}{(z+\alpha-i\beta)^{s+2}}\,d\beta \, d\alpha,
\qquad z\in \C_{+}.
\end{equation}
Using this, (\ref{Fa}) and (\ref{Fb}), we obtain  that
\begin{equation}\label{Dangle}
|(Q_s g)'(z)|\le \frac{(s+1)2^s}{\pi\cos^{s+2}\psi \,|z|}\|g\|_{\mathcal{V}_s},\qquad z\in \Sigma_\psi,
\end{equation}
and
\begin{equation*}
|z(Q_s g)'(z)|\le \frac{(s+1)2^s}{\pi\cos^{s+2}\psi}\int_{\C_{+}} \frac{|z|(\Re \l)^s\,|g(\l)|}{|z+\overline{\l}||\l|^{s+1}}\,dS(\l), \qquad z\in \Sigma_\psi.
\end{equation*}
Using the dominated convergence theorem again, we obtain
\begin{equation} \label{Lebeg11}
\lim_{|z|\to0,z\in\Sigma_\psi} z(Q_sg)'(z) = 0.
\end{equation}

We now give another formula for  $Q_s$.   Let $s=n+\delta>-1$ where $n\in \N\cup\{-1,0\}$ and $\delta\in [0,1)$, and let
\begin{equation}\label{Const46}
C_s:=\int_0^\infty\frac{dt}{(t+1)^{n+2}t^\delta}=
\int_0^1\frac{(1-\tau)^s}{\tau^\delta}\,d\tau=B(1-\delta,s+1).
\end{equation}
Then
\[
\int_0^\infty \frac{dt}{(\l+t)^{n+2}t^\delta} = \frac{C_s}{\l^{s+1}},\quad \l\in \C_{+}.
\]
Indeed, both sides of this equation are holomorphic functions of $\l\in\C_+$, and they coincide for $\l \in (0,\infty)$, so they coincide for all $\l\in\C_+$, by the identity theorem for holomorphic functions.  
Putting $\l = z+\a-i\b$, we obtain
\begin{equation}\label{46RRR}
(Q_s g)(z)
=- \frac{2^s}{\pi C_s}\int_0^\infty \int_0^\infty\int_{-\infty}^\infty
\frac{\alpha^s\,g(\alpha+i\beta)\,d\beta\,d\alpha}{(z+\alpha-i\beta+t)^{n+2}}\,\frac{dt}{t^\delta},\quad z\in \C_{+}.
\end{equation}

For $s>-1$ let $\mathcal{D}_s$  be the linear space of all holomorphic functions $f$ on $\C_{+}$ such that
\[
f'\in \mathcal{V}_s,
\]
equipped with the semi-norm
\[
\|f\|_{\mathcal D_{s,0}}:=\|f'\|_{\mathcal{V}_s}, \qquad f \in \mathcal D_s.
\]

If $\sigma>s>-1$, then it is immediate from \eqref{vst} that $\D_s \subset \D_\sigma$. 
We will exhibit some functions in $\D_s$ later in this section and in Section \ref{some_functions}.

In the rest of this section we will obtain a reproducing formula \eqref{qnrep} for functions from $\D_s$ and we will describe some basic properties which will be relevant for the sequel. To this aim, we first define and study the behaviour of operators $Q_s$ on 
the scale of $\mathcal D_s$-spaces. Recall that    
in \cite[Proposition 3.1]{BGT2} we showed that $Q$ maps $\mathcal{W}$ into $\Bes$. However $Q_s$ does not map the whole of $\mathcal{V}_s$ into $\mathcal{D}_s$.  For $s > -1$, a function $g \in \mathcal{V}_s$ for which $Q_sg \notin \D_s$, can be defined as follows:
\begin{multline*}
g(\rho e^{i\varphi}) := \left(\cos^{s+1} \varphi (\rho-\sin\varphi) \log^2(\rho-\sin\varphi)\right)^{-1}, \\
1< \rho < 2 - \sin\varphi, \, \pi/4 < \varphi < \pi/2,
\end{multline*}
and $g(z)= 0$ for all other points in $\C_+$.  We do not give details in this paper.   Instead we will show in Propositions \ref{Furt} and \ref{prim} that $Q_s$ maps $\mathcal{V}_s$ boundedly into $\mathcal{D}_{\sigma}$ for any $\sigma>s$, and it maps holomorphic functions in $\mathcal{V}_s$ into $\D_s$.  
 We need the following auxiliary lemma which will be useful in a number of instances.

\begin{lemma}\label{D0012}
Let $h\in L^1[0,1] \cap  L^\infty[1/2,1]$ be a positive function.
Let $s>-1$, $\beta>1/2$, and
\begin{equation} \label{defG}
G_{h,\beta,s}(\varphi):= \int_{-\pi/2}^{\pi/2}\cos^s\psi \int_0^1\frac{h(t)\,dt}{(t^2+2t\cos(\varphi+\psi)+1)^\beta}\,d\psi,\quad \varphi\in (-\pi/2,\pi/2).
\end{equation}
\begin{enumerate}[\rm(a)]
\item If $2\beta-s-2\le0$, then
\begin{equation}\label{DerAn00}
K_{h,\beta,s} := \sup_{|\varphi|<\pi/2}\,G_{h,\beta,s}(\varphi)<\infty.
\end{equation}
\item If $2\beta-s-2\ge0$, then
\begin{equation}\label{DerAn01}
\tilde{K}_{h,\beta,s} := \sup_{|\varphi|<\pi/2}\, \cos^{2\beta-s-2}\varphi \, G_{h,\beta,s}(\varphi)<\infty.
\end{equation}
\end{enumerate}
\end{lemma}

\begin{proof}
Since $G_{h,\beta,s}(-\varphi)=G_{h,\beta,s}(\varphi)$, we may assume that $\varphi\in [0,\pi/2)$.
Now
\begin{align*}
G_{h,\beta,s}(\varphi) &= \int_{-\pi/2}^{0}\cos^s\psi \int_0^1\frac{h(t)\,dt}{(t^2+2t\cos(\varphi+\psi)+1)^\beta}\,d\psi\\
&\null \hskip30pt +
\int_0^{\pi/2}\cos^s\psi\int_0^1\frac{h(t)\,dt}{(t^2+2t\cos(\varphi+\psi)+1)^\beta}\,d\psi\\
&=: G_{h,\beta, s}^{-}(\varphi)+G_{h,\beta, s}^+(\varphi),
\end{align*}
and we estimate these two integrals separately.

Since $\varphi \in [0,\pi/2)$, $\beta>1/2$, and $s>-1$, we have
\begin{align*}\label{Eqa1}
G_{h,\beta, s}^{-}(\varphi) &\le \int_{-\pi/2}^{0}\cos^s\psi \,d\psi\int_0^1 h(t)\,dt
= \frac{\|h\|_{L^1[0,1]}}{2}B\left(\frac{s+1}{2},\frac{1}{2}\right).
\end{align*}

For the second integral, note that if $\phi = \varphi+\psi \in [0, \pi)$,
\[
t^2+2t\cos\phi+1\ge \begin{cases} \frac{1}{4},\quad &t\in [0,1/2], \\ 
(1-t)^2 + 1 + \cos\phi, \quad  &t\in [1/2,1]. \end{cases}
\]
Hence
\begin{align*}
\lefteqn{\hskip-30pt\int_0^1\frac{h(t)}{(t^2+2t\cos\phi+1)^\beta}\,dt } \\
&\le 4^\beta\int_0^{1/2} h(t)\,dt+
\int_{1/2}^1\frac{h(t)}{((t-1)^2+(1+\cos\phi))^\beta}\,dt\\
&\le  4^\beta \|h\|_{L^1[0,1/2]}+
\|h\|_{L^\infty[1/2,1]}\int_0^\infty\frac{d\tau}{(\tau^2+ 2\cos^2(\phi/2))^{\beta}} \\
&\le \frac{C_{h,\beta}}{\cos^{2\beta-1}(\phi/2)},
\end{align*}
for some constant $C_{h,\beta}>0$.  Replacing $\varphi$ by $\pi/2 -\varphi$ and $\psi$ by $\pi/2-\psi$, and using
\[
\omega\ge \sin \omega\ge \frac{2}{\pi}\omega,\qquad \omega\in (0,\pi/2),
\]
 we infer that if  $\varphi\in [0,\pi/2)$, then
\begin{align*}
G^+_{h,\beta,s}(\pi/2-\varphi) &\le C_{h,\beta} \int_0^{\pi/2}\frac{\sin^s\psi}{\sin^{2\beta-1}((\varphi+\psi)/2)} \,d\psi\\
&\le C_{h,\beta}\pi^{2\beta-1} \int_0^{\pi/2}\frac{\psi^s}{(\varphi+\psi)^{2\beta-1}} \,d\psi.
\end{align*}

In case (a), when $2\beta-s-2 \le 0$, we have
\[
\int_0^{\pi/2} \frac{\psi^s}{(\varphi+\psi)^{2\b-1}} \,d\psi \le \int_0^{\pi/2} \psi^{s-2\b+1} \,d\psi = \frac{\pi^{s-2\b+2}}{(s-2\b+2)2^{s-2\b+2}}<\infty.
\]
In case (b), when $2\beta - s -2 \ge 0$,
we obtain
\[
 \int_0^{\pi/2} \frac{\psi^s}{(\varphi+\psi)^{2\b-1}} \,d\psi \le \varphi^{s+2-2\b} \int_0^\infty \frac{t^s}{(t+1)^{2\b-1}} \,dt, 
\]
and hence
\begin{multline*}
\cos^{2\b-s-2}(\pi/2-\varphi) \, G_{h,\b,s}(\pi/2-\varphi) \le \varphi^{2\b-s-2} G_{h,\b,s}(\pi/2-\varphi)\\
\null \hskip-20pt \le C_{h,\beta} \pi^{2\beta-1} \int_0^\infty \frac{t^s}{(t+1)^{2\b-1}} \,dt < \infty,
\end{multline*}
for some constant $C_{h,\beta}>0$.
\end{proof}

Let $\D_{s,0}$ be the space of all functions $f \in \D_s$ such that $f$ has a sectorial limit $0$ at infinity.  
Then $(\D_{s,0}, \|\cdot\|_{\D_{s,0}})$ is a normed space, and we will see in Corollary \ref{Ban_D} that it is a Banach space.  

The following basic examples will play roles in several estimates later in the paper.  We start with the resolvent functions and their powers.

\begin{exa} \label{resd}
Let $\lambda=|\l|e^{i\varphi}\in \C_{+}$, and $r_\l(z)= (z+\l)^{-1}, z \in \C_+$.   Let $\gamma>0$, and consider
\[
r_{\l}^{\gamma}(z):=(z+\lambda)^{-\gamma},\quad z \in \C_+.
\]
Let $s>-1$.  Then
\begin{align}\label{rglest}
\|r_{\l}^{\gamma}\|_{\mathcal{D}_{s,0}}
&=\gamma \int_{-\pi/2}^{\pi/2}
\cos^s\psi\int_0^\infty \frac{d\rho}{\left|\rho e^{i\psi}+ |\l| e^{i\varphi}\right|^{\gamma+1}}\,d\psi\\
&=\frac{\gamma}{|\l|^{\gamma}}\int_{-\pi/2}^{\pi/2}
\cos^s\psi\int_0^\infty \frac{d\rho}{|\rho+e^{i(\varphi-\psi)}|^{\gamma+1}}\,d\psi  \nonumber \\
&=\frac{\gamma}{|\l|^{\gamma}}
\int_{-\pi/2}^{\pi/2}
\cos^s\psi\int_0^1 \frac{1+t^{\gamma-1}}
{(t^2+2t\cos(\varphi+\psi)+1)^{(\gamma+1)/2}}\,dt\,d\psi, \nonumber
\end{align}
where we have put $t=\rho$ for $\rho\le 1$ and $t = \rho^{-1}$  for $\rho>1$.
Now we apply Lemma \ref{D0012} with $h(t) = 1+t^{\gamma-1}$, $\beta = (\gamma+1)/2$, so $2\beta - s -2 = \gamma -s -1$.   Thus we obtain
\begin{equation} \label{rtl}
\|r_{\l}^{\gamma}\|_{\mathcal{D}_{s,0}}   \le \begin{cases} \dfrac{\gamma K_{h,(\gamma+1)/2,s}}{|\l|^\gamma}, &s>\gamma-1>-1, \\ \noalign{\vskip8pt}
\dfrac{\gamma \tilde{K}_{h,(\gamma+1)/2,s}}{|\l|^{\gamma}\cos^{\gamma-s-1}\varphi},  &\gamma-1>             s>-1.
 \end{cases}
\end{equation}
In particular, taking $\gamma=1$ and a fixed $s>0$, 
\begin{equation} \label{r_l}
\|r_\l\|_{\D_{s,0}} = \int_{-\pi/2}^{\pi/2} \int_0^\infty \frac{\cos^s \psi}{|\l+\rho e^{i\psi}|^2} \,d\rho \,d\psi
\le \frac{C_s}{|\l|}, \qquad\l \in \C_+.
\end{equation}
This estimate will play a crucial role in the proof of Theorem \ref{dcalculus_intro}.
\end{exa}

Next we consider some functions which appear frequently in the studies of holomorphic $C_0$-semigroups.   

\begin{exa}\label{dm} 
Let
\[
f_\nu(z):=z^{\nu} e^{-z},\quad z\in \C_{+},\quad \nu\ge0.
\]
We will show here that $f_\nu\in \mathcal{D}_s$ if and only if $s>\nu$.
Moreover, if $s >\nu$, then
\begin{equation}\label{frac}
\|f_\nu\|_{\D_{s,0}}\le 2 B\left(\frac{s-\nu}{2},\frac{1}{2}\right)\Gamma(\nu+1).
\end{equation}
This estimate will be crucial for operator estimates in Section \ref{norm_estimates}.

We have
\[
f_\nu'(z)=e^{-z}\left(\nu{z^{\nu-1}}-z^\nu\right)
\]
and
\begin{align*}
\|f_\nu\|_{\mathcal{D}_{s,0}}&=\int_{-\pi/2}^{\pi/2}
\cos^s\varphi\int_0^\infty {e^{-\rho\cos\varphi}}{\rho^{\nu-1}}\left|{\rho e^{i\varphi}}-\nu\right|\,d\rho\,d\varphi\\
&=\int_{-\pi/2}^{\pi/2}
\cos^{s-\nu-1}\varphi\int_0^\infty e^{-r}r^{\nu-1}|
r e^{-i\varphi}-\nu\cos\varphi|\,dr\,d\varphi.
\end{align*}
We use the estimates
\[
r |\sin\varphi| \le |r e^{-i\varphi}-\nu\cos\varphi| \le  r+\nu.
\]
If $s > \nu$, we obtain
\begin{align*}
\|f_\nu\|_{\D_{s,0}} &\le
2\int_0^{\pi/2}
\cos^{s-\nu-1}\varphi\,d\varphi \int_0^\infty e^{-r} r^{\nu-1}(r+\nu) \,dr \\
&= 2 B\left(\frac{s-\nu}{2},\frac{1}{2}\right)\Gamma(\nu+1). \notag
\end{align*}
If $s \le \nu$, then
\[
\|f_\nu\|_{\D_{s,0}} \ge
2\int_0^{\pi/2}
\cos^{s-\nu-1}\varphi \, \sin\varphi\,d\varphi \int_0^\infty e^{-\tau} \tau^{\nu} \,d\tau = \infty.
\]
This establishes the claims above.
\end{exa}

Finally we consider a function which will play an important role in our constructions of functional calculi in Section \ref{hardy_sobolev}.

\begin{exa}  \label{arc}
The function $\arccot$ is defined by
\begin{equation} \label{arcdef}
\arccot(z) = \frac{1}{2i} \log \left(\frac{z+i}{z-i}\right), \quad z \in \C_+.
\end{equation} 
Then $|\Re \arccot(z)| \le \pi/2$, $\arccot$ has sectorial limit $0$ at infinity, and
its derivative is $- (z^2+1)^{-1}$. 
It is easy to see that $\arccot \in \D_s$ for all $s>-1$.  For $s = 0$, we have
\begin{align} \label{arccot_D}
&\\
\lefteqn{\|\arccot\|_{\mathcal{D}_{0,0}}} \notag\\
&=\int_{-\pi/2}^{\pi/2}\int_0^\infty \frac{d\rho}{|1+\rho^2e^{2i\varphi}|}\,d\varphi
=\int_0^\infty \int_0^{\pi}\frac{d\psi}{(\rho^4+2\rho^2\cos\psi+1)^{1/2}}\,d\rho \notag\\
&\le \sqrt{\pi}\int_0^\infty \left(\int_0^{\pi} \frac{d\psi}{\rho^4+2\rho^2\cos\psi+1}\right)^{1/2}\,d\rho
= \sqrt{\pi}\int_0^\infty \left(\frac{\pi}{|\rho^4-1|}\right)^{1/2}\,d\rho\notag\\
&=\frac{\pi}{2} B(1/4,1/2)
 < 3 \pi.\notag
\end{align}
See \cite[2.5.16, (38)]{Prud} for the evaluation of the integral with respect to $\psi$.
\end{exa}

\begin{prop}\label{Furt}
Let $\sigma > s > -1$.  The following hold.
\begin{enumerate}[\rm(i)]
\item $\D_s$ is continuously embedded in $\D_\sigma$.
\item The restriction of $Q_\sigma$ to $\mathcal{V}_s$ is in $L(\mathcal{V}_s,\D_s)$.
\item $Q_s \in L (\mathcal{V}_s, \mathcal{D}_{\sigma})$.
\end{enumerate} 
\end{prop}

\begin{proof}
The first statement is immediate from the definitions of the spaces and \eqref{vst}.  

For the second statement, let $g \in \mathcal{V}_s$.   
From \eqref{norm}, \eqref{qder}, and the second case of \eqref{rtl} with $\gamma=\sigma+1$ and $h(t)=1+t^\sigma$,  we have
\begin{align*}
\lefteqn{\frac{\pi}{2^\sigma(\sigma+1)} \|Q_{\sigma}g\|_{\mathcal{D}_{s,0}}\hskip10pt}\\
&\le \int_{-\pi/2}^{\pi/2}\cos^{s}\psi \int_0^\infty
 \int_{-\pi/2}^{\pi/2} \cos^\sigma \varphi \int_0^\infty \frac{t^{\sigma+1}|g(t e^{i\varphi})|}{|\rho e^{i\psi}+t e^{-i\varphi}|^{\sigma+2}}\,d\rho\,d\varphi \,dt\,d\psi\\
&= \frac{1}{\sigma+1}\int_{-\pi/2}^{\pi/2} \cos^\sigma \varphi \int_0^\infty t^{\sigma+1}\|r_{t e^{-i\varphi}}^{\sigma+1}\|_{\D_{s,0}} |g(t e^{i\varphi})|\,dt\,d\varphi \\
&\le  \int_{-\pi/2}^{\pi/2} \cos^\sigma \varphi \int_0^\infty \frac{\tilde{K}_{h,(\sigma+2)/2,s}}{\cos^{\sigma-s}\varphi} |g(t e^{i\varphi})|\,dt\,d\varphi \\
&= \tilde{K}_{h,(\sigma+2)/2,s} \|g\|_{\mathcal{V}_s}.
\end{align*}
This establishes the second statement.

For the third statement, the same estimation, but with $s$ and $\sigma$ interchanged, and using the first case of 
\eqref{rtl} with $\gamma=s+1$ and $h(t)=1+t^s$, shows that
\[
\frac{\pi}{2^s(s+1)} \|Q_{s}g\|_{\mathcal{D}_{\sigma,0}}\le {K}_{h,(s+2)/2,\sigma} \|g\|_{\mathcal{V}_s}.
\]
This establishes the third statement. 
\end{proof}

For insight on why $\|r_{\l}^{\sigma+1}\|_{\D_{s,0}}$ appears in the proof above, we refer the reader to the proof of Theorem \ref{D00}.

\begin{prop} \label{prim}
Let $g \in \mathcal{V}_s$ be holomorphic, where $s>-1$.  Then $Q_s g \in \mathcal{D}_s$ and $(Q_s g)'= g$.
\end{prop}

\begin{proof}
First  we consider the case when $g\in \mathcal{V}_n$, where $n \in \N\cup\{0\}$.  Then $Q_n g$ is holomorphic and \eqref{qdef} holds for $s=n$.  It suffices to show  that $(Q_ng)' = g$.    Let
\begin{align*}
I(\a) &:= \int_{-\infty}^\infty \frac{|g(\a+i\b)|}{(\a^2+\b^2)^{(n+1)/2}} \,d\b, \qquad \alpha >0,\\
J(\b) &:= \int_0^\infty \frac{\a^n|g(\a+i\b)|}{(\a^2+\b^2)^{(n+1)/2}} \,d\a, \qquad \beta \in \R.
\end{align*}
From \eqref{norm} and Fubini's theorem, we see that $\int_0^\infty \a^n I(\a) \, d\a < \infty$.  Hence $I(\a)<\infty$ for almost all $\a>0$ and it follows that there exists a sequence  $(\alpha_j)_{j\ge 1}$ such that
\[
\lim_{j \to \infty}\alpha_j =\infty, \qquad \lim_{j\to\infty} I(\a_j) = 0.
\]
Similarly $\int_{-\infty}^\infty J(\b)\,d\b < \infty$, and so there exist sequences $(\beta_k^{\pm})_{k \ge 1}$ such that
\[
\lim_{k \to \infty}\beta_k^\pm = \pm\infty, \qquad \lim_{k\to\infty} J(\beta_k^\pm) = 0.
\]

Let $z \in \C_+$ be fixed.    Take $\alpha>0$ with $I(\a)<\infty$, and let $j$ be sufficiently large that $\alpha_j > 2\a + \Re z$, and $k$ be sufficiently large that $\b_k^- < \Im z < \b_k^+$.    We may apply the Cauchy integral formula around the rectangle with vertices $\a +i\beta_k^\pm$ and $\a_j + i \beta_k^{\pm}$, and we obtain
\begin{align*}
\lefteqn{\frac{2\pi}{n!} g^{(n)}(2\a+z)} \\
&= \int_{\b_k^-}^{\b_k^+} \frac{g(\a_j+i\b)}{(\a_j+i\b - z -2\a)^{n+1}} d\b
- \int_{\b_k^-}^{\b_k^+} \frac{g(\a+i\b)}{(\a+i\b - z -2\a)^{n+1}} d\b \\
&\null\hskip10pt -i \int_{\alpha}^{\alpha_j}\frac{g(s+i\beta_k^{-})}{(s+i\beta_k^{-}- z-2\a)^{n+1}}\,ds
+ i \int_{\alpha}^{\alpha_j}\frac{g(s+i\beta_k^{+})}{(s+i\beta_k^{+} - z -2\a)^{n+1}}\,ds.
\end{align*}
Letting $k \to \infty$, we obtain
\[
(-1)^n \frac{2\pi}{n!} g^{(n)}(2\a+z) = \int_{-\infty}^\infty \frac{g(\a+i\b)}{(z+\a-i\b)^{n+1}} \,d\b
- \int_{-\infty}^\infty \frac{g(\a_j+i\b)}{(z+\a_j-i\b)^{n+1}} \,d\b.
\]
Letting $j \to \infty$, we obtain
\begin{equation} \label{Qn}
(-1)^n \frac{2\pi}{n!} g^{(n)}(2\a+z) = \int_{-\infty}^\infty \frac{g(\a+i\b)}{(z+\a-i\b)^{n+1}} \,d\b.
\end{equation}
This holds for almost all $\a>0$.    Substituting this into \eqref{qdef}, and then integrating by parts and using Lemma \ref{Derin}, we infer that
\begin{align} \label{Qnn}
(Q_n g)(z) &= (-1)^{n+1} \frac{2^{n+1}}{n!} \int_0^\infty \alpha^{n} g^{(n)}(2\alpha+z)\,d\alpha \\
&=\frac{(-1)^{n+1}}{n!}\int_0^\infty \alpha^n g^{(n)}(\alpha+z)\,d\alpha = - \int_0^\infty g(\alpha+z)\,d\alpha. \nonumber
\end{align}
By Lemma \ref{Derin}, the integral
$\int_{0}^{\infty} g'(\alpha +z)\,d\alpha$ converges absolutely and uniformly for $z$ in compact subsets of $\C_+$.
So, differentiating under the integral sign we get  $(Q_n g)' = g$.

Now consider the case when $s=n+\delta>-1$, where $n \in \N\cup\{-1,0\}$, $\delta \in (0,1)$, and $g \in \mathcal{V}_n$ is holomorphic.   Then $g\in \mathcal{V}_{n+1}$, $n+1\in \N\cup\{0\}$, and \eqref{Qn} gives
\[
 \int_{-\infty}^\infty \frac{g(\a+i\b)}{(z+\a-i\b+t)^{n+2}} \,d\b=
(-1)^{n+1}\frac{2\pi}{(n+1)!} g^{(n+1)}(2\a+z+t),
\]
for $z \in \C_+$, $t>0$ and almost all $\a>0$.  We obtain from (\ref{46RRR}) and \eqref{Const46} that
\begin{align*}
C_s (Q_s g)(z)
&=- \frac{2^s}{\pi}\int_0^\infty \int_0^\infty\int_{-\infty}^\infty
\frac{\alpha^s\,g(\alpha+i\beta)\,d\beta\,d\alpha}{(z+\alpha-i\beta+t)^{n+2}}\,\frac{dt}{t^\delta}\\
&=(-1)^n\frac{2^{s+1}}{(n+1)!}\int_0^\infty \int_0^\infty \alpha^s
 g^{(n+1)}(2\a+z+t)\, d\alpha\,\frac{dt}{t^\delta}\\
&=\frac{(-1)^n}{(n+1)!}\int_0^\infty \int_t^\infty (\tau-t)^s
 g^{(n+1)}(\tau+z)
 \, d\tau\,\frac{dt}{t^\delta}\\
&=C_s\frac{(-1)^n}{(n+1)!}\int_0^\infty \tau^{n+1} \,g^{(n+1)}(\tau+z)\,d\tau.
\end{align*}
As in \eqref{Qnn}, it follows that $(Q_sg)' = g$.
\end{proof}

\begin{cor}
If $g \in \mathcal{V}_s$ is holomorphic, then $Q_\sigma g = Q_s g$ for all $\sigma\ge s$.
\end{cor}

\begin{proof}
This is immediate from Proposition \ref{prim} and \eqref{qinf}.
\end{proof}
\begin{remark}
The proof of the property $(Q_ng)'=g', n \ge 0,$ in Proposition 3.7 uses just improper convergence of the integrals $\int_{0}^{\infty} \alpha^{k} g^{(k)}(\alpha+z)\,d\alpha$ for $0 \le  k \le n-1.$
It is instructive to note that  if $g \in \mathcal V_n$ is holomorphic and $z \in \mathbb C_+$, then 
\begin{equation}\label{int_n} 
\int_{0}^{\infty} \alpha^{k} |g^{(k)}(\alpha+z)|\,d\alpha<\infty
\end{equation}
 for all $ k \ge 0.$
Indeed  if $g \in \mathcal V_n,$  then \eqref{int_n} holds for $k=n$  by \eqref{Qn} and the definition of norm in $\mathcal V_n.$ If $n \ge 1,$ then using Lemma \ref{Derin},
we infer that $g^{(n-1)}(\alpha+z)=-\int_\alpha^\infty g^{(n)}(s+z)\,ds,$ and hence 
by Fubini's theorem
\begin{equation*}
\int_0^\infty
\alpha^{n-1} |g^{(n-1)}(\alpha+z)|\,d\alpha
\le \frac{1}{n} \int_0^\infty \alpha^n | g^{(n)}(\alpha +z)| d\alpha <\infty.
\end{equation*}
Repeating this argument, we conclude that \eqref{int_n} holds also for  $k$ such that $0 \le k < n.$ 
If $k>n,$ then \eqref{int_n} follows directly from \eqref{Qn} and the inclusion  $\mathcal V_n \subset \mathcal V_k.$ 
\end{remark}

The following representation of functions in $\mathcal{D}_s$ has appeared in \cite[Corollary 4.2]{Al2014} (see also \cite[Lemma 3.13.2]{AlPel} for the case $s=1$).

\begin{cor}\label{Repr}
Let $f\in \mathcal{D}_s$, $s>-1 $. Then the sectorial limits 
\begin{align}
f(\infty): &=\lim_{z\to\infty ,\;z\in \Sigma_{\psi}}\,f(z),\label{finf}\\
f(0):&= \lim_{z\to 0,\;z\in \Sigma_{\psi}}\,f(z)\label{fzero},
\end{align}
exist in $\C$ for every $\psi \in (0,\pi/2)$.   Moreover, for all $z \in \C_{+}\cup \{0\}$,
\begin{align} \label{qnrep}
f(z)&=f(\infty)+(Q_s f')(z)\\
&=f(\infty)- \frac{2^s}{\pi}\int_0^\infty \alpha^s\int_{-\infty}^\infty
\frac{f'(\alpha+i\beta)\,d\beta}{(z+\alpha-i\beta)^{s+1}} \, d\alpha.\notag
\end{align}
With $f(0)$ defined as above, $f \in C(\overline{\Sigma}_\psi)$ for every $\psi\in (0,\pi/2)$.
\end{cor}

\begin{proof}
It follows from Proposition \ref{prim} that $(Q_s f')' = f'$.  The statements follow from \eqref{qinf} and \eqref{qzero}.
\end{proof}

\begin{cor}\label{Ban_D}
For every $s>-1$ the space $\mathcal D_s$ equipped with the norm
\[
\|f\|_{\mathcal D_s}:= |f(\infty)|+\|f\|_{\mathcal D_{s,0}}, \qquad f \in\mathcal D_s,
\]
is a Banach space.
\end{cor}

\begin{proof}
Let $s>-1$ be fixed and let $(f_k)_{k=1}^\infty$ be a Cauchy sequence in $\mathcal D_s$.
Then \eqref{angle} and Vitali's theorem imply that $(Q_s f_k')_{k=1}^\infty$ converges uniformly on each $\Sigma_\psi$ to a limit $g$ which is holomorphic on $\C_+$.  Moreover $(f_k(\infty))_{k=1}^{\infty}$ converges to a limit $\zeta \in\C$.   It follows from Proposition \ref{prim} that $(f_k)_{k=1}^\infty$ converges uniformly on $\Sigma_\psi$ to $f := \zeta + g$.  Then $(f'_k)_{k=1}^\infty$ converges pointwise on $\C_+$ to $g' = f'$.   Applying Fatou's lemma to the sequences $(\|f'_k-f'_n\|_{\mathcal{V}_s})_{n=k}^\infty$ for fixed $k$, one sees that $\|f_k' - f'\|_{\mathcal V_s} \to 0$.   So $f' \in \mathcal{V}_s$ and $f \in \D_s$.  By \eqref{qinf}, $f(\infty)=\zeta$ and so $\|f_k - f\|_{\mathcal{D}_s} \to 0$, $k\to\infty$.
\end{proof}

The argument used in the proof of Corollary \ref{Ban_D} also provides the following lemma of Fatou type (see also Lemma \ref{FatouH}).

\begin{cor}\label{Fatou}
Let $s>-1$ and $(f_k)_{k=1}^\infty \subset \mathcal D_s$ be such that $\sup_{k \ge 1} \|f_k\|_{\mathcal D_s}<\infty$ and $f(z)=\lim_{k \to \infty} f_k(z)$ exists for all $z \in \mathbb C_+$. Then $f \in \mathcal D_s$.
\end{cor}

Now employing \eqref{qnrep}, the estimates (\ref{angle}), (\ref{Dangle}) and  (\ref{Lebeg11}), and Lemma \ref{Derin}, we obtain the following estimates.

\begin{cor}\label{Cangle}
Let $\psi \in (0,\pi/2)$.   For all $f\in \mathcal{D}_s$, $s>-1$,
\[
|f(z)|\le \max\left(1,\frac{2^s}{\pi\cos^{s+1}\psi}\right)\|f\|_{\mathcal{D}_s},\qquad z\in \Sigma_\psi,
\]
and
\[
|f'(z)|\le \frac{(s+1)2^s}{\pi |z| \cos^{s+2}\psi}\|f'\|_{\mathcal{V}_s},\qquad z\in\Sigma_\psi.
\]
Moreover, $f$ is continuous on $\Sigma_\psi \cup \{0\}$ and
\[
\lim_{|z|\to0, z \in \Sigma_\psi} z f'(z) = 0.
\]
\end{cor}

\begin{rem} \label{remdr}
Corollary \ref{Cangle} implies that the point evaluation functionals $\delta_z, \, z \in \C_+$, are continuous on $\D_s, \, s>-1$.  Using \eqref{rtl} and the principle set out in Section \ref{prelims}, we see that the function $\l \mapsto r_{\l}^{\gamma}$ is holomorphic from $\C_+$ to $\D_s$, for any $s>-1$, $\gamma>0$.
\end{rem}

\subsection{More functions in $\D_s$ and their properties}\label{some_functions}

In this section we give more examples of functions from $\D_s$ and note some additional 
elementary properties which will be relevant for the sequel.

\begin{prop} \label{BDs}
For $s > 0$, $\Bes \embedi \D_s$.
\end{prop}

\begin{proof}
For $s>0$ and $f\in \Bes$ we have
 \begin{align*}
 \|f\|_{\mathcal{D}_{s,0}}&=|f(\infty)|
 +\int_0^\infty \alpha^s\int_{-\infty}^\infty\frac{|f'(\alpha+i\beta)|}
 {|\alpha+i\beta|^{s+1}} \,d\beta\,d\alpha \label{inclus_b}\\
 &\le \|f\|_{\infty}
 +2\int_0^\infty  \sup_{\beta\in \R}\,|f'(\alpha+i\beta)|\,
 \int_{0}^\infty\frac{dt}
 {(t^2+1)^{(s+1)/2}}\,d\alpha\notag\\
&\le  \max\{1, B(1/2,s/2)\}\|f\|_\Bes.\notag
\end{align*}
Thus $\Bes$ is continuously included in $\D_s$.  
\end{proof}

\begin{rem} \label{Brep}
The representation \eqref{qnrep} in Corollary \ref{Repr} extends the reproducing formula for $\Bes$, i.e., \eqref{qnrep} for $s=1$, to a larger class of functions.
\end{rem}
Recall that  $\LT \embedi \Bes$, see \cite[Section 2.4]{BGT}, Thus, in view of Proposition \ref{BDs}, we have
\[
 \LT \embedi \Bes \embedi \D_s, \qquad s>0.
\]
   We will show in Corollary \ref{Bdens} that $\mathcal B$ is dense in $\mathcal D_s$ for every $s>0$, and hence $\mathcal D_s$ is dense in $\mathcal D_\sigma$ for 
all $\sigma > s > 0$.   On the other hand, we will show in Corollary \ref{Bdens} that $\Bes$ is not dense in $\D^\infty_s$ for $s>0$.

For $f \in \Hol(\C_+)$, let
\[
\tilde{f}(z):=f(1/z), \qquad f_t(z): = f(tz), \qquad t >0, \quad z\in \C_{+}.
\]

\begin{lemma} \label{leminv}
Let $s>-1$ and $t>0$.  Then
\begin{enumerate}[\rm(i)]
\item 
$f\in \mathcal{D}_s$ if and only if $\tilde{f}\in \mathcal{D}_s$
and
\[
\|f-f(\infty)\|_{\mathcal{D}_s}=\|\tilde{f}- f(0)\|_{\mathcal{D}_s}.
\]
\item If $f$ is bounded away from $0$ and $f \in \D_s$, then $1/f \in \D_s$.
\item
 $f \in \D_s$ if and only if $f_t \in \D_s$, and $\|f\|_{\D_s} = \|f_t\|_{\D_s}$.
\end{enumerate}
\end{lemma}

\begin{proof}
Note that
\begin{align*}
\|\tilde{f}'\|_{\mathcal{V}_s}&=\int_{-\pi/2}^{\pi/2}\cos^s\varphi\int_0^\infty \frac{|f'(\rho^{-1}e^{-i\varphi})|}{\rho^2}\,d\rho\,d\phi\\
&=\int_{-\pi/2}^{\pi/2}\cos^s\varphi\int_0^\infty |f'(r e^{i\varphi})|\,dr\,d\varphi\\
&=\|f'\|_{\mathcal{V}_s}.
\end{align*}
Moreover, by Corollary \ref{Repr}, $\tilde{f}(\infty)=f(0)$.  
This proves (i).  The other statements are very easy.
\end{proof}

\begin{rems}  \label{e-z}
1.  Neither of the spaces $\Bes$ and $\mathcal{D}_0$ is contained in the other.  Indeed, the function $e^{-z} \in \LT \subset \Bes$ but $e^{-z} \not \in \mathcal D_0$ (see Example \ref{dm}).   On the other hand, there are bounded functions in $\D_0$ which are not in $\Bes$, for example the  function $\exp(\arccot z) \in \mathcal D_0$ and is bounded but it is not in $\Bes$ (see Example \ref{eac}).

More generally, for $\nu\ge0$, let $f_\nu(z) = z^\nu e^{-z}, \, z\in \C_+$.   Then $f_\nu \in \D_s$ if and only if $s>\nu$ (see Example \ref{dm}).   Note that, if $\nu>0$, $f_\nu$ is not bounded on any right half-plane.   One can show that if $f \in \D_{s_0}$, then
\[
|f(z)| \le \frac{2^s}{\pi} \|f\|_{\D_s,0}\left(1+\frac{4\b^2}{\a^2}\right)^{(s+1)/2}, \quad z =\a+i\b \in \C_+.
\]
The function $\log(1+z) e^{-z}$ is in $\D_s, s>0$, but is also unbounded on every right half-plane.

\noindent
2.  Since $e^{-z} \in \D_s$ for $s>0$, it follows from Lemma \ref{leminv} that the functions $e^{-t/z}$ are in $\mathcal D_s$ for all $t>0, \, s >0$.   This shows that functions $f \in \mathcal D_s$ may not have full limits at infinity or at zero.   However, the properties \eqref{finf} and \eqref{fzero} in Corollary \ref{Repr} establish values for $f$ at infinity and at zero as sectorial limits.

\noindent 3.  The spaces $\D_s, \, s>-1$, are invariant under shifts given by
\[
(T(\tau)f)(z) = f(z+\tau), \qquad f \in \D_s, \, \tau \in \C_+, \, z \in \C_+.
\]
Indeed these operators form a bounded $C_0$-semigroup on $\D_s$.   See Section \ref{shifts} for a proof.   On the other hand, $\D_s$ are not invariant under the vertical shifts when $\tau \in i\R$, as we see in the following example. 
\end{rems}

\begin{exa} \label{eac}
As stated in Example \ref{arc}, the function $\arccot$ is in $\D_s$ for all $s>-1$.   
Let
\[
g(z)=\exp(\arccot(z)),  \qquad \quad z\in \C_{+},
\]
Since $|\Re \arccot(z)| \le \pi/2$, $\|g\|_\infty = \exp(\pi/2) = g(\infty)$.   For $s>-1$, 
we have
\begin{align*}
\|g\|_{\mathcal{D}_s}&= |g(\infty)|+\|g \cdot (\arccot)'\|_{\mathcal{V}_s}
\le \exp(\pi/2)(1+\|\arccot\|_{\D_s}).
\end{align*}
Thus $g \in \D_s$ for all $s>-1$.   

However the boundary function of $g$ is not continuous at $z=\pm i$.
Indeed, for a fixed $\ep>0$,
\[
\arccot (i+i\ep)=\frac{1}{2i}\log\left(1+\frac{2}{\ep}\right),
\]
hence
\[
g(i+i\ep)=\exp{(-i\log(1+2/\ep)^{1/2})}
\]
does not have a limit as $\ep\to 0+$.

Note that $g(\ep + i)$ does not converge as $\ep\to0+$.   This means
that if $f(z):=g(z-i)$, then $f$ does not have a sectorial limit at 0 and therefore does not belong to $\D_s$ for any $s>-1$.  Thus $\D_s$ is not invariant under vertical shifts.
\end{exa}

\subsection{Bernstein functions and $\D_s$}
Recall that a holomorphic function $g :\C_+ \to \C_+$ is a \emph{Bernstein function} if it is of the form
\begin{equation} \label{bsfn}
g(z) = a + bz + \int_{(0,\infty)} (1-e^{-zs}) \, d\mu(s),
\end{equation}
where $a\ge 0$, $b\ge0$ and $\mu$ is a positive Borel measure on $(0,\infty)$ such that
$\int_{(0,\infty)} \frac{s}{1+s} \,d\mu(s) < \infty$.
The following properties of Bernstein functions $g$ will be used (these properties differ slightly from those used in \cite{BGTad}):
\begin{enumerate}
\item[(B1)] $g$ maps $\Sigma_\psi$ into $\Sigma_\psi$ for each $\psi \in [0,\pi/2]$, see \cite[Proposition 3.6]{Schill} or \cite[Proposition 2.1(1)]{BGTad}.
\item[(B2)] $g$ is increasing on $(0,\infty)$.
\item[(B3)]  For all $z \in \C_+$,
\[
g(\Re z) \le \Re g(z) \le |g(z)|, \qquad |g'(z)| \le  g'(\Re z).
\]
Here the first inequality follows from taking the real parts in \eqref{bsfn}, and the second inequality is shown in \cite[Section 3.5, (B3)]{BGT}.
\end{enumerate}
Further information on Bernstein functions can be found in  \cite{Schill}.

\begin{lemma} \label{bsds}
Let $g$ be a Bernstein function, $\lambda\in\C_+$, and
\[
f(z;\lambda): =(\l+g(z))^{-1},\qquad z\in \C_{+}.
\]
Then $f(\cdot;\l) \in \D_s$ for $s>2$ and 
\[
\|f(\cdot;\l)\|_{\D_{s,0}} \le \frac{2^s}{(s-2)|\l|} .
\]
\end{lemma}

\begin{proof}
We have
\[
f'(z;\lambda)=-\frac{g'(z)}{(\l+g(z))^2},
\]
and then, for $\psi\in (-\pi/2,\pi/2)$, using Lemma \ref{trig}, (B2) and (B3),
\begin{align*}
\int_0^\infty |f'(\rho e^{i\psi};\lambda)|\,d\rho
&=\int_0^\infty \frac{|g'(\rho e^{i\psi})|}
{|\l+g(\rho e^{i\psi})|^2} \,d\rho\\
&\le \frac{1}{\cos^2((|\psi|+\pi/2)/2)}
\int_0^\infty \frac{g'(\rho\cos\psi)}
{(|\l|+|g(\rho e^{i\psi})|)^2} \,d\rho\\
&\le\frac{1}{\cos^2((|\psi|+\pi/2)/2)}
\int_0^\infty \frac{g'(\rho\cos\psi)}
{(|\l|+g(\rho\cos\psi))^2} \,d\rho\\
&\le \frac{1}{\cos^2((|\psi|+\pi/2)/2)\cos\psi}
\int_0^\infty \frac{dt}
{(|\lambda|+t)^2}\\
&=
\frac{1}{\cos^2((|\psi|+\pi/2)/2)\cos\psi}\cdot \frac{1}{|\lambda|}.
\end{align*}

If $s>2$, then
\begin{align*}
|\lambda|\|f(\cdot;\lambda)\|_{\mathcal{D}_{s,0}}
&=|\l|\int_{-\pi/2}^{\pi/2}\cos^s\psi
\int_0^\infty |f'(\rho e^{i\psi};\lambda)|\,d\rho\,d\psi\\
&\le 2\int_0^{\pi/2}
\frac{\cos^{s-1}\psi}{\cos^2((\psi+\pi/2)/2)}\,d\psi
=
4\int_0^{\pi/2}
\frac{\cos^{s-1}\psi}{1-\sin\psi} \,d\psi\\
&=4\int_0^1 (1+t)^{s-2}(1-t)^{s-3}\,dt < \frac{2^s}{s-2}.   \qedhere
\end{align*}
\end{proof}

\subsection{Algebras associated with $\mathcal D_s$}
The spaces $\mathcal{D}_s, \, s>-1$, are not algebras, but there are some related algebras.   Consider the Banach spaces $\mathcal{D}_s^\infty:= \mathcal{D}_s\cap H^\infty (\C_{+})$ equipped with the norm
\[
\|f\|_{\mathcal{D}_s^\infty}:=\|f\|_\infty+\|f'\|_{\mathcal{V}_s}.
\]
Thus $\mathcal{D}_s^\infty$ is the space of bounded holomorphic functions on $\C_{+}$ such that
\[
f(\infty):=\lim_{|z|\to\infty, z\in \Sigma_\psi}\,f(z)
\]
exists for every $\psi\in (0,\pi/2)$, and
\[
\|f\|_{\mathcal{D}_s^\infty}:=\|f\|_{\infty}+\int_0^\infty \alpha^s \int_{-\infty}^\infty \frac{|f'(\alpha+i\beta)|}{(\alpha^2+\beta^2)^{(s+1)/2}}\,d\alpha\,d\beta  < \infty.
\]
Then $(\mathcal{D}_s^\infty, \|\cdot\|_{\D_s})$ is a Banach algebra, and in particular
\begin{equation}\label{algebra}
\|fg\|_{\mathcal{D}_s^\infty}\le \|f\|_{\mathcal{D}_s^\infty} \|g\|_{\mathcal{D}_s^\infty},\quad f,g\in \mathcal{D}_s^\infty.
\end{equation}
By Proposition \ref{BDs}, $\Bes\embedi \mathcal{D}_s^\infty$ for $s>0$, and the embeddings are continuous.

Example \ref{eac} shows that the function $g(z) := \exp(\arccot z)$ is in $\mathcal{D}_s^\infty$ for all $s>-1$, with $\|g\|_{\mathcal{D}_s^\infty} = \|g\|_{\mathcal{D}_s}$, and  consequently $\mathcal{D}_s^\infty$ is not invariant under vertical shifts.

It follows from Lemma \ref{leminv} that
\[
f\in \mathcal{D}_s^\infty \quad \text{if and only if}\quad
\tilde{f}(z):=f(1/z)\in \mathcal{D}_s^\infty,
\]
and
\[
\|f\|_{\mathcal{D}_s^\infty}=\|\tilde{f}\|_{\mathcal{D}_s^\infty}.
\]
Moreover, the spectrum of $f$ in $\mathcal{D}_s^\infty$ is the closure of the range of $f$.  In particular, the spectral radius of $f$ is $\|f\|_\infty$.

Now we consider the linear space
\[
\mathcal D_\infty:=\bigcup_{s>-1} \mathcal{D}_s.
\]
We will show that $\mathcal D_\infty$ is an algebra, which opens the way to an operator functional calculus on $\mathcal D_\infty$.

\begin{lemma}\label{Alg1} For $s,\sigma>-1$, let $f\in \mathcal{D}_s$ and $g\in \mathcal{D}_\sigma$.
Then
\[
h:=fg \in \mathcal{D}_{s+\sigma+1},
\]
and
\begin{equation} \label{fgest}
\|h\|_{\mathcal{D}_{s+\sigma+1}} \le
\left(2+\frac{2^s+2^\sigma}{\pi}\right)\|f\|_{\mathcal{D}_s}\,\|g\|_{\mathcal{D}_\sigma}.
\end{equation}
Hence,
$\mathcal{D}_\infty$ is an algebra.
\end{lemma}

\begin{proof}
By Corollary \ref{Cangle}, for $\rho>0$ and $|\varphi|< \pi/2$, we have
\begin{align*}
|f(\rho e^{i\varphi})| &\le \left(1+\frac{2^s}{\pi\cos^{s+1}\varphi}\right)\|f\|_{\mathcal{D}_s},\\
|g(\rho e^{i\varphi})| &\le \left(1+\frac{2^\sigma}{\pi\cos^{\sigma+1}\varphi}\right)\|g\|_{\mathcal{D}_\sigma}.
\end{align*}
Hence
\begin{align*}
\|h\|_{\mathcal{D}_{s+\sigma+1,0}} &= \int_{-\pi/2}^{\pi/2} \cos^{s+\sigma+1}\varphi \int_0^\infty \left| f'(\rho e^{i\varphi}) g(\rho e^{i\varphi}) + f(\rho e^{i\varphi}) g'(\rho e^{i\varphi})\right| \,d\rho\,d\varphi \\
&\le \left(1+\frac{2^\sigma}{\pi}\right)\|f\|_{\mathcal{D}_{s,0}}\,\|g\|_{\mathcal{D}_\sigma}
+\left(1+\frac{2^s}{\pi}\right)\|f\|_{\mathcal{D}_s}\,\|g\|_{\mathcal{D}_{\sigma,0}}.
\end{align*}
This shows that $h \in \mathcal{D}_q$, and \eqref{fgest} follows easily.
\end{proof}

\subsection{Derivatives of functions in $\D_s$}

This section further clarifies the behaviour of the derivatives of functions from $\mathcal D_s$, and Lemma \ref{D001} is of independent interest.  Corollary \ref{D00n} will be used in Section \ref{AnalF}.   For $m,n \in \N$, the notation $z^mf^{(n)}$ denotes the function mapping $z$ to $z^mf^{(n)}(z)$.  Moreover, $f_t$ is the function mapping $z$ to $f(tz)$ .

\begin{lemma}\label{D001}
Let $f\in \mathcal{D}_s$, $s>-1$. Then
$zf'\in \mathcal{D}_{s+1}$,   and there exists $C'_{s}$ (independent of $f$) such that
\begin{equation}\label{DerAnD}
\|zf'\|_{\mathcal{D}_{s+1}}\le  {C}'_{s}\|f\|_{\mathcal{D}_s}.
\end{equation}
and
\begin{equation} \label{DerAnE}
\|f_t-f_\tau\|_{\D_{s+1}} \le \frac{C'_s \|f\|_{\D_s}}{\min\{t,\tau\}} |t-\tau|, \qquad  t,\tau>0.
\end{equation}
\end{lemma}
 
\begin{proof}
Note that 
\[
\|zf'\|_{\mathcal{D}_{s+1}}\le \|f'\|_{\mathcal{V}_{s+1}}+\|zf''\|_{\mathcal{V}_{s+1}}
\le \|f\|_{\mathcal{D}_{s}}+\|zf''\|_{\mathcal{V}_{s+1}}.
\]
So, for (\ref{DerAnD}), it suffices to consider $\|zf^{''}\|_{\mathcal V_{s+1}}$.
The argument is similar to Example \ref{resd} and the proof of Proposition $\ref{Furt}$.
By Corollary \ref{Repr},
 for fixed $\sigma>s$, 
\[
f''(z)=-c_\sigma
\int_{-\pi/2}^{\pi/2} \cos^\sigma \varphi \int_0^\infty \frac{\rho^{\sigma+1} f'(\rho e^{i\varphi})}{(z+\rho e^{-i\varphi})^{\sigma+3}}\,d\rho\,d\varphi,\quad z\in \C_{+},
\]
where $c_\sigma=(\sigma+1)(\sigma+2)\frac{2^\sigma}{\pi}$.
Then similar estimates to those in Example \ref{resd} and Proposition \ref{Furt} give
\begin{align*}
\lefteqn{c_\sigma^{-1}\|zf''\|_{\mathcal{V}_{s+1}}}\\
&=\int_{-\pi/2}^{\pi/2}\cos^{s+1}\psi \int_0^\infty r
\left|
\int_{-\pi/2}^{\pi/2} \cos^\sigma\varphi \int_0^\infty 
\frac{\rho^{\sigma+1}f'(\rho e^{i\varphi})}{(r e^{i\psi}+\rho e^{-i\varphi})^{\sigma+3}}\,d\rho\,d\varphi\right|
\,dr\,d\psi\\
&\le \int_{-\pi/2}^{\pi/2} \cos^\sigma \varphi \int_{-\pi/2}^{\pi/2} \cos^{s+1}\psi
\int_{0}^\infty\int_0^\infty \frac{t |f'(\rho e^{i\varphi})|\,dt\,d\rho}{(t^2+2t\cos(\varphi+\psi)+1)^{(\sigma+3)/2}} 
\,d\psi\,d\varphi \\
&= \int_{-\pi/2}^{\pi/2} \cos^{\sigma-s}\varphi \, G_{h,\beta,s+1}(\varphi)
\int_0^\infty  |f'(\rho e^{i\varphi})|
\,d\rho\,d\varphi,
\end{align*}
where $h(t)=t+t^\sigma, \,  t\in (0,1)$, $\beta = (\sigma+3)/2$ and $G_{h,\beta,s+1}(\varphi)$ is defined in \eqref{defG}, noting that $2\beta - (s+1)-2 = \sigma-s$.
Now the estimate (\ref{DerAnD}) follows from Lemma \ref{D0012}(b).

For \eqref{DerAnE}, let $g = f_\tau - f_t$.  Without loss, assume that $0 < t < \tau$.  Then
\[
g'(z) = \tau f'(\tau z) - t f'(tz) = \int_t^\tau \left(f'(rz) + rz f''(rz)\right)  \, dr = \int_t^\tau \frac{d}{dz} (zf'(rz)) \,dr.
\]
Hence, by Fubini's theorem,
\begin{align*}
\|g\|_{\D_{s+1}} &\le \int_{\C_+} \int_t^\tau \frac{(\Re z)^{s+1}}{|z|^{s+2}} \left|\frac{d}{dz} (zf'(rz))\right|\,dr \,dS(z)\\
&= \int_t^\tau \frac{\|zf'\|_{\D_{s+1}}}{r} \,dr \le \frac{C'_s\|f\|_{\D_s}}{t} (\tau-t),
\end{align*}
since the $\D_{s+1}$-norm is invariant under the change of variable $z\mapsto rz$ (Lemma \ref{leminv}(iii)).
\end{proof}

The following corollary is easily proved by induction.

\begin{cor} \label{D00n}
If $f \in \D_\infty$ and $n \in \N$, then $z^nf^{(n)} \in \D_\infty$.
\end{cor}

\begin{remark}
Lemma \ref{D001} is sharp in the sense that for any $s>0$ and $\sigma \in (-1,s+1)$, there exist functions $f \in \D_s$ for which $zf' \notin \D_\sigma$.  For example, the function $f_\nu(z) := z^\nu e^{-z}$ has these properties if $\max\{0,\sigma-1\}<\nu< s$.   This follows directly from Example \ref{dm}.   
\end{remark}

\section{Hardy-Sobolev algebras on sectors}\label{hardy-sobolev}

\subsection{$H^p$-spaces on the right half-plane and their norms}

\label{hardy_rays}

In this section and in Section \ref{hardy_sec} we will study the Hardy spaces $H^1(\Sigma_\psi)$ defined on sectors $\Sigma_\psi, \, \psi \in (0,\pi)$.
The properties of such spaces 
 are similar to the properties of the classical Hardy space $H^1(\C_+)$, 
though their theory seems to be more involved. The Hardy spaces $H^p(\Sigma_\psi)$ have been studied, mostly for $p > 1$, 
but the results are scattered around
various places in the literature, which is often obscure, and some proofs contain rather complicated,
incomplete or vague arguments. We propose a streamlined
(and probably new) approach avoiding the use of Carleson measures or log-convexity, 
and we obtain a new result (Corollary \ref{HHH}) on the way. 
The case $p = 1$
does not require any significant adjustments as we illustrate below.  Standard references for the theory of Hardy spaces on the right half-plane are \cite{Duren1} and \cite{Garnett}.

We set out the situation  when $\psi=\pi/2$ in this section, and the case of general $\psi$ in Section \ref{hardy_sec}.
Although we are mainly interested
in $H^1$-spaces, we present statements
which are valid for $H^p$-spaces with $p \in [1, \infty)$, since the arguments are the
same for all such $p$.

Let $1\le p<\infty$. The classical Hardy space $H^p(\C_{+})$ in the right half-plane $\C_+$ is defined as
\[
H^p(\C_{+})=\left \{g\in \mbox{Hol}(\C_{+}): \|g\|_{p}:=\sup_{\a>0}\,
\left(\int_{-\infty}^\infty |g(\a+i\b)|^p\,d\b\right)^{1/p}<\infty \right \}.
\]
It is well-known that $\|\cdot\|_p$ is a norm on $H^p(\mathbb C_+)$ and  $(H^p(\C_+),\|\cdot\|_{p})$ is a Banach space.
Moreover, for almost every $t \in \mathbb R$ there exists a sectorial limit
$g(it):=\lim_{z \to it} g(z)$ in $\C$. For every $g \in H^p(\C_+)$ one has $\lim_{\a \to 0} g(\a+i\cdot)= g(i\cdot)$ in $L^p(\mathbb R)$,
and $\|g\|_{H^p(\C_+)} := \|g\|_p = \|g(i\cdot)\|_{L^p(\R)}$.

One may also consider the normed space $(H^{p}(\Sigma_{\pi/2}), \|\cdot\|_{H^p(\Sigma_{\pi/2})})$ as the space of all 
$g\in \Hol(\C_{+})$ such that
\[\|g\|_{H^p(\Sigma_{\pi/2})}:=\sup_{|\varphi|<\pi/2}\,
\left(\int_0^\infty \left(|g(te^{i\varphi})|^p+
|g(te^{-i\varphi})|^p\right)\,dt\right)^{1/p}<\infty.
\]
and $(H^{p}_{*}(\Sigma_{\pi/2}), \|\cdot\|_{H^p_{*}(\Sigma_{\pi/2})})$ as
\[
\left \{g\in \Hol(\C_{+}):
\|g\|_{H^p_*(\Sigma_{\pi/2})}:=\sup_{|\varphi|<\pi/2}\left(\int_0^\infty |g(te^{i\varphi})|^p\,dt\right)^{1/p} < \infty \right\}.
\]
It is clear that $H^p(\Sigma_{\pi/2})$ and $H^p_*(\Sigma_{\pi/2})$ coincide as vector spaces, and
\[
\|g\|_{H^p_*(\Sigma_{\pi/2})} \le \|g\|_{H^p(\Sigma_{\pi/2})}\le 2^{1/p}\|g\|_{H^p_*(\Sigma_{\pi/2})}.
\]

\begin{lemma}\label{11L1A}
Let $p\in [1,\infty)$.
Then $H^{p}(\Sigma_{\pi/2})\subset H^{p}(\C_{+})$ and
\begin{equation}\label{RelAB}
\|g\|_{p}\le \|g\|_{H^{p}(\Sigma_{\pi/2})}, \qquad g\in H^{p}(\Sigma_{\pi/2}).
\end{equation}
\end{lemma}

\begin{proof}
Fix $p \in [1,\infty)$. 
For fixed
$\gamma\in (1/2,1)$ define
\[
g_\gamma(z):=(\gamma z^{\gamma-1})^{1/p}g(z^\gamma)\in \Hol(\Sigma_{\pi/(2\gamma)}).
\]
Note that
\begin{equation}\label{Al11}
\int_0^\infty |g_\gamma(te^{i\varphi})|^p\,dt=
\int_0^\infty |g(te^{i\gamma\varphi})|^p\,dt, 
\end{equation}
and
\[
\int_{-\pi/2}^{\pi/2} \int_0^\infty |g_\gamma(te^{i\varphi})|^p\,dt \, d\varphi\le \frac{\pi}{2} \|g\|^p_{{H^{p}(\Sigma_{\pi/2})}}.
\]
By Fubini's theorem and H\"older's inequality there exist sequences
$(t_{1,n})_{n\ge 1}$ and $(t_{2,n})_{n\ge 1}$  such that $0<t_{1,n}<t_{2,n}$,
$t_{1,n}\to0$, $t_{2,n}\to \infty$ as $n\to\infty$, and
\begin{align}\label{seq112}
\lim_{n\to\infty}\, t_{1,n} \int_{-\pi/2}^{\pi/2}  |g_\gamma (t_{1,n}e^{i\varphi})|\,d\varphi&=0, \\
\lim_{n\to\infty}\,\int_{-\pi/2}^{\pi/2} |g_\gamma (t_{2,n}e^{i\varphi})|\,d\varphi&=0.  \label{seq113}
\end{align}

Let $\Omega_n:=\{z\in \C_{+}:\, t_{1,n}<|z|<t_{2,n}\}$.  By Cauchy's formula,
\begin{equation}\label{LL112}
g_\gamma(z)=\frac{\a}{\pi i}\int_{\partial\Omega_n}\frac{g_\gamma(\l) }{(\l-z)(\l+\overline{z})} \,d\l,\quad z= \a+i\b \in \Omega_n,
\end{equation}
for large $n$.  Passing to the limit in \eqref{LL112} as $n \to \infty$, and using (\ref{Al11}), (\ref{seq112}) and (\ref{seq113}),
we infer that $g_\gamma$ satisfies the Poisson formula
\begin{equation}\label{Puas2}
g_\gamma(z)=\frac{\a}{\pi}\int_{-\infty}^\infty
\frac{g_\gamma(it)}{(t-\b)^2+\a^2}\,dt.
\end{equation}
Hence, by  Young's inequality and (\ref{Al11}), for every $\a>0$,
\[
\|g_\gamma(\a+i\cdot)\|_{L^p(\R)}\le \|g_\gamma(i\cdot)\|_{L^p(\mathbb R)}\le \|g\|_{H^{p}(\Sigma_{\pi/2})}.
\]
Letting $\gamma\to 1$,  Fatou's Lemma implies  (\ref{RelAB}).
\end{proof}

\begin{lemma}\label{11L2A}
Let  $p\in [1,\infty)$.
Then $H^{p}(\C_{+})\subset H^{p}(\Sigma_{\pi/2})$ and
\begin{equation}\label{RelABC}
\|g\|_{H^{p}(\Sigma_{\pi/2})}\le \|g\|_{p}, \qquad g\in H^{p}(\C_{+}).
\end{equation}
\end{lemma}

\begin{proof}
First let $p=2$ so that $g\in H^2(\C_{+})$. Then, by
\cite[Ch.VIII, p.508]{Djrbashian},
there exists
$f\in L^1(\R, e^{\pi |t|}\,dt)$ such that $f \ge 0$ on $\mathbb R$, and for all  $|\varphi|\le \pi/2$,
\[
\int_0^\infty |g(te^{i\varphi})|^2\,dt=
\int_{-\infty}^\infty e^{2\varphi t} f(t)\,dt.
\]
Hence
\begin{align*}
\int_0^\infty \bigl(|g(te^{i\varphi})|^2+|g(t e^{-i\varphi})|^2\bigr)\,dt&=
2\int_{-\infty}^\infty \cosh(2\varphi t)f(t)\,dt \\
\le 2\int_{-\infty}^\infty \cosh(\pi t)f(t)\,dt&=\int_0^\infty \bigl(|g(it)|^2+|g(-it)|^2 \bigr)\,dt,\notag
\end{align*}
and (\ref{RelAB}) holds for $p=2$.

Let $p\in [1,\infty), \, p\not=2$, be fixed, and $g\in H^p(\C_{+})$. Then
$g(z)=B(z)\tilde g(z)$, $z \in \mathbb C_+$, where $B$ is the Blaschke product associated with $g$ and $\tilde g$ has no zeros in $\C_+$.  Then there is a well-defined holomorphic function  $g_p(z)=[\tilde g (z)]^{p/2}$ on $\C_+$ and $g_p \in H^2(\C_{+})$.  Using  (\ref{RelAB})
for $p=2$,
for all $|\varphi|<\pi/2$, we have
\begin{gather} 
\int_0^\infty \bigl(|g(te^{i\varphi})|^p+|g(te^{-i\varphi})|^p\bigr)\,dt
= \int_0^\infty \bigl(|g_p(te^{i\varphi})|^2+|g_p(te^{-i\varphi})|^2\bigr)\,dt\label{RRD} \\
\le \int_{-\infty}^\infty
|g_p(it)|^2\,dt=\int_{-\infty}^\infty
|g(it)|^p\,dt=\|g\|_p^p, \notag
\end{gather}
and (\ref{RelABC}) follows.
\end{proof}

Lemmas \ref{11L1A} and \ref{11L2A} imply the next statement.

\begin{cor}\label{HHH}
Let $p\in [1,\infty)$.
Then $H^{p}(\Sigma_{\pi/2})=H^{p}(\C_+)$, and for every $g\in H^{p}(\C_{+})$,
\begin{equation}\label{n_equality}
\|g\|_{H^{p}(\Sigma_{\pi/2})}=\|g\|_{H^{p}(\C_+)},
\end{equation}
and
\begin{equation}\label{Lux11}
\|g\|_{H^{p}_{*}(\Sigma_{\pi/2})}\le \|g\|_{H^p(\C_+)}\le 2^{1/p}\|g\|_{H^{p}_{*}(\Sigma_{\pi/2})}.
\end{equation}
\end{cor}

Note that the two-sided estimate (\ref{Lux11}) was proved in \cite{Sedl} and \cite{Luxemburg} in a more complicated way
(see also \cite{Akopjan},  \cite{dil}, \cite{Martirosian}, \cite{Kev1}, \cite{Kev2}). 
 The coincidence of norms in \eqref{n_equality} seems not to have been noted before. It appears to be quite useful, as we will see in the proof of Corollary \ref{Zb} below.

\begin{remark}\label{Lux}
The two-sided estimate (\ref{Lux11}) is sharp (and cannot be improved). Indeed, let $p\in [1,\infty)$ and let 
\[
f_k(z):=\frac{1}{\pi^{1/p}(z+1+ik)^{2/p}},\qquad k\in\N.
\]
Then for all $k$, we have $\|f_k\|_{H^p(\C_{+})}=1$ 
and, by direct estimates, \begin{align*}
&\left( \|f_k\|_{H^p_{*}(\Sigma_{\pi/2})} \right)^p=\int_0^\infty |f_k(te^{-i\pi/2})|^p\,dt\\
&=\frac{1}{\pi}\int_0^\infty \frac{dt}{(t-k)^2+1}\,dt
=\frac{1}{2}+\frac{\arctan k}{\pi}.
\end{align*}
Thus,
\[
\|f_0\|_{H^p(\C_{+})}=2^{1/p}\|f_0\|_{H^p_{*}(\Sigma_{\pi/2})} \qquad \text{and}\qquad
\lim_{k\to\infty}\,\frac{\|f_k\|_{H^p_{*}(\Sigma_{\pi/2})}}{\|f_k\|_{H^p(\C_{+})}}=1.
\]
In fact, for all $f \in \H^p_{*}(\Sigma_{\pi/2})$, the norm $\|f\|_{H^p_{*}(\Sigma_{\pi/2})}$ is attained at the boundary of $\Sigma_{\psi/2}$.
\end{remark}

\begin{cor}\label{Zb}
Let $g \in H^{p}(\Sigma_{\pi/2})$, $p\in [1,\infty)$.
Then there exist $g(\pm it):=\lim_{\varphi \to \pm \pi/2} g(te^{\pm i\varphi})$ for a.e.\ $t \in \mathbb R_+$,
$g(\pm i \cdot) \in L^p(\mathbb R_+)$, and
\begin{equation}\label{DH11}
\lim_{\varphi\to\pm \pi/2}\,
\int_0^\infty |g(te^{i\varphi})-g(\pm it)|^p\,dt=0.
\end{equation}
As a consequence, for every $g \in H^p(\Sigma_{\pi/2})$,
\begin{equation}\label{norm_h}
\|g\|_{H^p(\Sigma_{\pi/2})}=\|g(i\cdot)\|_{L^p(\mathbb R)}.
\end{equation}
\end{cor}

\begin{proof}
Let $p \in [1,\infty)$ be fixed. By Corollary \ref{HHH}, it suffices to prove \eqref{DH11} for
$g\in H^p(\C_{+})$.
If $g\in H^p(\C_{+})$, then as recalled above, for almost all $t \in \mathbb R$
there exists a sectorial limit $g(it):=\lim_{z \to it} g(z)$, and $g(i\cdot) \in L^p(\R)$.
 Therefore, we also have $\lim_{\varphi \to \pm \pi/2} g(te^{\pm i\varphi})=g(\pm it)$ for almost all $t$.
Using this and Fatou's Lemma, we infer from (\ref{RRD}) that
\begin{align*}
\limsup_{\varphi\to \pi/2}\,&\int_0^\infty \bigl(|g(te^{i\varphi})|^p+|g(te^{-i\varphi})|^p \bigr)\,dt\le
\int_{-\infty}^\infty |g(it)|^p\,dt\\
&\le \liminf_{\varphi\to\pi/2}\,\int_0^\infty \bigl(|g(te^{i\varphi})|^p+|g(te^{-i\varphi})|^p\bigr)\,dt,
\end{align*}
and hence
\begin{equation}\label{Duren}
\lim_{\varphi\to\pi/2}\,\int_0^\infty \bigl(|g(te^{i\varphi})|^p+|g(te^{-i\varphi})|^p \bigr)\,dt=
\int_{-\infty}^\infty |g(it)|^p\,dt.
\end{equation}
Then, by \eqref{Duren} and Lemma \ref{duren21}, using once again
 the pointwise a.e.\ convergence of  $g(te^{\pm i\varphi})$ to $g(\pm it)$ as $\varphi\to\pm \pi/2$,
we obtain \eqref{DH11}. Since $\|g\|_{H^p(\C_+)}=\|g(i\cdot)\|_{L^p(\mathbb R)},$ we get \eqref{norm_h} as well.
\end{proof}

For (formally) more general versions of   \eqref{n_equality} and  \eqref{Lux11} see \eqref{equivalent} and \eqref{iso} below.

\subsection{The spaces $H^1(\Sigma_\psi)$}\label{hardy_sec}

Now using the results of Section \ref{hardy_rays} for $\psi=\pi/2$, we develop basic properties of $H^1(\Sigma_\psi)$ for any $\psi \in (0,\pi)$.   Define the Hardy space $H^1(\Sigma_\psi)$ on the sector $\Sigma_\psi$ to be the space of all functions $f \in \Hol(\Sigma_\psi)$ such that
\begin{equation}\label{CC2}
\|f\|_{H^1(\Sigma_\psi)} := \sup_{|\varphi|<\psi}
\int_0^\infty \bigl(|f(te^{i\varphi})|+
|f(te^{-i\varphi})| \bigr) \,dt <\infty.
\end{equation}

We will also consider a non-symmetric version of $H^1(\Sigma_\psi)$, defined as
\begin{equation*}
H^1_{*}(\Sigma_{\psi}):=\left\{f \in \Hol\, (\Sigma_{\psi}):
\|f\|_{H^1_{*}(\Sigma_{\psi})}:=\sup_{|\varphi|< \psi}\,\int_0^\infty |f(t e^{i\varphi})|\,dt<\infty\right\}.
\end{equation*}

\begin{thm}\label{hardy0}
Let $\psi, \psi_1, \psi_2 \in (0,\pi)$.
\begin{enumerate}[\rm(i)]
\item  $f \in H^1 (\Sigma_{\psi})$ if and only if $f \in H^1_*(\Sigma_\psi)$, and then
\begin{equation}\label{equivalent}
2^{-1} \|f\|_{H^1(\Sigma_{\psi})} \le \|f\|_{H^1_*(\Sigma_{\psi})}\le \|f\|_{H^1(\Sigma_{\psi})}.
\end{equation}
\item For any $\psi_1,\psi_2 \in (0,\pi)$, the map
\begin{eqnarray*}
H^1(\Sigma_{\psi_1}) &\to& H^1(\Sigma_{\psi_2})\\
f(z) &\mapsto &\frac{\psi_1}{\psi_2} z^{(\psi_1/\psi_2) -1}f(z^{\psi_1/\psi_2})
\end{eqnarray*}
is an isometric isomorphism of $H^1(\Sigma_{\psi_1})$ onto $H^1(\Sigma_{\psi_2})$, and of $H^1_*(\Sigma_{\psi_1})$ onto $H^1_*(\Sigma_{\psi_2})$.
\item If $f \in H^1(\Sigma_\psi)$, then the limits
$f(re^{\pm i\psi}):=\lim_{\varphi \to \pm \psi} f(re^{i\varphi})$ exist a.e., and in the $L^1$-sense with respect to $r$.   Moreover
\begin{equation}\label{Boch1}
f(z)=
\frac{1}{2\pi i}
\int_{\partial\Sigma_{\psi}}\,\frac{f(\lambda)}{\lambda-z} \,d\lambda, \quad
z\in \Sigma_\psi.
\end{equation}
\item If  $f \in H^1(\Sigma_\psi)$, then
\begin{align}
\|f\|_{H^1(\Sigma_\psi)}&=\int_{0}^{\infty}\bigl(|f(te^{i\psi})|+
|f(te^{-i\psi})| \bigr) \,dt.  \label{iso} 
\end{align}
\item $H^1(\Sigma_{\psi})$ and $H^1_*(\Sigma_{\psi})$ are Banach spaces.
\end{enumerate}
\end{thm}

\begin{proof}
The proof of (i) is clear, and (ii) is a direct verification.   
For $\psi = \pi/2$, the statements (iii) and (iv), excluding \eqref{Boch1}, were proved in Corollaries \ref{HHH} and \ref{Zb}, and (v) is well-known.   Then the general cases are reduced to the case when $\psi=\pi/2$,
by means of (ii).   

The  Cauchy formula \eqref{Boch1} is well-known for $\psi=\pi/2$ (see, for example, \cite[Theorem 11.8]{Duren1}).   For general $\psi$, we may argue similarly to the proof of Lemma \ref{11L1A}, as follows. 

Since $f \in H^1(\Sigma_{\psi})$, 
\[
\int_0^\infty \int_{-\psi}^\psi |f(te^{i\varphi})| \,d\varphi\,dt < \infty.
\]
Hence there exist sequences  $(t_{1,n})_{n\ge1}$ and $(t_{2,n})_{n\ge1}$ such that $0 < t_{1,n} < t_{2,n}$, $t_{1,n} \to 0, \, t_{2,n} \to \infty$ as $n\to\infty$, and
\[
\lim_{n\to \infty} \int_{-\psi}^\psi t_{1,n} |f(t_{1,n} e^{i\varphi})| \,d\varphi = \lim_{n\to\infty}
\int_{-\psi}^\psi |f(t_{2,n} e^{i\varphi})| \,d\varphi = 0.
\]
By applying Cauchy's theorem around the boundary of
\[
\left\{z \in \partial \Sigma_{\psi - n^{-1}} : t_{1,n} < |z| < t_{2,n} \right\}
\]
for large $n$, and taking the limit, we obtain \eqref{Boch1}.
\end{proof}

\begin{rem}
In addition to \eqref{iso}, it is also possible to prove that
\[
 \|f\|_{H^1_{*}(\Sigma_\psi)} = \max \left(\int_{0}^{\infty}|f(te^{-i\psi})|\, dt,\int_{0}^{\infty}|f(te^{i\psi})|\,dt\right).
\]
This requires additional techniques, and it is not used in this paper.
\end{rem}

\subsection{Functions with derivatives in $H^1(\Sigma_\psi)$}

For $\psi\in (0,\pi)$, let us introduce the space
 \[\H_\psi:=\left\{ f \in \Hol(\Sigma_\psi): f' \in H^1 (\Sigma_{\psi})\right\}.\]
In view of Corollary \ref{HHH},
\begin{equation} \label {HH11}
  \H_{\pi/2} = H^{1,1}(\C_+):=\{f \in \Hol(\C_+): f'\in H^1(\C_+)\},
\end{equation}
and we may sometimes use the notation $H^{1,1}(\C_+)$ instead of $\H_{\pi/2}$.

Such function spaces are often called Hardy-Sobolev spaces, and we will also use
this terminology sporadically.
Spaces more general than $\H_\psi$ appear in \cite{Domelevo}.
Namely, for $f \in \Hol(\Sigma_\psi)$ it was required in \cite{Domelevo} that the boundary values of $f$ exist and belong (after an appropriate ``rescaling'') to a Besov space $B^s_{\infty, 1}, \, s>0$.   One can develop a similar approach to those spaces, but we do not see much advantage in such generality within the present context.

\begin{thm}\label{hardy2}
Let $f \in \H_\psi,\, \psi\in (0,\pi)$.
\begin{enumerate}[\rm(i)]
\item The function $f$ extends to a continuous bounded function on $\overline{\Sigma}_{\psi}$.
\item The limit
\[
f(\infty):=\lim_{|z|\to \infty, z \in \Sigma_\psi}\,f(z)
\]
exists.
\item One has
\begin{equation}\label{bound_for_sup}
\|f\|_{H^\infty(\Sigma_\psi)} \le |f(\infty)| + \|f'\|_{H^1(\Sigma_\psi)}.
\end{equation}
In particular, the evaluation functionals $\delta_z,\, z\in\Sigma_\psi$, are continuous on $\H_\psi$.
\end{enumerate}
\end{thm}

\begin{proof}
Let $\psi = \pi/2$.  Since $\mathcal H_{\pi/2}=H^{1,1}(\mathbb C_+) \subset \Bes$ by \eqref{HH11} and \cite[Proposition 2.4]{BGT}, the statement (i) 
 follows from \cite[Proposition 2.2(iv)]{BGT}, and (ii) follows from \cite[Proposition 2.4]{BGT}.  In the general case, the map $f(z) \mapsto f(z^{2\psi/\pi})$ is an isomorphism of $\H_\psi$ onto $H^{1,1}(\C_+)$, by Theorem \ref{hardy0}(ii), so (i) and (ii) hold for $\H_\psi$.   The statement (iii) is easily seen.
\end{proof}

It follows from Theorem \ref{hardy2} that $\H_\psi$ is an algebra for every $\psi\in (0,\pi)$.   We define a  norm on $\H_\psi$ by 
\begin{equation}
\|f \|_{\H_\psi}:=\|f\|_{H^{\infty}(\Sigma_\psi)}+\|f'\|_{H^1(\Sigma_\psi)}, \qquad f \in \H_\psi.
\end{equation}
This is easily seen to be an algebra norm.   Theorem \ref{hardy2}(iii) shows that 
\begin{equation} \label{hnorms}
\|f\|'_{\H_\psi} := |f(\infty)| + \|f'\|_{H^1(\Sigma_\psi)}
\end{equation}
is an equivalent norm on $\H^1(\Sigma_\psi)$.

The following lemma is simple, but crucial for our theory.   The completeness of the norm is a standard fact, the scale-invariance is trivial, and the final isomorphism follows from Theorem \ref{hardy0}(ii).

\begin{lemma}\label{simple}
For every $\psi\in (0,\pi)$, the space $(\H_\psi, \|\cdot\|_{\H_\psi})$ is a Banach algebra.  For $t>0$, the map $f(z) \mapsto f(tz)$ is an isometric algebra isomorphism on $\H_\psi$.
Moreover, for any $\psi_1,\psi_2 \in (0,\pi)$, the map
\begin{eqnarray*}
\H_{\psi_1}&\to& \H_{\psi_2}\\
f(z)&\mapsto& f(z^{\psi_1/\psi_2})
\end{eqnarray*}
is an isometric algebra isomorphism.
\end{lemma}

We now give some examples of functions in $\H_\psi$ which will play important roles in subsequent sections of this paper.  The first example is of similar type to Example \ref{resd}.

\begin{exas} \label{hexs}
Let $\psi \in (0,\pi/2)$, and $\l = |\l| e^{i\varphi} \in \Sigma_{\pi-\psi}$.
\begin{enumerate}[\rm1.]
\item Let $\gamma >0$, and consider the function $r_\l^\gamma(z) = (\l+z)^{-\gamma}, \,z \in \Sigma_{\psi}$.    Then 
 $r_{\l}^{\gamma}\in \H_\psi$ and
\begin{equation} \label{ff}
\|r_{\l}^{\gamma}\|_{\H_\psi} = 
\int_{\partial\Sigma_\psi} \frac{\gamma |dz|}{|z+\l|^{\gamma+1}} \le \frac{2}{\sin^{\gamma+1} ((\varphi-\psi)/2) \,|\l|^{\gamma}},
\end{equation}
where we have used Lemma \ref{trig}.  Thus $r_\l^\gamma \in \H_\psi$, and there exists $C_{\varphi,\psi,\gamma}$ such that
\begin{equation} \label{rtl2}
\|r_\l^\gamma\|'_{\H_\psi} \le \frac{C_{\varphi,\psi,\gamma}}{|\l|^\gamma}, \qquad \l \in \Sigma_\varphi, \, \varphi < \pi - \psi.
\end{equation}
In particular, if $\gamma=1$, then 
\begin{equation}\label{res_sigma_cal_intro}
\|r_{\l}\|_{\H_\psi} \le \frac{2}{\sin^{2} ((\varphi-\psi)/2) \,|\l|}, \qquad \l \in \Sigma_{\pi-\varphi}.
\end{equation}
This property will be important for the proof of Lemma \ref{DH111} and hence of Theorem \ref{DM122}, and eventually of Theorem \ref{sigma_cal_intro} (see Theorem \ref{sigma_cal}). A more general estimate will be given in Corollary \ref{BF}.
\item Let  $\gamma \in (0,\pi/(2\psi)$ and $e_{\gamma,\l}(z) := e^{-\l z^\gamma}, \,z \in \Sigma_\psi$.
Then $e_{\gamma,\l} \in \H_\psi$ and 
\[
\|e_{\gamma,\l}\|'_{\H_{\psi}} = \|e_{1,\l} \|'_{\H_{\gamma\psi}} \le \int_{\partial\Sigma_\psi} |\l| e^{-\Re\l z} \, |dz|
\le \frac{1}{\cos(\varphi+\psi)} + \frac{1}{\cos (\varphi-\psi)}.
\]
\end{enumerate}
\end{exas}

More examples can be found in Sections \ref{handd} and \ref{bandh}.  In particular, Lemma \ref{Dh1} shows that the restriction of any function in $\D_\infty$ to $\Sigma_\psi, \, \psi \in (0,\pi/2)$, belongs to $\H_\psi$.

The following lemma is a result of Fatou type closely related to Corollary \ref{Fatou}.

\begin{lemma}\label{FatouH}
Let $\psi \in (0,\pi)$ and $(f_k)_{k=1}^\infty \subset \H_\psi$ be such that $\sup_{k \ge 1} \|f_k\|_{\H_\psi}<\infty$ and $f(z):=\lim_{k \to \infty} f_k(z)$ exists for all $z \in \mathbb C_+$. Then $f \in \H_\psi$.
\end{lemma}

\begin{proof}
By Theorem \ref{hardy2}, the functions $\{f_k: k\ge 1\}$ are uniformly bounded on $\Sigma_\psi$.  By Vitali's theorem, $f$ is holomorphic, and $f_k'(z) \to f'(z)$ as $k \to \infty$ for each $z \in \Sigma_\psi$.   By Fatou's Lemma, for $|\varphi| < \psi$,
\[
\int_0^\infty |f'(te^{i\varphi})| \,dt \le \liminf_{k\to\infty} \int_0^\infty |f_k'(te^{i\varphi})| \,dt \le \sup_{k\ge1} \|f_k\|_{\H_\psi}.
\]
Thus $f \in \H_\psi$.
\end{proof}

\subsection{The spaces $H^{1,1}(\C_+)$ and $\lt L^1$}

In \cite[Proposition 2.4]{BGT}, we showed that $\H_{\pi/2} \embedi \Bes$.  We will now show a stronger result that $H^{1,1}(\C_+) \embedi \lt L^1 + \C \subset \LT$, where $(\LT,\|\cdot\|)_{\mathrm{HP}}$ is the Hille-Phillips algebra as in Section \ref{prelims}.  In particular, it shows that the Laplace transforms of singular measures on $(0,\infty)$ are not in $H^{1,1}(\C_+)$, which may be of interest.

\begin{thm}\label{ACM}
If $f\in H^{1,1}(\C_{+})$ then there exists $g \in L^1(\R_+)$ such that $f = f(\infty) + \lt g$.   Moreover, there is an absolute constant $C$ such that 
\begin{equation} \label{AW1}
\|f\|_{\mathrm{HP}} \le C\|f\|_{H^{1,1}(\C_+)}.
\end{equation}
\end{thm}

\begin{proof}
Let $f\in H^{1,1}(\C_{+})$, and, for $n\in\N$, let 
\begin{align*}
f_n(z) &:=f(z)-f(z+n),\qquad z \in \C_+,  \\
g_n(t) &:= -\frac{2}{\pi} t(1-e^{-nt})\int_0^\infty \alpha e^{-\alpha t}
\int_{-\infty}^\infty
f'(\alpha+i\beta) e^{-i\beta t}\,d\beta\,d\alpha, \quad t>0.
\end{align*}
Then $f_n \in H^{1,1}(\C_+)$, and $\|f_n\|_{H^{1,1}(\C_+)} \le 2\|f\|_{H^{1,1}(\C_+)}$.  Moreover,
\begin{align*}
|g_n(t)| &\le \frac{2}{\pi}t(1-e^{-nt})\int_0^\infty \alpha e^{-\alpha t}
\int_{-\infty}^\infty
|f'(\alpha+i\beta)|\,d\beta\,d\alpha\\
&\le \frac{2}{\pi} \|f'\|_{H^1(\C_{+})}t(1-e^{-nt})\int_0^\infty \alpha e^{-\alpha t}
\,d\alpha= \frac{2}{\pi}\|f'\|_{H^1(\C_{+})}\frac{(1-e^{-nt})}{t}\\
&\le  \frac{2n}{\pi} \|f'\|_{H^1(\C_{+})}.
\end{align*} 
By the reproducing formula for $\Bes$ (see Remark \ref{Brep} or \cite[Proposition 2.20]{BGT}) and Fubini's theorem,  
\begin{align*}
f_n(z)=&-\frac{2}{\pi}\int_0^\infty
\alpha\int_{-\infty}^\infty
f'(\alpha+i\beta)
\left(\frac{1}{(z+\alpha-i\beta)^{2}}-
\frac{1}{(z+n+\alpha-i\beta)^{2}}\right)
d\beta\,d\alpha\\
=&-\frac{2}{\pi}
\int_0^\infty \alpha
\int_{-\infty}^\infty 
f'(\alpha+i\beta)
\left(\int_0^\infty e^{-(z+\alpha-i\beta)t}
t(1-e^{-nt})\,dt\right)\,d\beta\,d\alpha\\
=&-\frac{2}{\pi}\int_0^\infty e^{-zt} g_n(t)\,dt,\quad z\in \C_{+}.
\end{align*}
It follows that $(\pi/2) f_n'$ is the Laplace transform of $t g_n(t)$, and then $tg_n(t)$ is the inverse Fourier transform  of $(\pi/2)f_n'(i\cdot)\in L^1(\R)$.  By Hardy's inequality in the form of \cite[p.198]{Duren1},
\begin{equation}\label{H11}
\int_0^\infty |g_n(t)|\,dt =\int_0^\infty \frac{|tg_n(t)|}{t}\,dt \le
 \frac{\pi}{4}\|f_n'\|_{H^1(\C_{+})}
\le \frac{\pi}{2}\|f'\|_{H^1(\C_{+})}.
\end{equation}
Moreover, 
\[
f(z)=f(\infty)+\lim_{n\to\infty} f_n(z),\qquad z\in \overline{\C}_{+},
\]
and then by \cite[Theorem 1.9.2]{Rud} 
we obtain that
$f\in \mathcal{LM}$ and (\ref{AW1}) holds.

Now let
\[
u(z) := f'(z+1) -\int_0^\infty e^{-zt}e^{-t}t \, d\mu(t),\qquad z\in \C_{+}.
\]   
Since $f\in H^\infty(\C_{+})$ and $f'\in H^1(\C_{+})$, we have that $u\in H^1(\C_+) \cap H^\infty(\C_+) \subset H^2(\C_{+})$.  Hence $u = \lt h$ for some $h \in L^2(\R_+)$, so 
\[
u(z) = \int_0^\infty e^{-zt} h(t) \, dt, \qquad z \in \C_+.
\]
From the uniqueness properties of Laplace transforms it follows that 
\[
-e^{-t}t\,\mu(dt)=h(t)\,dt.
\]
Thus $\mu$ is absolutely continuous on $(0,\infty)$, with Radon-Nikodym derivative $g.$   Since $\mu$ is a bounded measure, $g \in L^1(\R_+)$, and 
\[
\mu(dt)=f(\infty)\delta_0 +g(t)\,dt.
\]
Hence $f = f(\infty) + \lt g$.
\end{proof}

\subsection{The spaces $\H_\psi$ and $\D_s$}  \label{handd}
 Since $H^{1,1}(\C_+) = \H_{\pi/2}$ (by Corollary \ref{HHH}), we have shown in Theorem \ref{ACM} that
\[
\H_{\pi/2} \embedi \lt L^1 + \C \subset \LT \embedi \Bes \embedi \D_s
\]
if $s >0$.  In the next lemma, we show that, for all $s>-1$, $\H_{\pi/2} \embedi \D_s$ and $\mathcal D_s \embedr \H_\psi$ for every $\psi\in (0,\pi/2)$.  Moreover it follows that $\D_s^\infty \embedr \H_\psi$.  

\begin{lemma}\label{Dh1}
\begin{enumerate}[\rm(i)]
\item  If $f\in \H_{\pi/2}$, then $f\in \mathcal{D}_s^\infty$ for every $s>-1$, and
\begin{align}\label{embedh}
\|f\|_{\mathcal D_{s,0}}
&\le B\left(\frac{s+1}{2},\frac{1}{2}\right)\|f'\|_{H^1(\C_+)}, \\
\|f\|_{\mathcal D_{s}^\infty }
&\le \max \left\{1, B\left(\frac{s+1}{2},\frac{1}{2}\right)\right\}\|f\|_{\H_{\pi/2}}.  \notag
\end{align}
\item If $f\in \mathcal{D}_s$, $s>-1$, then 
$f \in \H_\psi$ for every $\psi \in (0,\pi/2)$, and
\begin{equation}\label{Haab}
\|f'\|_{H^1(\Sigma_\psi)}\le \frac{2^{s}}{\pi\cos^{s+2}(\psi/2+\pi/4)}\|f\|_{\mathcal{D}_{s,0}}.
\end{equation}
\end{enumerate}
\noindent Thus, for all $\psi \in (0,\pi/2)$ and $s>-1$, there are natural continuous embeddings
\[
\H_{\pi/2} \embedi \mathcal D_s \embedr  \H_\psi.
\]
\end{lemma}

Note that the estimates  \eqref{rtl2} and \eqref{embedh} for functions $r_\l^\gamma$ reproduce the estimate \eqref{rtl}, with different constants.

\begin{proof}
Let $s >-1$ be fixed, and let $f \in \H_{\pi/2}$.  Using \eqref{equivalent}, we have
\begin{align*}
\|f\|_{\mathcal D_{s,0}}&= \int_{-\pi/2}^{\pi/2}\cos^s\varphi
\int_0^\infty |f'(te^{i\varphi})|\,dt\,d\varphi \\
&\le \left(\int_{-\pi/2}^{\pi/2}\cos^s\varphi\,d\varphi\right)
\sup_{|\varphi|\le \pi/2}\int_0^\infty |f'(te^{i\varphi})|\,dt \\
&= B\left(\frac{s+1}{2},\frac{1}{2}\right)\|f'\|_{H^{1}(\C_+)},
\end{align*}
and (i) follows.

To prove (ii),  note that if $f \in \mathcal D_s$, then by  Corollary \ref{Repr},
\begin{equation*}
f'(z)=\frac{(s+1)2^s}{\pi}\int_0^\infty \alpha^s\int_{-\infty}^\infty \frac{f'(\alpha+i\beta)}{(z+\alpha-i\beta)^{s+2}}\,d\beta\,d\alpha,
\quad z\in \C_{+}.
\end{equation*}
Hence  using \eqref{Fc},  for every $\psi \in (0,\pi/2)$,
we obtain
\begin{align*}
\lefteqn {\hskip-10pt 2^{-s}\int_0^\infty |f'(te^{\pm i\psi})|\,dt}\\
&\le \frac{(s+1)}{\pi}\int_0^\infty
\int_0^\infty \alpha^s\int_{-\infty}^\infty \frac{|f'(\alpha+i\beta)|}{|te^{\pm i\psi}+\alpha-i\beta|^{s+2}}\,d\beta\,d\alpha\,dt\\
&\le  \frac{s+1}{\pi\cos^{s+2}(\psi/2+\pi/4)}
\int_0^\infty \alpha^s\int_{-\infty}^\infty 
\left(
\int_0^\infty \frac{|f'(\alpha+i\beta)|}{(t+|\alpha+i\beta|)^{s+2}}\right) \,dt\,d\beta\,d\alpha\\
&= \frac{1}{\pi\cos^{s+2}(\psi/2+\pi/4)}
\int_0^\infty \alpha^s\int_{-\infty}^\infty \frac{|f'(\alpha+i\beta)|}{|\alpha+i\beta|^{s+1}}\,d\beta\,d\alpha\\
&=\frac{1}{\pi\cos^{s+2}(\psi/2+\pi/4)}\|f'\|_{\mathcal{V}_s},
\end{align*}
and \eqref{Haab} follows.
\end{proof}

For a function $f \in \Hol(\Sigma_\psi)$, $\gamma>0$ and $0<\varphi\le \min\{\pi,\psi/\gamma\}$ , define $f_\gamma \in \Hol(\Sigma_{\varphi})$ by
\[
f_\gamma(z):= f(z^\gamma), \quad z \in \Sigma_{\varphi}.
\]

\begin{cor} \label{gggC}
Let $f \in \D_s, \, s>-1$, and let $\gamma \in (0,1)$.   Then $f_\gamma \in \D_\sigma^\infty \cap \H_{\pi/2}$ for all $\sigma>-1$.   Moreover, for each $s>-1$ and $\sigma>-1$, there exist constants $C_{s,\sigma,\gamma}$ and $\tilde C_{s,\gamma}$ such that
\[
\|f_\gamma\|_{\D_\sigma^\infty} \le C_{s,\sigma,\gamma} \|f\|_{\D_s}  \quad \text{and} \quad 
\|f_\gamma\|_{\H_{\pi/2}} \le \tilde C_{s,\gamma} \|f\|_{\D_s},
\qquad f \in \D_s.
\]
\end{cor}

\begin{proof}
Using Lemma \ref{Dh1}, (i) and (ii), and Lemma \ref{simple}, we see firstly that $f \in \H_{\pi\gamma/2}$, and then $f_\gamma \in \H_{\pi/2} \subset \D_\sigma^\infty$.    Moreover each of the embeddings from Lemma \ref{Dh1} is continuous, and the map $f \mapsto f_\gamma$ is isometric.
\end{proof}

Now we relate the spaces $\mathcal H_\psi$ and $\mathcal D_s$ to another class of spaces used in the literature on functional calculi.
For $\psi \in (0,\pi)$, let 
\begin{equation} \label{edef}
\mathcal E_\psi:=\left \{f \in {\Hol} (\Sigma_\psi): \|f\|_\psi:=\sup_{\varphi \in (0,\psi)}\int_{\partial \Sigma_{\varphi}} \frac{|f(z)|}{|z|}|dz| <\infty\right \}.
\end{equation}
It is easy to see that $(\mathcal E_\psi, \|\cdot \|_{\psi})$ is a Banach space,
and that
\[
\mathcal E_\psi=\{f \in {\Hol} (\Sigma_\psi): f(z)/z \in H^1(\Sigma_\psi)\}.
\]

\begin{prop}\label{weight1}
Let $f \in \mathcal E_\psi$ and let $g(z):=f(z)/z$.
Then, for every $\varphi\in (0,\psi)$,
\begin{equation}\label{Q2}
\|f'\|_{H^{1}(\Sigma_\varphi)} \le  \frac{1}{2\pi}
\left(\frac{\pi-\psi-\varphi}{\sin(\psi+\varphi)}+
\frac{\pi-\psi+\varphi}{\sin(\psi-\varphi)}\right)\|g\|_{H^1(\Sigma_\psi)},
\end{equation}
where we set $\frac{0}{\sin 0}:= 1.$ Thus, $\mathcal E_\psi \embedr \mathcal H_{\varphi}, \, \varphi\in (0,\psi)$, and $\mathcal E_{\pi/2} \embedi \mathcal D_s,\,  s>0$.
\end{prop}

\begin{proof}
By Cauchy's theorem,  for every $z\in\Sigma_\psi$,
\[
f'(z)
=\frac{1}{2\pi i}\int_{\partial\Sigma_\psi}\,
\frac{\l g(\l)\,d\l}{(z-\l)^2},
\]
Hence, for every $\varphi\in(0,\psi)$, by Fubini's theorem,
\begin{align*}
\int_0^\infty \left(|f'(\rho e^{i\varphi})| + |f'(\rho e^{-i\varphi})|\right)\,d\rho 
\le\frac{1}{2\pi}
\int_{\partial\Sigma_\psi}|g(\l)| \, J(\l, \varphi)
\, |d\l|,
\end{align*}
where
\begin{align*}
J(\l, \varphi)&:= \int_0^\infty\left(\frac{|\l|}{|\rho e^{i\varphi}-\l|^2}+
\frac{|\l|}{|\rho e^{-i\varphi}-\l|^2}\right)\ d\rho\\
\\
&=\int_0^\infty\frac{d\rho}{\rho^2-2\rho\cos(\psi+\varphi)+1}
+\int_0^\infty\frac{d\rho}{\rho^2-2\rho\cos(\psi-\varphi)+1}
\\
&=\frac{\pi-\psi-\varphi}{\sin(\psi+\varphi)}+
\frac{\pi-\psi+\varphi}{\sin(\psi-\varphi)},
\end{align*}
in view of \cite[p.311, (25)]{Prud}.  Hence $f \in \H_\varphi$ 
and \eqref{Q2} follows. Thus, $\mathcal E_\psi \embedr \mathcal H_{\varphi}$ for all  $\varphi\in (0,\psi)$.
Recalling Lemma \ref{Dh1} and that $\H_{\pi/2} = H^{1,1}(\mathbb C_+)$, we have proved the inclusion $\mathcal E_{\pi/2} \embedi \mathcal D_s, \, s>0$,
as well.
\end{proof}

Note that the inclusions in Proposition \ref{weight1} are strict. Indeed,  if $f(z)=z(z+1)^{-1}$ then  one has
$f \in \mathcal H_\psi$ for every $\psi \in (0,\pi)$, but $f \notin \mathcal E_{\psi}$ for any $\psi \in (0,\pi)$.  Moreover, if
\[
f(z)=\frac{z\,e^{-z}}{(z+1)\log^2(z+2)},
\]
then $z^{-1}f \in H^1(\C_{+})$, and $f \in \mathcal{D}_0$, but 
$f \not\in \mathcal{D}_s$ for any $s\in (-1,0)$.

The spaces $\mathcal E_\psi$ are studied in \cite[Chapter 10]{Hytonen};
see also \cite[Appendix H2]{Hytonen}, \cite[Section 6]{Haase_Sp} and \cite[Appendix C]{Haak_H}. 
To ensure the algebra property and to relate the spaces to the $H^\infty$-calculus, the authors considered the algebras $H^\infty(\Sigma_\psi)\cap \mathcal E_\psi$.  Lemma \ref{weight1} shows that the spaces  $\mathcal E_\psi$ are fully covered within the framework of the algebras $\mathcal D_\infty$ and $\mathcal H_\psi$.  These algebras will be associated to the more powerful functional calculi constructed in Corollary \ref{Compat} and Theorem \ref{sigma_cal}.

\subsection{Bernstein functions and $\H_\psi$}  \label{bandh}
To illustrate the relevance of the Hardy-Sobolev spaces, we show that the ``resolvent'' of a Bernstein function belongs to an appropriate Hardy-Sobolev space.  This observation will be used in Section \ref{norm_estimates} to provide a new proof of the permanence of subordination for holomorphic semigroups, one of the main results of \cite{GT15}; see also \cite{BGTad} and  \cite{BGT}.

Let $g$ be a Bernstein function, $\psi\in (0,\pi/2)$, and $\lambda \in \Sigma_{\pi-\psi}$.   
Let
\[
f(z, \lambda):=(\lambda+g(z))^{-1},\qquad
z\in \Sigma_\psi.
\]
If $\l \in \C_+$, it follows from Lemmas \ref{bsds} and \ref{Dh1} that $f(\cdot, \lambda) \in \mathcal \H_\psi$, and $\|f(\cdot,\l)\|_{\H_\psi} \le C_\psi/|\l|$, where $C_\psi$ is independent of $g$ and $\l\in\C_+$.   In order to obtain the correct angle, we will need to extend this to $\l \in \Sigma_{\pi-\varphi}$, where $\varphi \in (\psi,\pi/2)$.

\begin{cor}\label{BF}
Let $g$ be a Bernstein function, $\psi\in (0,\pi/2)$, $\varphi \in (\psi,\pi)$ and $\lambda \in \Sigma_{\pi-\varphi}$.  Then
\begin{equation}\label{bernst}
\|f(\cdot, \lambda)\|_{\mathcal \H_\psi} \le
2\left(\frac{1}{\sin(\min(\varphi,\pi/2))}+\frac{2}{\cos\psi \sin^2((\varphi-\psi)/2)}\right)\frac{1}{|\lambda|}.
\end{equation}
\end{cor}

\begin{proof}
For fixed $\psi \in (0,\pi/2)$,   $\varphi\in (\psi,\pi)$, and
$\lambda\in \Sigma_{\pi-\varphi}$, observe that
\[
\|f'(\cdot,\lambda)\|_{H^1(\Sigma_\psi)}
\le \int_0^\infty \left( \frac{|g'(te^{i\psi})|}{|\lambda+g(te^{i\psi})|^2} + \frac{|g'(te^{-i\psi})|}
{|\lambda+g(te^{-i\psi})|^2} \right) \,dt.
\]
Using the property (B1) for Bernstein functions and \eqref{Fc},  we have
\[
|\lambda+g(te^{\pm i\psi})|\ge
\sin((\varphi-\psi)/2)(|\lambda|+|g(te^{\pm i\psi})|), \qquad t \ge 0.
\]
Moreover, in view of (B3), for all $t \ge 0$,
\begin{align*}
|g(te^{\pm i\psi})|\ge
\Re g(te^{\pm i\psi})\ge g(t\cos\psi)\quad \text{and} \quad
|g'(e^{\pm i\psi}t)|&\le g'(t\cos\psi).
\end{align*}
Using (B2), we have
\begin{align*}
\|f'(\cdot, \lambda)\|_{H^1(\Sigma_\psi)}
&\le \frac{2}{\sin^2((\varphi-\psi)/2)}
\int_0^\infty\frac{g'(t\cos\psi)}
{(|\lambda|+g(t\cos\psi ))^2} \,dt\\
&\le  \frac{2}{\cos\psi \sin^2((\varphi-\psi)/2)}
\int_0^\infty\frac{ds}
{(|\lambda|+s)^2}\\
&=
\frac{2}{\cos\psi \sin^2((\varphi-\psi)/2)}\frac{1}{|\lambda|}.
\end{align*}
Thus $f(\cdot,\lambda) \in \H_\psi$ and $f(\infty,\lambda) = (\lambda + g(\infty))^{-1}$.  Since $|\arg\lambda|<\pi-\varphi$ and $g(\infty) \in [0,\infty]$, 
\begin{equation}\label{res_limit}
|f(\infty,\lambda)| 
\le \frac{1}{\sin(\min(\varphi,\pi/2))|\lambda|}.
\end{equation}
Now \eqref{bernst} follows from Theorem  \ref{hardy2}(ii).
\end{proof}

\subsection{Representations for functions in $\H_\psi$}  \label{arccotrep}

In this section we derive a reproducing formula for functions from $\H_\psi$
and obtain certain alternative representations for its kernel.

\begin{prop}\label{FrepH}
Let $f \in \H_\psi, \, \psi \in (0,\pi)$. Let $\gamma=\frac{2\psi}{\pi}$ and
\begin{equation}\label{fgamma_c}
 f_\gamma(z):=f(z^{\gamma}),\qquad z\in \C_{+}.
\end{equation}
Then
\begin{equation}\label{RepA1}
f(z)=f(\infty)-\frac{1}{\pi}\int_0^\infty \int_{-\infty}^\infty \frac{f_\gamma'(\alpha+i\beta)}{z^{1/\gamma}+\alpha-i\beta}\,d\beta\,d\alpha,\qquad
z\in \Sigma_\psi \cup \{0\}.
\end{equation}
\end{prop}

\begin{proof}
Since $f \in  \H_\psi$, Lemma \ref{simple} implies that $f'_\gamma \in H^1(\C_{+})$. Hence, by Lemma \ref{Dh1} and Theorem \ref{hardy2},
we have $f_\gamma\in \mathcal{D}_0\cap C(\overline {\mathbb C}_+)$. Then, in view of Corollary \ref{Repr},
\[
f_\gamma(z)=f_\gamma(\infty)-\frac{1}{\pi}\int_0^\infty \int_{-\infty}^\infty \frac{f_\gamma'(\alpha+i\beta)}{z+\alpha-i\beta}\,d\beta\,d\alpha,
\qquad z\in \C_{+}\cup \{0\},
\]
and (\ref{RepA1}) follows.
\end{proof}

\begin{cor}\label{alpha}
Let $f\in \mathcal{D}_\infty$, and $\gamma\in (0,1)$.
If $f_\gamma$ is given by \eqref{fgamma_c}, then
\begin{equation}\label{RepA}
f(z)=f(\infty)-\frac{1}{\pi}\int_0^\infty \int_{-\infty}^\infty \frac{f_\gamma'(\alpha+i\beta)}{z^{1/\gamma}+\alpha-i\beta}\,d\beta\,d\alpha,\quad
z\in \Sigma_{\pi\gamma/2}\cup\{0\}.
\end{equation}
\end{cor}

\begin{proof}
By Lemma \ref{Dh1}, $f_\gamma \in \H_{\pi\gamma/2}$, so \eqref{RepA} follows from \eqref{RepA1}.
\end{proof}

The next reproducing formula for functions in $\H_\psi$ resembles
\cite[Lemma 7.4]{Boy}, and it was, in fact, inspired by \cite[Lemma 7.4]{Boy}.  In particular, this formula replaces the double (area) integral in
\eqref{RepA1} with a  line integral, it involves boundary values of $f'$ rather than scalings of $f'$ (such as in \eqref{RepA1}), and it offers a different kernel which might sometimes be easier to deal with.  The $\arccot$ function has been defined in \eqref{arcdef}.

\begin{prop}\label{L7.4}
Let $f \in \mathcal \H_\psi$, $\psi\in(0,\pi)$.   Let $\nu=\pi/(2\psi)$, and 
\begin{equation}\label{fpsi}
 f_\psi(t):=\frac{f(e^{i\psi}t)+f(e^{-i\psi}t)}{2}, \quad t>0.
\end{equation}
Then
\begin{equation}\label{arctan}
f(z)=f(\infty)-\frac{2}{\pi}\int_0^\infty
f_\psi'(t)\, {\arccot}(z^\nu/t^{\nu})\,dt,\qquad z\in \Sigma_\psi\cup \{0\}.
\end{equation}
\end{prop}

\begin{proof}
Since $\arccot \l \in \mathcal D_0$ (see Example \ref{arc}), \eqref{qnrep} shows that
\begin{equation}\label{ArcN}
\arccot(\l)=\frac{1}{\pi}\int_0^\infty\int_{-\infty}^\infty \frac{(\l+u-iv)^{-1}}{(u+iv)^2+1}\,dv\, du,
\quad \l\in \C_{+}.
\end{equation}

Let $\gamma=1/\nu$, and $f_\gamma$ be given by \eqref{fgamma_c}.
By Proposition \ref{FrepH}, for $z \in \Sigma_\psi \cup \{0\}$, 
\begin{equation}\label{AlA}
f(z)-f(\infty)=
\frac{1}{\pi}\int_0^\infty \int_{-\infty}^\infty
\frac{f_\gamma'(\a+i\b)}{z^{\nu}+\a-i\b}\,d\b\,d\a.
\end{equation}
It follows from Lemma \ref{simple} that $f_\gamma' \in H^1(\Sigma_{\pi/2}) = H^1(\C_+)$, so by Cauchy's formula for functions in $H^1(\C_+)$ \cite[Theorem 11.8]{Duren1}, we have, for $\l  \in \C_{+}$,
\begin{align*}
f_\gamma'(\l) &= \frac{1}{2\pi} \int_{-\infty}^\infty \frac{f_\gamma'(it)}{\l-it} \,dt +  \frac{1}{2\pi}\int_{-\infty}^\infty \frac{f_\gamma'(it)}{\l+it} \,dt \\
&= - \frac{i}{\pi}\int_{-\infty}^\infty
\frac{tf_\gamma'(it)}{\l^2+t^2} \,dt\\
&=\frac{i\gamma}{\pi }\int_0^\infty
\frac{e^{i\pi(\gamma-1)/2}f'(e^{i\pi\gamma/2}t^\gamma)t^\gamma -
e^{-i\pi(\gamma-1)/2}f'(e^{-i\pi\gamma/2}t^\gamma)t^\gamma}{\l^2+t^2} \,dt\\
&=\frac{\gamma}{\pi}\int_0^\infty
\left(e^{i\psi}f'(e^{i\psi}t^\gamma)
+e^{-i\psi}f'(e^{-i\psi}t^\gamma)\right)\frac{t^{\gamma}}{\l^2+t^2} \,dt\\
&=\frac{2}{\pi}\int_0^\infty
\frac{f_\psi'(s) s^{\nu}}{\l^2+s^{2\nu}} \,ds.
\end{align*}
Therefore, by (\ref{AlA})
and (\ref{ArcN}), we obtain
\begin{align*}
&f(z)-f(\infty)\\
&=
-\frac{2}{\pi^2}\int_0^\infty f'_\psi(s)
\left(\int_0^\infty \int_{-\infty}^\infty
\frac{s^\nu\,d\b\, d\a}{(z^{\nu}+\a-i\b)((\a+i\b)^2+s^{2\nu})}\right)\,ds\\
&=-\frac{2}{\pi^2}\int_0^\infty f'_\psi(s)
\left(\int_0^\infty \int_{-\infty}^\infty
\frac{d\b\, d\a}{((z/s)^{\nu}+\a-i\b)((\a+i\b)^2+1)}\right)\,ds\\
&=-\frac{2}{\pi}\int_0^\infty f'_\psi(s)\arccot(z^\nu/s^\nu)\,ds,
\end{align*}
so \eqref{arctan} holds for $z \in \Sigma_\psi \cup \{0\}$.
\end{proof}

Proposition \ref{L7.4} motivates a more careful study of the kernel $\arccot (z^\nu)$.   The integral representation of this kernel will be crucial in deriving fine estimates for the $\H$-calculus for operators in Section \ref{alternative}. 

\begin{lemma}\label{log}
Let
$\psi\in (0,\pi)$ and  $\nu={\pi}/{(2\psi)}$.
Then
\begin{align}\label{rew}
\arccot(z^{\nu})&=\frac{1}{2\pi}\int_0^\infty
V_\psi(z,t)\log\left|\frac{1+t^\nu}{1-t^\nu}\right|
\,\frac{dt}{t}+\frac{1}{4 i}\int_{\Gamma_\psi}\,
\frac{d\l}{\l-z},\quad z\in \Sigma_\psi,
\end{align}
where 
\begin{align}\label{vpsi}
V_\psi(z,t)&:=-\frac{t}{2}
\left(\frac{ e^{-i\psi}}
{z-t e^{-i\psi}}
+\frac{ e^{i\psi}}{z-t e^{i\psi}}\right), \\
\Gamma_\psi&:= \{\l: |\l|=1,\,\arg\l \in (\psi,2\pi-\psi)\}.  \notag
\end{align}
\end{lemma}

\begin{proof}
We have
\begin{align*}
\arccot(z^{\nu})
=\frac{1}{2i}\log\left|\frac{z^\nu+i}{z^\nu-i}\right|
+\frac{1}{2}\arg\left(\frac{z^\nu+i}
{z^\nu-i}\right),\qquad z\in \Sigma_\psi.
\end{align*}
Since $\nu\psi=\pi/2$, for every $t>0, \, t\ne1$,
\[
\lim_{z\in \Sigma_\psi,\,z\to te^{i\psi}}
\arccot(z^{\nu})
=\frac{1}{2i}\log\left|\frac{1+t^\nu}{1-t^\nu}\right|
+\frac{1}{2}\arg\left(\frac{t^\nu+1}
{t^\nu-1}\right),
\]
and
\[
\lim_{z\in \Sigma_\psi,\,z\to te^{-i\psi}}\arccot(z^{\nu})
=-\frac{1}{2i}\log\left|\frac{1+t^\nu}{1-t^\nu}\right|
+\frac{1}{2}\arg\left(\frac{t^\nu-1}
{t^\nu+1}\right).
\]
Here,
\[
\arg\left(\frac{t^\nu+1}
{t^\nu-1}\right)=\arg\left(\frac{t^\nu-1}
{t^\nu+1}\right)=\pi\chi_{(0,1)}(t),\quad t>0, \, t\ne1,
\]
where $\chi_{[0,1]}$ is the characteristic function of $(0,1)$.
So,
\begin{equation}\label{zv}
\lim_{z\in \Sigma_\psi,\,z\to te^{\pm i\psi}}
\arccot(z^{\nu})
=\mp\frac{i}{2}\log\left|\frac{1+t^\nu}{1-t^\nu}\right|
+\frac{\pi}{2}\chi_{[0,1]}(t),
\end{equation}

Now fix $z \in \Sigma_\psi$.  Using
$\limsup_{|\l|\to\infty,\,\l\in \C_{+}}\,|\l\arccot \l|<\infty$, it follows from
 Cauchy's theorem that
\[
\arccot(z^{\nu})=\frac{1}{2\pi i}\int_{\partial\Sigma_\psi}\,
\frac{\arccot(\l^{\nu})}{\l-z}\,d\l, \quad \int_{\l\in\partial\Sigma_\psi, |\l|<1} \frac{d\l}{\l-z} = \int_{\Gamma_\psi} \frac{d\l}{\l-z}.
\]
Thus by \eqref{zv},
\begin{align*}
\arccot(z^{\nu})
&=\frac{1}{4\pi}\int_0^\infty
\log\left|\frac{1+t^\nu}{1-t^\nu}\right|
\left(\frac{1}{t-e^{i\psi}z}
+\frac{1}{t-e^{-i\psi}z}\right)\,dt
+\frac{1}{4 i}\int_{\Gamma_\psi}\,
\frac{d\l}{\l-z}\\
&=\frac{1}{2\pi}\int_0^\infty
V_\psi(z,t)\log\left|\frac{1+t^\nu}{1-t^\nu}\right|
\,\frac{dt}{t}+\frac{1}{4 i}\int_{\Gamma_\psi}\,
\frac{d\l}{\l-z}. \qedhere
\end{align*}
\end{proof}

\begin{rem}\label{Value}
Letting $z \to 0$  in (\ref{rew}) we obtain
\begin{align*}
\frac{\pi}{2}
&=\frac{1}{2\pi}\int_0^\infty
\log\left|\frac{1+t^\nu}{1-t^\nu}\right|
\frac{dt}{t}
+\frac{\pi-\psi}{2},
\end{align*}
hence
\begin{equation}\label{remA}
\frac{1}{2\pi}\int_0^\infty
\log\left|\frac{1+t^\nu}{1-t^\nu}\right|
\frac{dt}{t}=\frac{\psi}{2}.
\end{equation}
\end{rem}

\section{Dense sets in $\D_s$ and $\H_\psi$} \label{density}

In this section we establish some results concerning density and approximations in our spaces. 

\subsection{Dense subsets of $\D_s$ and some applications}

Let $\mathcal{R}(\C_+)$ be the linear span of $\{r_\l : \l \in \C_+\}$, and $\wt{\mathcal{R}}(\C_+)$ be the sum of $\mathcal{R}(\C_+)$ and the constant functions.  
Using Example \ref{hexs}(1) and Lemma \ref{Dh1}, we have
\begin{equation} \label{rinh}
\mathcal{\wt{R}}(\C_+) \subset \H_{\pi/2} \embedi \D_s, \quad s>-1.
\end{equation}
 
\begin{thm} \label{D00}
The space
\[
\wt{\mathcal{R}}(\C_+):=\left\{a_0 + \sum_{k=1}^n a_k(\lambda_k+z)^{-1}: n \in \mathbb N, \; a_k \in \C, \; \lambda_k\in \C_{+}\right\}
\]
is dense in $\mathcal{D}_s$ for each $s>-1$.
\end{thm}

\begin{proof}
Let $\mathcal{R}_{\D_s}(\C_+)$ be the closure of $\mathcal{R}(\C_+)$ in $\D_{s}$, 
and let $f \in \D_s$.   It follows from Example \ref{resd} and Remark \ref{remdr} (or a direct estimate) that the function
\[
(R_f)(\l) := - \frac{2^s}{\pi} f'(\l) r_{\overline{\l}}
\]
is continuous from $\C_+$ to $\D_s$, and it is Bochner integrable with respect to area measure $S$ on $\C_+$.  Since point evaluations are continuous on $\D_s$ (Remark \ref{remdr}), it follows from Corollary \ref{Repr} that 
\[
Q_s f' = \int_{\C_+} R_f(\l) \,dS(\l)
\]
as the Bochner integral of a continuous function.  Hence $Q_s f'$ belongs to the closure in $\D_s$ of the linear span of the range of the integrand, which is contained in $\mathcal{R}_{\D_s}(\C_+)$.   Now $f = f(\infty) + Q_sf'$ which is in the closure of $\wt{\mathcal{R}}(\C_+)$ in $\D_s$.
\end{proof}

From Proposition \ref{Furt}, we have the continuous inclusion
\[
\D_s \embedi \D_\sigma \quad \text{if $\sigma>s>-1$},
\]
and from \eqref{rinh}, Theorem \ref{ACM} and Proposition \ref{BDs}, we have
\[
\wt{\mathcal{R}}(\C_+) \subset H^{1,1}(\C_+) \embedi \lt L^1 + \C \subset \LT \embedi \Bes \embedi \D_s^\infty \embedi \D_s \quad \text{if $s>0$}.
\]
Here $\lt L^1+\C$ is the sum of $\lt L^1$ and the constant functions, and it is a closed subspace of $\LT$.  The following density results hold.

\begin{cor}  \label{Bdens}
\begin{enumerate}[\rm1.]
\item If $\sigma>s>-1$, then $\D_s$ is dense in $\D_\sigma$.
\item For $s>-1$, the spaces $H^{1,1}(\C_+)$ and $\D_s^\infty$ are dense in $\D_s$.
\item For $s>0$, the spaces $\mathcal{L}L^1+\C$,  $\mathcal{LM}$ and $\mathcal{B}$ are dense in $\mathcal{D}_s$.
\item For $s>0$, the spaces $\wt{\mathcal{R}}(\C_+)$, $H^{1,1}(\C_+)$, $\LL + \C$, $\LT$ and $\Bes$ are not dense in $\D_s^\infty$.
\end{enumerate}
\end{cor}

\begin{proof}
The first three statements are immediate from Theorem \ref{D00}.

Since any function in $\Bes$ extends continuously to $i\R$, the same holds for the closure of $\Bes$ in $\mathcal{D}_s^\infty$ when $s>0$.  The function
$f(z)=e^{-1/z}\in \mathcal{D}_s^\infty$ for $s>0$,
(see Remarks \ref{e-z} and Example \ref{dm} with $\nu=0$),  but $f$ is not continuous at  $z=0$.  This establishes the final statement.
\end{proof}

The function $g(z) = \exp(\arccot z)$ considered in Example \ref{eac} provides another example of a function from $\mathcal{D}_0^\infty$ which is discontinuous on $i\R$, and so does not belong to the closure of $\Bes$ in $\D^\infty_s$ for $s>0$. 

In order to obtain operator norm-estimates for functions $f^{(n)}$ applied to semigroup generators (see Theorem \ref{AnalF}), we will need a stronger version of Corollary \ref{D00n} on differentiability of $t \to f(t\cdot)$ in the $D_s$-norm.   
We first prove a lemma, and we present the stronger statement in Corollary \ref{AlgDer}.

\begin{lemma}\label{raF}
Let $\lambda\in \C_{+}$, $\tau>0$, $t \in (\tau/2,2\tau)$, and define
\[
g_{t,\tau,\l}(z):=\frac{r_\l(tz)-r_\l(\tau z)}{t-\tau}-zr_\l'(\tau z),\qquad z\in \C_{+}.
\]
Then, for every $s>-1$, 
\begin{equation}\label{Nder}
\lim_{t\to\tau}\,\|g_{t,\tau,\l}\|_{\mathcal{D}_s}=0.
\end{equation}
\end{lemma}

\begin{proof}
We have
\begin{align*}
g_{t,\tau,\l}'(z)&=\left(-\frac{t}{(tz+\lambda)^2}+
\frac{\tau}{(\tau z+\lambda)^2}\right)\frac{1}{t-\tau}
+\frac{1}{(\tau z+\lambda)^2}-\frac{2\tau z}{(\tau z+\lambda)^3}\\
&= \frac{\tau t z^2 - \l^2}{(tz+\l)^2(\tau z+\l)^2} + \frac{\l-\tau z}{(\tau z+\l)^3}  \to0,  \qquad t \to \tau.
\end{align*}
Hence 
\[
|g_{t,\tau,\l}'(z)|\le \frac{C_{\tau,\l}}{1+|z|^2}, \qquad t \in (\tau/2,2\tau), \, z \in \C_+,
\]
for some $C_{\tau,\l}$.   Since
\[
\int_{-\pi/2}^{\pi/2}  \cos^s\psi\int_0^\infty \frac{d\rho}{1+\rho^2}\,d\psi<\infty
\]
for any $s>-1$, the dominated convergence theorem implies \eqref{Nder}.
\end{proof}

\begin{cor}\label{DcorA}
Let $f\in \mathcal{D}_s$, $s>-1$.  For  $\tau>0$, let $zf'_\tau$ denote the function mapping $z$ to $zf'(\tau z)$.  Then
\[
\lim_{t\to \tau}\left\|
\frac{f_t-f_\tau}{t-\tau}-zf'_\tau\right\|_{\mathcal{D}_{s+1}}=0.
\]
\end{cor}

\begin{proof}
Let $\tau>0$ be fixed, and
\[
(R_{t,\tau} f)(z) := \frac{f(t z)-f(\tau z)}{t-\tau}-zf'_\tau, \qquad f \in\D_s, \, t>\tau/2.
\]
By Lemma \ref{D001}, $\{R_{t,\tau} : t> \tau/2\}$ is a bounded subset of $L(\D_s,\D_{s+1})$.   By Lemma \ref{raF},
\[
\lim_{t\to\tau} \|R_{t,\tau}r_\l\|_{\D_{s+1}} = 0, \qquad \l \in \C_+.
\]
Since the linear span of the functions $r_\l$ and the constants is dense in $\mathcal{D}_s$ (see Theorem \ref{D00}), the assertion follows.  
 \end{proof}

\begin{cor}\label{AlgDer}
Let $f\in\mathcal{D}_s$, $s>-1$, and let
\[
G(t)(z) : = f(tz), \qquad F_n(t)(z) := z^nf^{(n)}(tz),  \quad n\in \N, \,t>0, \, z \in \C_+. 
\]
Then $G$ and $F_n$ map $(0,\infty)$ into $\D_{s+n}$,  $G$ is $n$-times differentiable as a function from $(0,\infty)$ to $\D_{s+n}$, and
\[
F_n = G^{(n)}.
\]
\end{cor}

\begin{proof}
Firstly, $f(tz) \in \D_s \subset \D_{s+n}$, so $G$ maps $(0,\infty)$ into $\D_{s+n}$.

The proof is by induction on $n$.   The case $n=1$ is given by Corollary \ref{DcorA}.   Assume that $G^{(n)} = F_n$ with values in $\D_{s+n}$, and let $f_n(z) = z^n f^{(n)}(z)$. Then 
\[
G^{(n)}(t)(z) =  F_n(t)(z) = t^{-n}f_n(tz).
\]
By Corollary \ref{DcorA} applied to $f_n\in\D_{s+n}$, $G^{(n)}$ is differentiable with respect to $t$, when considered as a function with values in $\D_{s+n+1}$.  Finally,  
\[
G^{(n+1)}(t)(z) = \frac{d}{dt} (z^n f^{(n)}(tz)) = z^{n+1} f^{(n+1)}(tz) = F_{n+1}(t)(z).  \qedhere
\]
\end{proof}

\subsection{Approximations via change of variables}
Here we consider approximations of $f$ from $\D_s$ and $\H_\psi$ by the functions $f_\gamma(z) = f(z^\gamma)$ as $\gamma \to1-$.

\begin{prop} \label{gggP}
Let $\gamma \in (0,1)$.  The following hold.
\begin{enumerate}[\rm1.]
\item
Let $s>-1$, and $f \in \D_s$.  Then
\begin{equation} \label{Dc1}
\|f_\gamma\|_{\D_s} \le \|f\|_{\D_s},
\end{equation}
and
\begin{equation} \label{Dc2}
\lim_{\gamma\to1-} \|f_\gamma-f\|_{\D_s} = 0.
\end{equation}
\item
Let $\psi\in (0,\pi)$and $g\in \H_{\psi}$. Then
\begin{equation} \label{Sal1}
\lim_{\gamma\to1-} \|g_\gamma -g\|_{\H_\psi}= 0.
\end{equation}
\end{enumerate}
\end{prop}

\begin{proof}
1.
First let $g \in H^{1,1}(\C_{+})$.  If $0<\psi<\varphi\le \pi/2$, then it follows from Theorem \ref{hardy0}(iv) that
\begin{equation}\label{Afs}
\int_0^\infty \left(|g'(te^{i\varphi})|+|g'(te^{-i\varphi})|\right)\,dt
\ge
\int_0^\infty \left(|g'(te^{i\psi})|+|g'(t e^{-i\psi})|\right)\,dt.
\end{equation}
Hence
\begin{align*}
\|g\|_{\D_s} &= |g(\infty)|+
\int_0^{\pi/2}\cos^s\varphi \int_0^\infty
\left(|g'(te^{i\varphi})|+|g(te^{-i\varphi})|\right)\,dt\,d\varphi \\
&\ge |g(\infty)|+
\int_0^{\pi/2}\cos^s\varphi \int_0^\infty
\left(|g'(te^{i\gamma\varphi})|+|g'(te^{-i\gamma\varphi})|\right)\,dt\,d\varphi \\
&= \|g_\gamma\|_{\D_s}.
\end{align*}
Since $\H_{\pi/2}$ is dense in $\D_s$ (see Corollary \ref{Bdens}), it follows that the map $g \mapsto g_\gamma$ extends to a contraction on $\D_s$, and this contraction maps $f$ to $f_\gamma$.  Then \eqref{Dc1} holds.

Now \eqref{Dc2} follows from \eqref{Sal1}, \eqref{Dc1}, and the fact that $\H_{\pi/2}$ is continuously and densely embedded in $\D_s$ (see Proposition \ref{Dh1} and Corollary \ref{Bdens}).

2.  Since the norms $\|\cdot\|_{\H_\psi}$ and $\|\cdot\|'_{\H_\psi}$ are equivalent and $g_\gamma(\infty) = g(\infty)$, it suffices to show that 
\[
\|g'_\gamma - g'\|_{H^1(\Sigma_\psi)} = \int_{\partial\Sigma_\psi} |g_\gamma'(z)-g(z)| \,|dz| \to 0.
\]
By Lemma \ref{simple}, we have
\[
\int_{\partial\Sigma_\psi}|g'_\gamma(z)|\,|dz|
=\int_{\partial\Sigma_{\gamma\psi}}|g'(z)|\,|dz|.
\]
Applying Theorem \ref{hardy0}(iii) to $g'$,
\[
\lim_{\gamma\to 1-} \int_{\partial\Sigma_{\gamma\psi}}|g'(z)|\,|dz|  
=\int_{\partial\Sigma_{\psi}}|g'(z)|\,|dz| 
\]
and
\[
\lim_{\gamma\to1-} g'_\gamma(z) = g'(z), \quad \text{for almost all $z \in\partial \Sigma_\psi$}.
\]
Now the statement \eqref{Sal1} follows from Lemma \ref{duren21}.
\end{proof}

\subsection{Density of rational functions in $\H_\psi$}

In Theorem \ref{D00} and Corollary \ref{Bdens}, we established that $\wt{\mathcal{R}}(\C_+)$ and several larger spaces are dense in $\D_s$, for $s>-1$ or $s>0$.  In particular, we noted that $H^{1,1}(\C_+)$ is dense in $\D_s$.

Let $\psi\in (0,\pi)$ and  
\[
\H_{\psi,0}=\{f\in \H_\psi:\,f(\infty)=0\},
\]
with the norm
\[
\|f\|_{\H_{\psi,0}}=\|f'\|_{H^1(\Sigma_\psi)}.
\]
By \eqref{bound_for_sup}, this norm is equivalent to $\|\cdot\|_{\H_\psi}$ on $\H_{\psi,0}$.

Let $\psi \in (0,\pi)$ and $\mathcal{R}(\Sigma_\psi)$ be the linear span of $\{r_\lambda :\lambda \in \Sigma_{\pi-\psi}\}$.  Let $\mathcal{R}_\H(\Sigma_\psi)$ be the closure of $\mathcal{R}(\Sigma_\psi)$ in $\H_{\psi,0}$.
 We will prove that $\H_{\psi,0} = \mathcal R_{\mathcal H}(\Sigma_\psi)$.  Thus the rational functions which vanish at infinity and have simple poles outside $\overline\Sigma_\psi$ are dense in $\H_\psi$ modulo constants.
This fact may be known, but we did not find it in the literature.   Our proof involves several lemmas given below and it may be of interest as a piece of function theory.  The following lemma, relating to the function spaces $\mathcal E_\varphi$ from Definition \eqref{edef}, is the key step in our proof.

\begin{lemma}\label{DH111}
Let $\psi\in (0,\pi)$, $\varphi\in (\psi,\pi)$
and let $f\in H^1(\Sigma_{\varphi})$. 
If
\[
\int_{\partial\Sigma_{\varphi}}\frac{|f(\lambda)|}{|\lambda|}\,|d\lambda|<\infty,
\]
then 
\begin{equation}\label{GGG1}
f\in \mathcal{R}_\H(\Sigma_\psi).
\end{equation}
\end{lemma}

\begin{proof}
From \eqref{ff} for $\gamma=1$, the function $F(\lambda) := r_\l$ maps $\Sigma_{\pi-\psi}$ into $\H_{\psi,0}$, and is locally bounded.  Moreover $F$ is holomorphic (see Theorem \ref{hardy2}(iii) and Section \ref{prelims}) and its derivative is $-r_\l^2$.

The Cauchy integral formula \eqref{Boch1} may be written as
\begin{equation*}
f(z)=
\frac{1}{2\pi i}
\int_{\partial\Sigma_{\varphi}} f(\l) F(-\l)(z) \,d\lambda, \quad
z\in \Sigma_\psi.
\end{equation*}
From \eqref{res_sigma_cal_intro}, we obtain
\[
\int_{\partial\Sigma_\varphi} \|f(\l)F(-\l)\|_{\H_{\psi,0}} \,|d\l| \le \frac{1}{2\pi \sin^2 ((\varphi-\psi)/2)
} \int_{\partial\Sigma_\varphi} \frac{|f(\l)|}{|\l|}\,d\l < \infty.
\]
Thus 
\[
f = \frac{1}{2\pi i} \int_{\partial\Sigma_\varphi} f(\l)F(-\l) \,d\l
\]
as a Bochner integral in $\H_{\psi,0}$, with continuous integrand, so it may be approximated in the $\H_{\psi,0}$-norm by Riemann sums of the integrand, which lie in $\mathcal{R}(\Sigma_\psi)$. 
Hence $f\in \mathcal{R}_\H(\Sigma_\psi)$.
\end{proof}

The next step in the proof is to construct a family of functions which serve as an approximate identity for $\mathcal H_{\varphi,0}$ when restricted to any sector smaller than $\Sigma_\varphi$.

\begin{lemma}\label{Add}
Let $\varphi \in (0,\pi)$ and $\epsilon\in (0,1)$, and let 
\begin{equation}\label{Deps}																	
g_\epsilon(z):= \frac{2z^\epsilon}{1+z^\epsilon}(1+\epsilon z)^{-2},\qquad
z\in \C\setminus (-\infty,0].
\end{equation}
Then $g_\epsilon\in \H_{\varphi,0} \cap H^1(\Sigma_\varphi)$ and   
\begin{equation}\label{DR1}
 \sup_{\epsilon \in (0,1)}
\|g_\epsilon'\|_{H^1(\Sigma_\varphi)} <\infty.
\end{equation}
Moreover, 
\begin{equation}\label{DR2}
\lim_{\epsilon\to 0}\,g_\epsilon(z)=1,\quad z\in \C\setminus (-\infty,0],
\end{equation}
and, for $0<a<b<\infty$, there exists $C_{\varphi,a,b}$ such that 
\begin{equation}\label{DR3}
|g'_\epsilon(z)|\le C_{\varphi,a,b}\cdot \epsilon, \quad z\in\partial\Sigma_\varphi,\,
|z|\in (a,b).
\end{equation}
\end{lemma}

\begin{proof}
It is clear that $g_\ep \in H^1(\Sigma_\varphi)$ and $g_\ep(\infty)=0$.  
Moreover
\begin{equation*}
g'_\epsilon(z)=\frac{2\epsilon z^{\epsilon}}{z(1+z^\epsilon)^2(1+\epsilon z)^2}-
\frac{4\epsilon z^{\epsilon}}{(1+z^\epsilon)(1+\epsilon z)^3}.
\end{equation*}
Applying Lemma \ref{trig}, there is a constant $C_\varphi$ 
such that, for $z \in \partial\Sigma_\varphi$ and $t=|z|$,
\begin{equation} \label{DerD}
|g'_\epsilon(z)|\le C_\varphi \epsilon t^\epsilon
\left(\frac{2}{t(1+t^\epsilon)^2(1+\epsilon t)^2}+
\frac{4}{(1+t^\epsilon)(1+\epsilon t)^3}\right).
\end{equation}
Hence
\begin{align*}
\|g'_\epsilon\|_{H^1(\Sigma_\varphi)} &\le 2\epsilon C_\varphi\int_0^\infty
\frac{t^\epsilon\,dt}{t(1+t^\epsilon)^2(1+\epsilon t)^2} + 4\epsilon C_\varphi\int_0^\infty
\frac{t^{\epsilon}\,dt}{(1+t^\epsilon)(1+\epsilon t)^3} \\ 
&\le 2 \epsilon C_\varphi \int_1^\infty \frac{dt}{(1+\epsilon t)^2} +  2\epsilon C_\varphi \int_0^1 t^{\epsilon-1}\,dt + 4 \epsilon C_\varphi \int_0^\infty \frac{dt}{(1+\epsilon t)^3} \\
&= 6 C_\varphi.
\end{align*}
  This yields \eqref{DR1}.

The property \eqref{DR2} is straightforward, and \eqref{DR3} follows from \eqref{DerD}.
\end{proof}

Lemmas \ref{DH111} and \ref{Add} enable us to show that any function $f \in \H_{\varphi,0}$, when restricted to $\Sigma_\psi, \, \psi \in (0,\varphi)$, can be approximated by rational functions (with simple poles) in $\H_{\psi,0}$.

\begin{lemma}\label{DM1}
Let $\psi \in (0,\pi)$, $\varphi\in (\psi,\pi)$, and let
$
f\in \H_{\varphi,0}.
$
Then 
\begin{equation}\label{dens}
f\in \mathcal{R}_\H(\Sigma_\psi).
\end{equation}
\end{lemma}

\begin{proof}
Assume first that $f\in H_{\varphi,0}$ and $f(0)=0$, and let $g_\epsilon$ be defined by
\eqref{Deps}.
Then
$f g_\epsilon \in 
\H_{\varphi,0} \cap H^1(\Sigma_{\varphi})$,
and
\[
\int_{\partial \Sigma_{\varphi}} \frac{|f(z)g_\epsilon(z)|}{|z|}\,|dz|<\infty.
\]
By Lemma \ref{DH111}, $fg_\ep \in \mathcal{R}_\H(\Sigma_\psi)$.  

Note that
\[
\|f(1-g_\epsilon)\|_{\H_{\varphi,0}} \le
\|f'(1-g_\epsilon)\|_{H^1(\Sigma_{\varphi})}+
\|f g'_\epsilon\|_{H^1(\Sigma_{\varphi})}.
\]
We will prove that
\begin{equation}\label{Dlim}
\lim_{\epsilon\to 0}\,\|f'(1-g_\epsilon)\|_{H^1(\Sigma_\varphi)}=0,
\end{equation}
and
\begin{equation}\label{Dlim1}
\lim_{\epsilon\to 0}\,\|fg'_\epsilon\|_{H^1(\Sigma_{\varphi})}=0.
\end{equation}

By \eqref{bound_for_sup} and \eqref{DR1},
\[
\sup_{\epsilon\in(0,1) } \|g_\epsilon\|_{H^\infty(\Sigma_\varphi)}\le \sup_{\epsilon \in (0,1)} 
\|g'_\epsilon\|_{H^1(\Sigma_\varphi)}<\infty,
\]
so, using \eqref{DR2} and
the dominated convergence theorem, we have
\[
\|f'(1-g_\epsilon)\|_{H^1(\Sigma_{\varphi})}=
\int_{\partial\Sigma_\varphi} |f'(\lambda)(1-g_\epsilon(\lambda))|\,|d\lambda|
\to 0,\quad \epsilon\to 0.
\]
For $0<a<b<\infty$, 
\begin{align*}
\|f g'_\epsilon\|_{H^1(\Sigma_{\varphi})}
&=\int_{\partial\Sigma_{\varphi}}|f(\lambda)||g_\epsilon'(\lambda)|\,|d\lambda|\\
&\le \sup_{|z|>b,\,z\in\partial\Sigma_{\varphi}}\,|f(z)|\,
\int_{|\lambda|>b, \,\lambda\in \partial\Sigma_{\varphi}} |g'_\epsilon(\lambda)|\,|d\lambda|\\
&\null\hskip20pt +\sup_{|z|<a,\,z\in\partial\Sigma_{\varphi}}\,|f(z)|\,
\int_{|\lambda|<a, \,\lambda\in \partial\Sigma_{\varphi}} |g'_\epsilon(\lambda)|\,|d\lambda|\\
&\null\hskip20pt +\|f\|_{H^{\infty}(\Sigma_{\varphi})}
\int_{|\lambda|\in (a,b), \,\lambda\in \partial\Sigma_{\varphi}} 
|g'_\epsilon(\lambda)|\,|d\lambda|\\
&\le \left(\sup_{|z|>b,\,z\in\partial\Sigma_{\varphi}}\,|f(z)|
+\sup_{|z|<a,\,z\in\partial\Sigma_{\varphi}}\,|f(z)|\right)
\|g'_\epsilon\|_{H^1(\Sigma_{\varphi})}\\
&\null\hskip20pt +2(b-a)
\|f\|_{H^{\infty}(\Sigma_{\varphi})}
\sup_{|z|\in(a,b),\,z\in \partial\Sigma_{\varphi}}\,|g'_\epsilon(z)|.
\end{align*}
Letting first $\epsilon\to0$, using \eqref{DR1}-\eqref{DR3} along with Vitali's theorem, and then letting  $a\to0$ and $b \to \infty$, using $f(0)=f(\infty)=0$, we obtain
\[
\lim_{\epsilon\to 0}\|f g'_\epsilon\|_{H^1(\Sigma_{\varphi})}=0.
\]
We have now proved the assertions \eqref{Dlim} and \eqref{Dlim1}.   Thus we  obtain  \eqref{dens} under the additional assumption that $f(0)=0$.

Now let $f\in \H_{\varphi,0}$ be arbitrary.  Then consider
\[
f_0(z):=f(z)-2f(0)\left(\frac{1}{z+1}-\frac{1}{z+2}\right),
\]
and note that
\[
f_0 \in \H_{\varphi,0} \cap H^1(\Sigma_{\varphi}),\qquad f_0(0)=0.
\]
Then, by the above, 
\[
f_0\in \mathcal{R}_\H(\Sigma_\varphi),
\]
and hence \eqref{dens} holds.
\end{proof}

We now approximate functions $f \in \H_{\psi,0}$ by functions from $\H_{\psi',0}, \, \psi'>\psi$, using the change of variables from Proposition \ref{gggP}.

For $f\in \H_{\psi,0}$ and $f_\gamma(z) = f(z^\gamma)$ for $\gamma \in (0,1)$, we now have
\[
f_\gamma\in \H_{\varphi_\gamma,0},\qquad \varphi_\gamma:=\min\{\gamma^{-1}\psi,\pi\}>\psi.
\]
By \eqref{Sal1},
\[
\lim_{\gamma\to 1-}\,\|f-f_\gamma\|_{\H_{\psi,0}}=0.
\]
By Lemma \ref{DM1}, $f_\gamma \in \mathcal{R}_\H(\Sigma_\psi)$.   Thus we obtain the following result that the rational functions with simple poles are dense in $\H_{\psi,0}$. 

\begin{thm}\label{DM122}
Let $\psi\in (0,\pi)$. Then
\begin{equation}\label{GGG122}
\H_{\psi,0}=\mathcal{R}_\H(\Sigma_\psi).
\end{equation}
\end{thm}

\section{Convergence Lemmas}

In this section we formulate Convergence Lemmas for functions in $\D_s$ and $\H_\psi$, composed with fractional powers.

\begin{lemma}\label{48St2}
Let $s>-1$ and $(f_k)_{k=1}^\infty\subset \mathcal{D}_s$ be such that
\begin{equation*}\label{Bound50A}
\sup_{k \ge 1}\|f_k\|_{\mathcal{D}_s}<\infty,
\end{equation*}
and for every $z\in \C_{+}$ there exists
\[
f_0(z):=\lim_{k\to\infty}\,f_k(z).
\]
Let $g\in \mathcal{D}_s$  satisfy
\[
g(0)=g(\infty)=0.
\]
For $\gamma\in (0,1)$,  let
\begin{equation}\label{fgamma}
f_{k,\gamma}(z) = f_k(z^\gamma), \quad g_\gamma(z) = g(z^\gamma), \qquad z \in \C_+.
\end{equation}
Then
\begin{equation}\label{Se50A}
\lim_{k\to\infty}\,\|f_{k,\gamma}g_\gamma\|_{\mathcal{D}_s}=0.
\end{equation}
\end{lemma}

\begin{proof}
By Corollary \ref{Fatou}, $f_0 \in \mathcal D_s$.  So without loss of generality,
we can assume that $f_0=0$.
By Corollary \ref{gggC}, $g_{\gamma}$ and $f_{k, \gamma}$ belong to the algebra $\mathcal{D}_s^\infty$, so $f_{k, \gamma}g_{\gamma}\in \mathcal{D}_s$.  Moreover
\[
B:=\sup_{k \ge 1}\left(\|f_{k, \gamma}\|_{\mathcal{D}_s}+\|f_{k, \gamma}\|_{\infty}\right)< \infty.
\]

Let $0 < r < R < \infty$ and $\Omega_{r,R} = \{z \in \C_+ : r \le |z| \le R\}$.   By Vitali's theorem,
\[
\lim_{k\to\infty}\,\left(|f_{k, \gamma}(z)|+|f_{k, \gamma}'(z)|\right)=0
\]
uniformly on $\Omega_{r,R}$.    Therefore the integrals
\[
\int_{\Omega_{r,R}} \frac{(\Re z)^s}{|z|^{s+1}} |f'_{k,\gamma}(z) g_\gamma(z)| \,dS(z), \quad \int_{\Omega_{r,R}} \frac{(\Re z)^s}{|z|^{s+1}} |f_{k,\gamma}(z) g'_\gamma(z)| \,dS(z),
\]
tend to $0$ as $k\ \to \infty$.   Moreover,
\[
\int_{\C_+ \setminus \Omega_{r,R}} \frac{(\Re z)^s}{|z|^{s+1}} |f'_{k,\gamma}(z) g_\gamma(z)| \,dS(z) \le \sup_{z \in \C_+ \setminus \Omega_{r,R}} |g_\gamma(z)| \, \|f_{k,\gamma}\|_{\D_s},
\]
and
\[
\int_{\C_+ \setminus \Omega_{r,R}} \frac{(\Re z)^s}{|z|^{s+1}} |f_{k,\gamma}(z) g'_\gamma(z)| \,dS(z) \le \|f_{k,\gamma}\|_\infty \int_{\C_+ \setminus\Omega_{r,R}} \frac{(\Re z)^s}{|z|^{s+1}} |g'_\gamma(z)| \,dS(z).
\]
Hence
\[
\limsup_{k\to\infty} \| f_{k,\gamma}g_\gamma\|_{\D_s} \le B \left( \sup_{\C_+ \setminus \Omega_{r,R}} |g_\gamma(z)|  + \int_{\C_+ \setminus \Omega_{r,R}} \frac{(\Re z)^s}{|z|^{s+1}} |g'_\gamma(z)| \,dS(z) \right).
\]
Letting $r\to0$ and $R \to \infty$, we obtain the assertion \eqref{Se50A}.
\end{proof}

The following result is a convergence lemma for $\H_\psi$, analogous to Lemma \ref{48St2}.

\begin{lemma}\label{H11cl}
Let $\psi\in(0,\pi)$ and $(f_k)_{k=1}^\infty\subset \H_\psi$ be such that
\begin{equation*}\label{Bound50B}
\sup_{k \ge 1}\|f_k\|_{\H_\psi}<\infty,
\end{equation*}
and for every  $z\in \C_{+}$ there exists
\[
f_0(z):=\lim_{k\to\infty}\,f_k(z).
\]
Let $g\in \H_\psi$  satisfy
\[
g(0)=g(\infty)=0.
\]
For $\gamma\in (0,1)$ and $k \in \mathbb N,$  let $f_{k,\gamma}$ and $g_\gamma$ be given by \eqref{fgamma}.
Then
\begin{equation}\label{Se50B}
\lim_{k\to\infty}\,\|f_{k,\gamma}g_\gamma\|_{\H_\psi}=0.
\end{equation}
\end{lemma}

\begin{proof}
The proof is similar to Lemma \ref{Se50A}.

By Lemma \ref{FatouH}, $f_0 \in \mathcal H_\psi$. Thus, we will assume that $f_0=0$.
Let $\gamma \in (\psi/\pi,1)$.  By Lemma \ref{simple},
$g_{\gamma}$ and $f_{k, \gamma}$ belong to the algebra $\H_{\psi/\gamma} \subset \H_\psi$, so $f_{k, \gamma}g_{\gamma}\in \H_\psi$.  Moreover
\[
B:=\sup_{k \ge 1}\left(\|f_{k, \gamma}\|_{\H_\psi}+\|f_{k, \gamma}\|_{H^\infty(\Sigma_\psi)}\right)< \infty.
\]

Let $0 < r < R < \infty$ and $I_{r,R} = \{z \in \partial\Sigma_{\psi}: r \le |z| \le R\}$.   By Vitali's theorem,
\[
\lim_{k\to\infty}\,\left(|f_{k, \gamma}(z)|+|f_{k, \gamma}'(z)|\right)=0
\]
uniformly on $I_{r,R}$.    Therefore the integrals
\[
\int_{I_{r,R}} |f'_{k,\gamma}(z) g_\gamma(z)| \,|dz|, \quad \int_{I_{r,R}} |f_{k,\gamma}(z) g'_\gamma(z)| \,|dz|,
\]
tend to $0$ as $k\ \to \infty$.   Moreover,
\[
\int_{\partial\Sigma_\psi \setminus I_{r,R}} |f'_{k,\gamma}(z) g_\gamma(z)| \,|dz| \le \sup_{z \in \partial\Sigma_\psi \setminus I_{r,R}} |g_\gamma(z)| \, \|f_{k,\gamma}\|_{\H_\psi},
\]
and
\[
\int_{\partial\Sigma_\psi \setminus I_{r,R}} |f_{k,\gamma}(z) g'_\gamma(z)| \,|dz| \le \|f_{k,\gamma}\|_{H^\infty(\Sigma_\psi)} \int_{\partial\Sigma_\psi \setminus I_{r,R}} |g'_\gamma(z)| \,|dz|.
\]
Hence
\[
\limsup_{k\to\infty} \| f_{k,\gamma}g_\gamma\|_{\H_\psi} \le B \left( \sup_{\partial\Sigma_\psi \setminus I_{r,R}} |g_\gamma(z)|  + \int_{\partial\Sigma_\psi \setminus I_{r,R}} |g'_\gamma(z)| \,|dz| \right).
\]
Letting $r\to0$ and $R \to \infty$, we obtain \eqref{Se50B}.
\end{proof}

\section{The $\mathcal D$-calculus and its compatibility}\label{defD}

Here we discuss functional calculus for sectorial operators $A$ of angle less than $\pi/2$ and functions $f \in \D_\infty$.  Since $f$ is bounded on a closed sector containing the spectrum of $A$ (Corollary \ref{Repr}), $f(A)$ may be considered via the extended holomorphic (sectorial) calculus.   If $A$ is injective then $f(A)$ can be defined that way as a closed operator, but we will show that $f(A)$ is a bounded operator when $f \in \D_\infty$.   Our methods provide estimates for $\|f(A)\|$, and we will adapt the results in Section \ref{hardy_sobolev} to take account of the angle of sectoriality, by using fractional powers of operators (cf.\ Corollary \ref{gggC}).

Recall that a densely defined operator $A$ on a Banach space $X$ is {\it sectorial of angle} $\theta \in [0,\pi/2)$ if $\sigma(A) \subset \overline \Sigma_\theta$ and, for each $\varphi \in (\theta,\pi]$,
\begin{equation} \label{Aa0}
 M_{\varphi}(A) := \sup \left\{\|z(z+A)^{-1}\| : z \in \Sigma_{\pi-\varphi}\right\} < \infty.
\end{equation}
The {\it sectorial angle} $\theta_A$ of $A$ is the minimal such $\theta$.  Note that $M_\varphi(A)$ is a decreasing function of $\varphi$.

Let $\Sect(\theta)$ stand for the class of all sectorial operators of angle $\theta$ for $\theta \in [0,\pi/2)$ on Banach spaces, and denote $\Sect(\pi/2-) := \bigcup_{\theta \in [0,\pi/2)} \Sect(\theta)$.  Then $A \in \Sect(\pi/2-)$ if and only if  $-A$ generates a (sectorially) bounded holomorphic $C_0$-semigroup on $X$ of angle $(\pi/2)-\theta_A$, in the sense that the semigroup has a holomorphic extension to $\Sigma_{(\pi/2)-\theta_A}$ which is bounded on each smaller subsector.   Note that these semigroups are sometimes called {\it sectorially} bounded holomorphic semigroups in the literature. However, in this paper,  we will adopt the convention that bounded holomorphic semigroups are bounded on sectors.   We will denote the semigroup as $(e^{-tA})_{t\ge0}$, and $e^{-tA}$ then agrees with $e_t(A)$ defined in the Hille-Phillips calculus, where $e_t(z) = e^{-tz}$.  One may consult \cite{HaaseB} for the general theory of sectorial operators, and \cite[Section 3.7]{ABHN} for the theory of holomorphic semigroups. 

Let $A$ be a closed, densely defined operator on a Banach space $X$ such that
\begin{equation} \label{Aa}
\sigma(A) \subset \overline\C_+ \qquad \text{and} \qquad M_A:=  M_{\pi/2}(A) = \sup_{z\in \C_{+}}\|z(z+A)^{-1}\|<\infty.
\end{equation}
Then $\|A(z+A)^{-1}\| \le M_A+1, \,z\in \C_+$, and Neumann series (see \cite[Lemma 1.1]{V1}) imply that $\sigma(A) \subset \Sigma_\theta \cup \{0\}$ and
\begin{equation}\label{Aaa1}
\|z(z+A)^{-1}\|\le 2M_A,\qquad z\in \Sigma_{\pi-\theta},
\end{equation}
where
\[
\theta:=\arccos(1/(2M_A))<\pi/2.
\]
So $A \in \Sect(\theta) \subset \Sect(\pi/2-)$.   Conversely, if $A \in \Sect(\theta)$ where $\theta \in [0,\pi/2)$, then \eqref{Aa} holds.   Thus $-A$ generates a bounded holomorphic semigroup if and only if
(\ref{Aa}) holds.  The constant $M_A$ is a basic quantity associated with $A$,
and we call it the {\it sectoriality constant} of $A$.   Note that $M_{tA} = M_A$ for all $t>0$.

A set $S$ of sectorial operators on the same Banach space $X$ is {\it uniformly sectorial} of angle $\theta$ if $S \subset \Sect(\theta)$ and, for each $\varphi \in (\theta,\pi)$, there exists $C_{\varphi}$ such that $M_\varphi(A) \le C_\varphi$ for all $A \in S$.  Thus $S$ is uniformly sectorial of some angle $\theta<\pi/2$ if and only if each $A \in S$ satisfies \eqref{Aa} and $\sup_{A \in S} M_A < \infty$.

In the presentation of the $\D$-calculus that follows, we assume that the reader is familiar with the holomorphic functional calculus for sectorial operators, as in \cite{HaaseB}, and in particular with the Hille-Phillips calculus for negative generators of bounded $C_0$-semigroups.   We will make extensive use of fractional powers of sectorial operators in the form $(A+z)^{-\gamma}$ where $\gamma >0$.  If $\gamma$ is not an integer, these operators are fractional powers which can be defined in many ways (see \cite{MS}), including using the holomorphic functional calculus (see \cite[Chapter 3]{HaaseB}).   All these approaches are consistent with each other.  
Since $\D_\infty = \bigcup_{n=0}^\infty \D_n$, it is possible to define the $\D$-calculus without using fractional powers, and this would simplify some proofs (for example, Lemma \ref{fractional} becomes trivial, and the formulas \eqref{Sti} and \eqref{s1} would not be needed).  Thus we could define the $\D$-calculus without using fractional powers, and in particular we could define the fractional powers $(z+A)^{-\gamma}$ for all $\gamma > 0$.  
This definition would be consistent with other definitions (see Theorem \ref{Compatible}). Then we could define the $\D_s$-calculus for all $s>-1$ in the way described below, using fractional powers in \eqref{formulaD}.  

The following simple lemma for fractional powers is a version of the moment inequality applied to the sectorial operator $(A+z)^{-1}$.

\begin{lemma}\label{fractional}
Let $A \in \operatorname{Sect}(\pi/2-)$, and $\gamma>0$. Let $\lceil\gamma\rceil$ be the ceiling function of $\gamma$, i.e.\ the smallest integer in $[\gamma,\infty)$.  Then
\[
\|(A+z)^{-\gamma}\|\le M_A^{\lceil\gamma\rceil}/|z|^\gamma,\quad z\in \C_{+}.
\]
\end{lemma}

\begin{proof}
When $\gamma \in \N$, the estimate is trivial.   

Let $\gamma \in (0,1)$. By the compatibility of our calculus with the holomorphic functional calculus for fractional powers, we may use the following standard Stiletjes formula (see \cite[(3.52)]{ABHN}, for example):
\begin{equation}\label{Sti}
(A+z)^{-\gamma}=\frac{\sin(\pi\gamma)}{\pi}\int_0^\infty t^{-\gamma}(A+z+t)^{-1}\,dt,\quad z\in \C_{+}.
\end{equation}
Next, let $z=\rho e^{i\varphi}\in \C_{+}$. Then, using Cauchy's theorem,
\begin{align*}
(A+z)^{-\gamma}&=\frac{\sin(\pi\gamma)}{\pi}\int_0^{e^{i\varphi \infty}} w^{-\gamma}(A+\rho e^{i\varphi}+w)^{-1}\,dw\\
&=e^{i(1-\gamma )\varphi}\frac{\sin(\pi\gamma)}{\pi}\int_0^\infty s^{-\gamma}(A+\rho e^{i\varphi}+s e^{i\varphi})^{-1}\,ds.
\end{align*}
So,
\[
\|(A+z)^{-\gamma}\|\le M_A\frac{\sin(\pi\gamma)}{\pi}
\int_0^\infty s^{-\gamma}(\rho+s)^{-1} \,ds = \frac{M_A}{\rho^\gamma}. 
\]

In other cases, $\gamma = (\lceil\gamma\rceil - 1) + \delta$ where $\delta \in (0,1)$, and the estimate follows from the two previous cases.
\end{proof}

Now let $s >-1$ be fixed and let $f\in \mathcal{D}_s$.
We define
\begin{equation}\label{formulaD}
f_{\mathcal{D}_s}(A):=f(\infty)-
\frac{2^s}{\pi}\int_0^\infty \alpha^s\int_{-\infty}^\infty f'(\alpha+i\beta)(A+\alpha-i\beta)^{-(s+1)}\,d\beta\,d\alpha.
\end{equation}
Note that when $s = 1$ and $f \in \Bes$, \eqref{formulaD} coincides with the definition of $f(A)$ as given by
the $\Bes$-calculus in \cite{BGT}, cf.\ \eqref{fcdef}.

This definition is valid as the following simple proposition shows.

\begin{prop}\label{bounded_s}
Let $A \in \Sect(\pi/2-)$, and $s>-1$. 
\begin{enumerate}[\rm1.]
\item The map $f \mapsto f_{\mathcal{D}_s}(A)$ is bounded from $\D_s$ to $L(X)$.
\item For $f \in \D_s$,
\begin{equation}
\lim_{\ep\to0+}\,f_{\mathcal D_s}(A+\ep) =f_{\mathcal D_s}(A),  \label{epsd11}
\end{equation}
in the operator norm topology.
\end{enumerate}
\end{prop}

\begin{proof}
Lemma \ref{fractional} and \eqref{formulaD} imply that
\begin{equation}\label{estimateD}
\|f_{\mathcal{D}_s}(A)\|\le |f(\infty)|+
\frac{2^s M_A^{\lceil{s+1}\rceil}}{\pi}\|f'\|_{\mathcal{V}_s}.
\end{equation}
Thus the boundedness of the map $f \mapsto f_{\mathcal{D}_s}(A)$ from $\D_s$ to $L(X)$ follows.

By a standard Laplace transform representation for negative fractional powers \cite[(3.56)]{ABHN},
\begin{equation} \label{s1}
(A+z)^{-(s+1)} = \frac{1}{\Gamma(s+1)} \int_0^\infty t^s e^{-tz} e^{-tA} \,dt, \qquad z \in \C_+.
\end{equation}
By the dominated convergence theorem, $z \mapsto (A+z)^{-(s+1)}$ is continuous (even holomorphic) on $\C_+$ in the operator norm topology.

The operators $(A+\epsilon)_{\epsilon\ge0}$ are uniformly sectorial of angle $\theta$, more precisely,  $M_{A+\ep} \le M_A$ \cite[Proposition 2.1.1 f)]{HaaseB}.  By Lemma \ref{fractional}, this implies that
\begin{equation}\label{UnEE}
\|(A+\epsilon+z)^{-(s+1)}\|\le M_A^{\lceil{s+1}\rceil} |z|^{-(s+1)}, \qquad  \epsilon\ge 0, \quad z\in \mathbb C_{+}.
\end{equation}
Now \eqref{epsd11} follows from applying \eqref{formulaD} to $A+\epsilon$, letting $\epsilon\to0+$, and using the dominated convergence theorem.
\end{proof}

\begin{rem}  \label{A+e}
The property \eqref{epsd11} can be compared with Corollary \ref{SemDa} where a direct proof is given that the shifts form a bounded holomorphic $C_0$-semigroup on $\D_s$.  To deduce \eqref{epsd11}, one also needs that if $f_\ep(z) = f(z+\ep)$, then $(f_\ep)_{\D_s}(A) = f_{\D_s}(A+\ep)$.  By Theorem \ref{D00}, it suffices that this holds for $f = r_\l, \, \l\in \C_+$, i.e., to show that $(r_\l)_{\D_s}(A) = (\l+A)^{-1}$.   We show this in Theorem \ref{dcalculus}, but the argument uses \eqref{epsd11}.
\end{rem}

Let $f_{\text{HP}}(A)$ stand for a function $f$ of $A$ defined by the Hille-Phillips functional calculus when $f$ is in the Hille-Phillips algebra $\LT$, and let $f_{\operatorname{Hol}}(A)$ denote a function $f$ of $A$ given by the holomorphic functional calculus when $f$ is in the domain of that calculus.   The following statement shows that both calculi are compatible with the $\mathcal D$-calculus,
and moreover that the definitions of $f_{\D_s}(A)$ agree for the various values of $s$ for which $f \in \D_s$.

\begin{thm}\label{Compatible}
Let $A \in \Sect(\pi/2-)$, and let $f \in \D_s,\, s>-1$.
\begin{enumerate}[\rm(i)]
\item If $A$ is injective, then
\begin{equation}\label{Ds11}
f_{\mathcal{D}_s}(A)=f_{\mathrm{Hol}}(A).
\end{equation}
\item  
If $\sigma\ge s$ then
\begin{equation}\label{Ds12}
f_{\mathcal{D}_\sigma}(A)=f_{\mathcal{D}_s}(A).
\end{equation}
\item If $f \in \mathcal{LM} \cap \mathcal{D}_s$, then
\begin{equation}\label{Ds13}
f_{\mathcal{D}_s}(A)=f_{{\rm HP}}(A).
\end{equation}
In particular, \eqref{Ds13} holds if $f \in \mathcal{LM}$ and $s>0$.
\end{enumerate}
\end{thm}

\begin{proof} 
We start by proving \eqref{Ds11}. Assume that $A$ is injective, and $A \in \Sect(\theta)$, where $\theta \in (0, \pi/2)$.  Let $\psi \in (\theta, \pi/2)$.   Let $f \in \D_s, \, s>-1$, and assume (without loss of generality) that $f(\infty)=0$.  By the definition of the holomorphic functional calculus,
\[
A(1+A)^{-2}f_{\operatorname{Hol}}(A)=\frac{1}{2\pi i}\int_{\partial \Sigma_\psi} \frac{\l f(\l)}{(\l+1)^2}(\l-A)^{-1}\,d\l.
\]
Since
\[
f(\l)=-\frac{2^s}{\pi}\int_0^\infty \alpha^s\int_{-\infty}^\infty\frac{f'(\alpha+i\beta)}
{(\l+\alpha-i\beta)^{s+1}}\,d\beta\,d\alpha,\quad \l\in \C_{+},
\]
Fubini's theorem and Cauchy's theorem imply that
\begin{align*}
\lefteqn{A(1+A)^{-2}f_{\operatorname{Hol}}(A) \hskip10pt} \\
&=- \frac{2^s}{\pi} \int_0^\infty \a^s \int_{-\infty}^\infty  f'(\a+i\b) \left(\frac{1}{2\pi i}
\int_{\partial \Sigma_\psi}\,\frac{\l(\l-A)^{-1}}{(\l+1)^2(\l+\alpha-i\beta)^{s+1}}\,d\l\right)
d\beta \, d\alpha\\ 
&= -\frac{2^s}{\pi}A(1+A)^{-2} \int_0^\infty\alpha^s \int_{-\infty}^\infty f'(\alpha+i\beta)(A+\alpha-i\beta)^{-(s+1)}
d\beta \, d\alpha\\
&=A(1+A)^{-2}f_{\mathcal{D}_s}(A).
\end{align*}
Hence (\ref{Ds11}) holds.

Now we no longer assume that $A$ is injective.  We infer by \eqref{Ds11} that
\[
f_{\mathcal{D}_s}(A+\epsilon)=f_{\mathcal{D}_\sigma}(A+\epsilon)
\]
 for all $\epsilon>0$ and $\sigma \ge s$.   
Letting $\epsilon\to 0$ and using Proposition \ref{bounded_s}, we obtain the assertion (\ref{Ds12}).

Finally, if $f\in \mathcal{LM} \cap \mathcal D_s$ for some $s > -1$, then $f\in \mathcal{B}\cap \mathcal D_s$, and using \eqref{Ds12}, we have
\[
f_{\text{HP}}(A)=\Phi_A(f)=f_{\mathcal{D}_1}(A)=f_{\mathcal{D}_s}(A). \qedhere
\]
\end{proof}

\begin{rem}\label{coincid}
If $f$ has zero polynomial limits at zero and at infinity in the sense of \cite[p.27]{HaaseB}, then the proof above does not require the regularisation factor $\lambda(\lambda+1)^{-2}$.  Hence $f_{\mathcal D}(A)=f_{\mathrm{Hol}}(A)$ regardless of the injectivity of $A$.
One can show that $f_{\mathcal D}(A)=f_{\mathrm{Hol}}(A)$  even when $f$ belongs to the extended Riesz-Dunford function class, (for example, $f(z)=e^{-tz},\, t >0$),
but we omit a discussion of this here (cf.\ the proof of Lemma \ref{resap}).
\end{rem}

Recall that $\mathcal D_s \subset \mathcal D_\sigma$ if $-1< s \le \sigma$, and the space 
\[
\mathcal D_\infty:=\bigcup_{s>-1} \mathcal D_s
\] 
is an algebra, by Lemma \ref{Alg1}.   Thus it is a plausible and natural domain for a functional calculus, which we now define.

\begin{thm}\label{dcalculus}
Let $A \in \Sect(\pi/2-)$. The formula \eqref{formulaD} defines an algebra homomorphism:
\begin{eqnarray*}
\Psi_A : \mathcal D_{\infty} \mapsto L(X), \qquad \Psi_A(f)=f_{\D_s}(A),  \qquad f \in \D_s, \, s>-1.
\end{eqnarray*}
Moreover, 
\begin{enumerate}[\rm(i)]
\item $\Psi_A(r_\lambda) = (\lambda+A)^{-1}, \quad \lambda \in \C_+$.
\item $\Psi_A$ is bounded in the sense that for every $s>-1$ there exists $C_s(A)$ such that
\begin{equation}\label{bounded}
\|\Psi_A(f)\|\le |f(\infty)| + C_s(A)\|f\|_{\mathcal D_s}, \qquad f \in \D_s.
\end{equation}
Specifically, \eqref{estimateD} holds.
\end{enumerate}
Moreover, $\Psi_A$ is the unique algebra homomorphism from $\D_\infty$ to $L(X)$ which satisfies {\rm(i)} and {\rm(ii)}.
\end{thm}

The homomorphism $\Psi_A$ will be called the {\it $\mathcal D$-calculus}.  

\begin{proof}
It follows from \eqref{Ds12} that $\Psi_A$  is well-defined by \eqref{formulaD}, and from \eqref{Ds11} that $(r_\l)_{\D_s}(A+\epsilon) = (\l+\epsilon+A)^{-1}$ for $\ep>0$.  Letting $\epsilon \to 0$ and using \eqref{epsd11} gives (i).  Moreover \eqref{bounded} is a direct consequence of Proposition \ref{bounded_s}.

We will now prove that $\Psi_A$ is a homomorphism. Let $f,g\in \mathcal D_\infty$. Then  $f\in \mathcal D_r$ and $g\in \mathcal D_t$ for some strictly positive $s$ and $t$, hence
$f g\in\mathcal  D_{s+t+1}$ by Lemma \ref{Alg1}.  
Since $\mathcal {LM}$ is dense in $\mathcal D_s$ for every $s>0$ by Corollary \ref{Bdens},
there exist $(f_n)_{n=1}^\infty$ and $(g_n)_{n=1}^\infty$ from $\mathcal {LM}$
such that
\[
f_n \to f \quad \text{in} \quad   \mathcal D_s \qquad \text{and} \qquad
g_n \to g \quad  \text{in}\quad   \mathcal D_t, \qquad n \to \infty,
\]
and then, in view of Lemma \ref{Alg1},
$f_n g_n \to f g$ in $\mathcal D_{s+t+1}$ as $n \to \infty$.

By the product rule for the HP-calculus and \eqref{Ds13},
\[
\Psi_A(f_n g_n)=(f_n g_n)_{\mathrm{HP}}(A)=\Psi_A(f_n)\Psi_A(g_n), \qquad n \ge 1.
\]
Passing to the limit when $n\to\infty$ and using \eqref{bounded}, it follows that
\[
\Psi_A(fg)=\Psi_A(f)\Psi_A(g).
\]

Let $\Psi : \D_\infty \to L(X)$ be an algebra homomorphism satisfying (i) and (ii).  Then $\Psi$ and $\Psi_A$ coincide on $\{r_\lambda : \lambda \in \C_+\}$.  Since $A$ is densely defined, $\Psi(1)=1$ (see \cite[Section 6, Observation (2)]{BGT2}, so $\Psi$ and $\Psi_A$ coincide on the span of these functions which is dense in 
$\D_s$ (Theorem \ref{D00}).   Since $\Psi$ and $\Psi_A$ are both bounded on $\D_s$, it follows that they coincide on each $\D_s$ and hence on $\D_\infty$.
\end{proof}

\begin{remark}  \label{Dsn}
If $A$ is an operator for which a $\D$-calculus exists with the properties (i) and (ii) given in Theorem \ref{dcalculus}, then $A\in\Sect(\pi/2-)$.  This follows from \eqref{r_l} and the properties (i) and (ii).  By combining this with Theorem \ref{dcalculus}, we obtain Theorem \ref{dcalculus_intro}.  
Note also that, if (i) holds for some $\l\in\C_+$, then it holds for all $\l\in\C_+$, by the resolvent identity.
\end{remark}

The Banach algebras $\D_s^\infty, \, s>-1$, are subalgebras of $\D_\infty$, so we obtain the following corollary by restricting the $\D$-calculus.

\begin{cor}\label{Compat}
Let $A \in \Sect(\pi/2-)$ and $s>-1$.  Then there exists a bounded algebra homomorphism $\Psi_A^s: \mathcal D^\infty_s \mapsto L(X)$ such that
\begin{enumerate}[\rm(i)]
\item $\Psi_A^s(r_\l)=(\l+A)^{-1}, \quad \l \in \mathbb C_+$.
\item $\Psi^s_A$ is bounded in the $\D_s$-norm, i.e., there exists $C_s(A)$ such that
\begin{equation*}
\|\Psi_A(f)\|\le C_s(A) \|f\|_{\mathcal D_s}, \qquad f \in \D_s.
\end{equation*}
\end{enumerate}
Moreover $\Psi_A^s$ is the unique algebra homomorphism from $\D_s^\infty$ to $L(X)$ which satisfies {\rm(i)} and {\rm(ii)}.
\end{cor}

From now onwards, we will write $f_\D(A)$ instead of $\Psi_A(f)$, for $f \in \D_\infty$.   
When $f(z) = e^{-tz}$, we will continue to use the notation $e^{-tA}$ for $f(A)$, since the $\mathcal D$-calculus agrees with the HP-calculus by Theorem \ref{Compatible}.  So $(e^{-tA})_{t\ge 0}$ form the $C_0$-semigroup generated by $-A$, and it extends to a bounded holomorphic semigroup.
 
Let $g : \C_+ \to \C_+$ be a holomorphic function and assume that $r_\l \circ g \in \D_\infty$ for all $\l \in \C_+$.   Since the functions $(r_\l)_{\l\in\C_+}$ satisfy the resolvent identity, the operators 
\[
((r_\l \circ g)_\D(A))_{\l\in\C_+} \subset L(X)
\]
also satisfy the resolvent identity, i.e., they form a pseudo-resolvent.   In particular their kernels and their ranges are independent of $\l$, and they form the resolvent of an operator $B$ if and only if the common kernel is $\{0\}$, and the domain of $B$ is the common range of the pseudo-resolvent (see \cite[Section VIII.4]{Yos}).

\begin{cor} \label{dcomp}
Let $A, B \in \Sect(\pi/2-)$,
and let $g : \C_+ \to \C_+$ be holomorphic.  Assume that, for each $s>-1$, there exists $\sigma>-1$ such that 
\begin{enumerate}[\rm(a)]
\item For all $f \in \D_s$, $f\circ g \in \D_\sigma$, and
\item For all $\l\in\C_+$, $(r_\l \circ g)_\D(A) = (\l+B)^{-1}$.
\end{enumerate}
Then  $(f \circ g)_\D(A) = f_\D(B)$ for all $f \in \D_\infty$.
\end{cor}

\begin{proof}
By assumption (a), Corollary \ref{Cangle} and the Closed Graph Theorem, $f \mapsto f \circ g$ is a bounded map from $\D_s$ to $\D_\sigma$.   Moreover the  $\D$-calculus for $A$ is a bounded map from $\D_\sigma$ to $L(X)$.   Hence the composition is a bounded map from $\D_s$ to $L(X)$, and by assumption (b) it sends $r_\l$ to $(\l+B)^{-1}$ for all $\l \in \C_+$.  Moreover the maps collectively form an algebra homomorphism from $\D_\infty$ to $L(X)$.   By the uniqueness in Theorem \ref{dcalculus}, this map is the $\D$-calculus for $B$.
\end{proof}

In the context of Corollary \ref{dcomp},  the operator  $B$  is sometimes written as $g(A)$, but the precise meaning depends on the specific situation.

There is also a version of Corollary \ref{dcomp} for fixed values of $s$ and $\sigma$, using the $\D_s^\infty$- and $\D_\sigma^\infty$-calculi.

\begin{exas} \label{gcomp}
Examples of functions $g$ and operators $B$ which satisfy the conditions of Corollary \ref{dcomp} include the following.
\begin{enumerate}[\rm1.]
\item $g(z) = z^{-1}$, if $A$ is injective (with dense range); $\sigma=s$, $B=A^{-1}$.    Then $f \circ g$ is the function $\tilde f \in \D_s$ as in Lemma \ref{leminv}.  Note that $\tilde f_\D(A)$ is defined as a bounded operator on $X$, even if $A$ is not injective.   If $A_0$ is the restriction of $A$ to $X_0$, the closure of the range of $A$, then $\tilde f_\D(A)$ acts as $A_0^{-1}$ on $X_0$ and as the sectorial limit $f(\infty)$ on the kernel of $A$.  If $X$ is reflexive this determines $\tilde f_\D(A)$ on $X$.
\item $g(z) = tz$, where $t>0$; $\sigma=s$, $B = tA$.   See Lemma \ref{leminv}.
\item $g(z) = z + \eta$, where $\eta \in \C_+$; $\sigma=s$, $B = A +\eta$.  See Remark \ref{A+e}.
\item $g(z) = z^\gamma$, where $\gamma \in (0,1)$; $s>\-1,\sigma>-1$; $B = A^\gamma$ (as defined in the holomorphic functional calculus).   See Corollary \ref{gggC} for the assumption (a) in Corollary \ref{dcomp}, and Corollary \ref{fpD} below for the assumption (b). The result also holds for $\gamma \in (1, \pi/(2\theta))$.
\item Examples (2), (3) for $\eta \in \R_+,$ and (4) above are Bernstein functions.  By Lemma \ref{bsds}, $r_\l \circ g \in \D_s$ for all Bernstein functions $g$ and $s>2$.   We will show in the proof of Theorem \ref{BF1} that $(r_\l\circ g)_D(A) = (\l+g(A))^{-1}$, where $g(A)$ is a sectorial operator.   
\end{enumerate}
\end{exas}

In Example \ref{gcomp}(4) above, we have introduced a fractional power $A^\gamma$, where $\gamma>0$.   These operators are defined in various ways, including the extended holomorphic functional calculus.  
To justify the example, we need the following lemma about fractional powers, which is probably known at least in simpler form.  For $\gamma \in (0,1)$ and $\nu\in\N$, it follows easily from a standard result \cite[Proposition 5.1.4]{MS}.   
We give a proof that uses the holomorphic functional calculus for fractional powers as in \cite[Section 3.1]{HaaseB}.

\begin{lemma} \label{resap}
Let $A \in \Sect(\pi/2-)$, $\gamma \in (0,1)$, and $\nu>0$.  In the operator norm topology,
\[
\lim_{\epsilon\to0+} ((A+\epsilon)^\gamma + z)^{-\nu} = (A^\gamma+z)^{-\nu}, \quad z\in \C_+.
\]
\end{lemma}

\begin{proof}
Let $z \in \C_+$ be fixed, and let
\[
f(\l) = \left(\l^\gamma + z\right)^{-\nu}, \qquad \l\in \overline{\C}_+.
\]
Then $f \in H^\infty(\C_+)$, and, by considering the derivative of $\mu \mapsto (\mu+z)^{-\nu}$, we see that there exists a constant $C$ (depending on $z$) such that
\begin{equation*}\label{f2}
|f(\l)-f(0)|  \le C |\l|^{\gamma}, \quad 
|f(\l)| \le C |\l|^{-\gamma\nu}, \qquad \l\in\C_+.
\end{equation*}
Thus $f$ has polynomial limits at $0$ and $\infty$, and so $f$ belongs to the extended Riesz-Dunford class defined in \cite[Lemma 2.2.3]{HaaseB}.   In other words,
\[ 
g_0(\l) := f(\l) - z^{-\nu}(1+\l)^{-1}
\]
has polynomial decay at $0$ and $\infty$. 
Moreover, there exists a constant $C'$ (independent of $\ep$) such that
\begin{equation}\label{f3}
|f(\l+\ep)-f(\ep)|  \le C' |\l|^{\gamma}, \quad |f(\l+\ep)| \le C' |\l|^{-\gamma\nu}, \qquad \l\in\C_+, \, \ep \in(0,1).
\end{equation}
Let
\[
g_\ep(\l) = f(\l+\ep) - f(\ep)(1+\l)^{-1}.
\]
Using the definition of the primary functional calculus \cite[Section 2.3.1]{HaaseB}, 
we have 
\begin{align*}
f(A+\ep)-f(A)&= g_\ep(A)-g_0(A) + (f(\ep) -f(0))(I+A)^{-1}, \\
g_\ep(A)-g_0(A)&=\frac{1}{2\pi i} \int_{\Sigma_\psi}  \left(g_\ep(\l)-g_0(\l)\right)
 (\l-A)^{-1} \,d\l,
\end{align*}
where $\psi \in (\theta,\pi/2)$.   By the dominated convergence theorem,
\[
\lim_{\ep \to 0+} \|(g_\ep(A) - g_0(A)\|=0.
\]
The pointwise convergence of $g_\ep -g_0$ to zero is clear, and the existence of an integrable majorant follows easily from \eqref{f3}. 
 \end{proof}

\begin{cor} \label{fpD}
Let $A\in\Sect(\pi/2-)$, and $\gamma \in (0,1)$.  Let $f \in \D_\infty$ and $f_\gamma(z) =  f(z^\gamma)$.  Then
\begin{enumerate}[\rm1.]
\item
In operator norm,
\[
\lim_{\ep\to0+} f_{\D}((A+\ep)^\gamma) = f_{\D}(A^\gamma).
\]
\item
 $(f_\gamma)_{\D}(A) = f_{\D}(A^\gamma)$.
\end{enumerate}
\end{cor}

\begin{proof}
The proof of (1) follows from Lemma \ref{resap} in essentially the same way as the last paragraph of the proof of Proposition \ref{bounded_s}.

By Corollary \ref{gggC}, $h \in \D_\infty$.  By \eqref{Ds11} and the Composition Theorem for the holomorphic functional calculus \cite[Theorem 2.4.2]{HaaseB}, we have $h_{\D}(A+\epsilon) = f_{\D}((A+\epsilon)^\gamma)$.   Letting $\epsilon\to0+$ and using \eqref{epsd11} and (1), we obtain (2).
\end{proof}

\section{The calculus on Hardy-Sobolev algebras}\label{hardy_sobolev}

Given the negative generator $A$ of a bounded holomorphic $C_0$-semigroup on a Banach space $X$, the $\mathcal D$-calculus allows us to extend the $\mathcal B$-calculus to a much larger class of functions. A drawback of the $\mathcal D$-calculus is that it does not respect
the sectoriality angle of $A$, so the results within the $\mathcal D$-calculus are independent of the sectoriality angle and confined to holomorphic functions on $\C_+$.    To remedy that problem, we introduce in this section a version of the $\mathcal D$-calculus adjusted to an appropriate Hardy-Sobolev algebra on a sector in the right half-plane.  While the Hardy-Sobolev algebra has a ``stronger'' norm, it appears to be an adequate substitute for $\mathcal D_\infty$ in the setting of sectors, and it has significant applications, as we will see in Section \ref{norm_estimates}.  

The basic idea is a very simple change of variable in the $\mathcal D$-calculus.   If $\Psi^s_A$ is the $\mathcal D_s^\infty$-calculus for a sectorial operator  $A$, then one sets 
$\Upsilon_A(f):= \Psi^s_{A^\gamma}(f_{1/\gamma})$ for appropriate values of  $\gamma$, determined by the sectoriality angle $\theta_A$ of $A$.  This definition does not depend on the precise choice of $\gamma$, by Corollary \ref{fpD}.  The definitions also agree for different $s>-1$ by Theorem \ref{Compatible}(ii), and we set $s=0$ for convenience.  As we show below this eventually leads to a new calculus for Hardy-Sobolev algebras on sectors.

Throughout this section, we assume that
  $A\in\Sect(\theta)$, $0<\theta<\psi<\pi$, and we let
$\gamma:=\pi/(2\psi)$, and, as before,
\[
M_\psi(A):=\sup_{z\in \Sigma_{\pi-\psi}}\,\|z(A+z)^{-1}\|.
\]
Then $A^\gamma\in \Sect(\pi/2-)$, and $\theta_{A^\gamma} =\gamma\theta_A \in (\theta_A/2,\pi/2)$ \cite[Proposition 3.1.2]{HaaseB}.   We are particularly interested in cases where $\psi$ is close to $\theta$, so that $\H_\psi$ is as large as possible.

\subsection{The operator $f(A)$ for $f\in \H_\psi$}

Recall from Lemmas \ref{simple} and \ref{Dh1}(i) that if $f\in \H_\psi$ and $f_{1/\gamma}(z)=f(z^{1/\gamma}),$ then
\[
f'_{1/\gamma}\in H^1(\C_{+}),\qquad  f_{1/\gamma}(\infty)=f(\infty),
\]
and consequently $f_{1/\gamma}\in \mathcal{D}_0$.    Together with Proposition \ref{FrepH} this motivates the following definition of the operator $f_{\mathcal H}(A)$ by means of the $\mathcal{D}$-calculus applied to $A^{\gamma}$:  
\begin{equation}\label{sigma_def}
f_{\mathcal H}(A):=f_{1/\gamma}(\infty)-\frac{1}{\pi}\int_0^\infty\int_{-\infty}^\infty f'_{1/\gamma}(\a+i\b)
(A^{\gamma}+\a-i\b)^{-1}\,d\b\,d\a.
\end{equation}
The right-hand side of \eqref{sigma_def} converges in the uniform operator topology, and, by \eqref{estimateD} and \eqref{embedh},
\begin{equation}\label{A2Dop}
\|f_{\mathcal H}(A)\|\le |f(\infty)|+\frac{M_{A^\gamma}}{\pi}\|f_{1/\gamma}\|_{\mathcal{D}_0,0}
\le  |f(\infty)|+ M_{A^\gamma} \|f'\|_{H^1(\psi)} \le M_{A^\gamma} \|f\|'_{\H_\psi}.
\end{equation}

If $f\in \H_\psi$ and $A$ is injective, then $f_{\operatorname{Hol}}(A)$ can also be defined using the holomorphic functional calculus and the Composition Rule within it \cite[Theorem 2.4.2]{HaaseB}:
\[
f_{\operatorname{Hol}}(A)=f_{1/\gamma,\operatorname{Hol}}(A^\gamma) :=\frac{1}{2\pi i}[A^\gamma(1+A^{\gamma})^{-2}]^{-1}
\int_{\partial \Sigma_\omega}\,\frac{\l f(\l^{1/\gamma})}{(\l+1)^2}(\l-A^\gamma)^{-1}\,d\l,
\]
for $0<\theta<\omega<\psi$.   The following proposition shows that our definition \eqref{sigma_def} of $f_{\mathcal H}(A)$ coincides with $f_{\operatorname{Hol}}(A)$, when $A$ is injective, and various other properties  
are easily deduced from the definition above and corresponding properties of the $\D$-calculus.

\begin{prop} \label{Hprops}
Let $f \in \H_\psi$, and $A \in \Sect(\theta)$, where $0 \le \theta < \psi < \pi$.
\begin{enumerate}[\rm(i)]
\item $f_\H(A)$ does not depend on the choice of $\psi$.
\item If $\nu \in (0,1)$ and $f_\nu(z) = f(z^\nu)$, then $(f_\nu)_\H(A) = f_\H(A^\nu)$.
\item If $f \in \D_\infty$ and $\theta < \pi/2$, then $f_\H(A) = f_\D(A)$.
\item If $A$ is injective, then $f_\H(A) = f_{\operatorname{Hol}}(A)$.
\item In the operator norm topology, $\lim_{\epsilon\to0+} f_\H(A+\epsilon) = f_\H(A)$. 
\end{enumerate}
\end{prop}

\begin{proof} Statements (i), (ii) and (iii) follow from Corollary \ref{fpD}(ii) and Lemma \ref{simple}.   Statement (iv) follows from (iii) and Theorem \ref{Compatible}(i).    Statement (v) follows from Proposition \ref{bounded_s}.

An alternative direct proof of (iv) can be given as follows.
We may assume that $f(\infty)=0$.    Since $f_{1/\gamma}\in \mathcal{D}_0$, Corollary \ref{Repr} gives
\[
f_{1/\gamma}(z)= -\frac{1}{\pi}\int_0^\infty \int_{-\infty}^\infty\frac{f_{1/\gamma}'(\a+i\b)}{z+\a-i\b}\,d\b\,d\a,\qquad z\in \C_{+}.
\]
Using Fubini's theorem and some basic properties of the holomorphic functional calculus,
we obtain
\begin{align*}
&\lefteqn{\null\hskip-5pt A^\gamma(1+A^\gamma)^{-2}f_{\Hol}(A)}\\
&=-\frac{1}{\pi}\int_0^\infty \int_{-\infty}^\infty f_{1/\gamma}'(\a+i\b)
\left(\frac{1}{2\pi i}
\int_{\partial \Sigma_\omega}\,\frac{\l(\l-A^\gamma)^{-1}}{(\l+1)^2(\l+\a-i\b)}\,d\l\right)
d\b \,d\a\\
&= -A^\gamma(1+A^\gamma)^{-2}\left(
\frac{1}{\pi}\int_0^\infty \int_{-\infty}^\infty f_{1/\gamma}'(\a+i\b)(A^\gamma+\a-i\b)^{-1}
d\b\, d\a\right)\\
&=A^\gamma(1+A^\gamma)^{-2}f_{\mathcal{H}}(A),
\end{align*}
and (iv) follows.
\end{proof}

Now we can formally define the $\mathcal H$-calculus.

\begin{thm}\label{sigma_cal}
Let $A \in \Sect(\theta)$, where $\theta \in (0,\pi/2)$.   For any $\psi \in  (\theta, \pi)$ the formula \eqref{sigma_def}
defines a bounded algebra homomorphism:
\begin{eqnarray*}
\Upsilon_A : \mathcal \H_\psi \mapsto L(X), \qquad \Upsilon_A(f)=f_\H(A).
\end{eqnarray*}
The homomorphism $\Upsilon_A$ satisfies
 $\Upsilon_A(r_\l)=(\l+A)^{-1}$ for all $\l \in \Sigma_{\pi-\psi}$, and it is the unique homomorphism with these properties.
\end{thm}

The homomorphism $\Upsilon_A$ will be called the \emph{$\mathcal H$-calculus} for $A$.

\begin{proof}
The boundedness of $\Upsilon_A$ follows from either \eqref{NewEs} or \eqref{A2Dop}.
The homomorphism property is implied by Theorem \ref{Compat}.
Indeed, employing the functional calculus $\Psi^0_A$ on $\mathcal D^\infty_0$ given by Corollary \ref{Compat}, 
 one has 
\begin{align*}
\Upsilon_A(fg) &= \Psi^0_A((fg)_{1/\alpha}) = \Psi^0_A(f_{1/\alpha}g_{1/\alpha}) \\
 &= \Psi^0_A(f_{1/\alpha}) \Psi^0_A(g_{1/\alpha}) = \Upsilon_A(f)\Upsilon_A(g).
\end{align*}
The uniqueness follows from Theorem \ref{DM122}.
\end{proof}

\begin{remark}
If $A$ is any operator for which there is an $\H_\psi$-calculus as in Theorem \ref{sigma_cal}, then $A \in \Sect(\theta)$ for some $\theta\in(0,\psi)$.   This follows from \eqref{res_sigma_cal_intro}, and in combination with Theorem \ref{sigma_cal} this yields the proof of Theorem \ref{sigma_cal_intro}.   If $\Upsilon_A(r_\l) = (\l+A)^{-1}$ for some $\l\in \Sigma_{\pi-\psi}$ then this holds for all $\l \in \Sigma_{\pi-\psi}$, by the resolvent identity.
\end{remark}

\subsection{The operator $\arccot(A^\gamma)$ and the $\arccot$-formula}\label{alternative}

In this section, we derive an alternative to the formula \eqref{sigma_def} for the $\H$-calculus,  
in the form of an operator counterpart of  Proposition \ref{L7.4} for scalar functions.  In addition to its intrinsic interest, it helps us to compare our approach with the approach developed by Boyadzhiev \cite{Boy}, as we do at the end of this section.

We introduce as an operator kernel the function
\[
g(z) :=\arccot(z) = \frac{1}{2i} \log \left(\frac{z+i}{z-i}\right), \quad z \in \C_+,
\]
already considered in Example \ref{arc}.   Note that  $g \in \mathcal{D}_0$, $ g(\infty)=0$, $g'(z)=-(z^2+1)^{-1}$,
and \eqref{ArcN} holds:
\[
\arccot(z)=\frac{1}{\pi}
\int_0^\infty\int_{-\infty}^\infty\frac{(z+\a-i\b)^{-1}}{(\a+i\b)^2+1}\,d\b\,d\a\qquad z\in \C_{+}.
\]

Let $A \in \Sect(\theta)$, where $\theta \in [0,\pi)$.  Let $\psi \in (\theta,\pi)$ and $\gamma=\pi/(2\psi)$.
By the $\mathcal{D}$-calculus in \eqref{formulaD} and \eqref{estimateD},
\begin{equation}\label{defcot}
\arccot(A^\gamma) := \arccot_{\mathcal{D}}(A^\gamma)=\frac{1}{\pi}
\int_0^\infty\int_{-\infty}^\infty\frac{(A^\gamma+\a-i\b)^{-1}}{(\a+i\b)^2+1}\,d\b\,d\a,
\end{equation}
where the integral converges in the operator norm, and
\begin{equation}\label{NormA}
\|\arccot(A^\gamma)\|\le\frac{M_{A^\gamma}}{\pi}\|\arccot z\|_{\mathcal{D}_0} < 3M_{A^\gamma},\;\;
M_{A^\gamma}:=\sup_{z\in \C_{+}}\,\|z(A^\gamma+z)^{-1}\|.
\end{equation}

We will provide an alternative estimate for the operator $\arccot (A^\gamma/s^\gamma)$.   
Following (\ref{rew}), we may formally write
\begin{align} \label{acint}
\arccot_{{\rm int}}(A^{\gamma})&:=\frac{1}{4\pi}\int_0^\infty
\log\left|\frac{1+t^\gamma}{1-t^\gamma}\right|
\left((t-e^{i\psi}A)^{-1}
+(t-e^{-i\psi}A)^{-1}\right)\,dt\\
&\null\hskip30pt +\frac{1}{4 i}\int_{|\l|=1,\,\arg\l\in (\psi,2\pi-\psi)}\,
(\l-A)^{-1}\,d\l. \notag
\end{align}

\begin{lemma}\label{arc_bound}
Let $A \in \Sect(\theta)$ and $\gamma = \pi/(2\psi)$, where $0 \le \theta < \psi < \pi$.   The operator $\arccot_{{\rm int}}(A^{\gamma})$  is well-defined and
\begin{equation}\label{NesT}
\|\arccot_{{\rm int}}(A^{\gamma})\|\le M_\psi(A)\frac{\pi}{2}.
\end{equation}
\end{lemma}

\begin{proof}
Using (\ref{remA}), one notes that
\begin{align*}
\|\arccot(A^{\gamma})\|&\le
\frac{M_\psi(A)}{2\pi}\int_0^\infty
\log\left|\frac{1+t^\gamma}{1-t^\gamma}\right|
\,\frac{dt}{t}
+\frac{M_\psi(A)}{4}\int_{|\l|=1,\; \arg \l\in (\psi,2\pi-\psi)}\,\frac{|d\l|}{|\l|}\\
&=M_\psi(A)\left(\frac{\psi}{2}+
\frac{(\pi-\psi)}{2}\right)
=M_\psi(A)\frac{\pi}{2}. \qedhere
\end{align*}
\end{proof}

The next lemma shows that the formula $\arccot_{{\rm int}}(A^{\gamma})$ coincides 
with the definition of $\arccot (A^{\gamma})$ by the $\mathcal D$-calculus.  When $A$ is injective, $\arccot(A^\gamma)$ is defined in the holomorphic functional calculus by
\begin{equation}\label{archol}
\arccot_{\Hol}(A^\gamma):=\frac{1}{2\pi i}[A(1+A)^{-1}]^{-1}
\int_{\partial\Sigma_\omega}\,\frac{\l\arccot(\l^\gamma)}{\l+1}(\l-A)^{-1}\,d\l,
\end{equation}
where $0<\theta<\omega<\psi<\pi$.

\begin{lemma}\label{Ss2}
Under the assumptions above, $\arccot_{{\rm int}}(A^\gamma)=\arccot_{\mathcal{D}}(A^\gamma)$. If $A$ is injective, then
$\arccot_{{\rm int}}(A^{\gamma})=\arccot_{\Hol}(A^{\gamma})$.
\end{lemma}

\begin{proof} Assume first that $A$ is injective.  Using \eqref{archol}, (\ref{rew}) and \eqref{defcot}
we obtain
\begin{align*}
\lefteqn{A(1+A)^{-1} \arccot_{\rm{Hol}} (A^\gamma)} \\
&=  \frac{1}{2\pi i}
\int_{\partial\Sigma_\omega}\,\frac{z\arccot(z^\gamma)}{z+1}(z-A)^{-1}\,dz\\
&=\frac{1}{4\pi}\int_0^\infty \log\left|\frac{1+t^\gamma}{1-t^\gamma}\right|
\left(\frac{1}{2\pi i}
\int_{\partial\Sigma_\omega}\,\frac{z}{z+1}(z-A)^{-1}\,
V_\psi(z, t)\,dz \right)\,\frac{dt}{t}\\
&\hskip30pt + \frac{1}{4i}\int_{|\l|=1,\,\arg \l\in (\psi,2\pi-\psi)}\,\left(
\frac{1}{2\pi i}\int_{\partial\Sigma_\omega} \frac{z (z-A)^{-1}}{(z+1)(\l-z)}
dz \right) d \l \\
&=A(A+1)^{-1}\frac{1}{4\pi}\int_0^\infty \log\left|\frac{1+t^\gamma}{1-t^\gamma}\right|
\left((t-e^{i\psi}A)^{-1}+(t-e^{-i\psi}A)^{-1}\right)\,dt\\
&\null\hskip30pt +A(A+1)^{-1}\frac{1}{4i}\int_{|\l|=1,\,\arg \l\in (\psi,2\pi-\psi)}\,
(\l-A)^{-1}\,
d \l\\
&=A(A+1)^{-1}\arccot_{{\rm int}}(A^{\gamma}).
\end{align*}
Thus the second statement holds. Moreover, 
$\arccot_{\Hol}(A^\gamma)=\arccot_{\mathcal{D}}(A^\gamma)$, by Theorem \ref{Compatible}, and the first statement follows.

If $A$ is not injective, we have $\arccot_{\rm int}((A+\ep)^\gamma) = \arccot_\D((A+\ep)^\gamma)$ from the case above.  When $\ep\to0+$, the left-hand side converges in operator norm to $\arccot_{\rm int}(A^\gamma)$ by applying the dominated convergence theorem in \eqref{acint}, and the right-hand side converges to $\arccot_{\D}(A^\gamma)$ by Proposition \ref{bounded_s}.
\end{proof}

\begin{thm}\label{Sovp}
Let $A \in \Sect(\theta)$ and $\gamma=\pi/(2\psi)$, where $0 \le \theta < \psi < \pi$.  If $f \in \H_\psi$, and $f_\psi$ is given by \eqref{fpsi}, then
\begin{equation}\label{formula_a}
f_{\mathcal H}(A)= f(\infty)-\frac{2}{\pi}\int_0^\infty
f_\psi'(t) \,\arccot(A^\gamma/t^{\gamma})\,dt
\end{equation}
where the integral converges in the uniform operator topology, and
\begin{equation}\label{NewEs1}
\|f_{\mathcal H}(A)\|
 \le |f(\infty)| + \frac{M_\psi(A)}{2} \|f'\|_{H^1(\Sigma_\psi)} \le {M_\psi(A)} \|f\|_{\H_\psi}.
\end{equation} 
Moreover, if $f(\infty)=0$, then
\begin{equation}\label{NewEs2}
\|f_{\mathcal H}(A)\| \le \frac{M_\psi(A)}{2} \|f\|_{\H_\psi}.
\end{equation}
\end{thm}

\begin{proof}
Let $\varphi \in (\theta,\psi)$ be fixed.
Let $g(z)=\arccot(z)$ and $g_\gamma(z)=\arccot(z^\gamma)$.   Since $g \in \D_0 \subset \H_{\gamma \varphi}$, $g_\gamma \in \H_\varphi$.   Now
\begin{align}\label{arccotes}
\int_0^\infty |g'_\gamma(\rho e^{\pm i\varphi})|\,d\rho &=
 \gamma\int_0^\infty \rho^{\gamma-1}|g'(\rho^\gamma e^{\pm i\gamma\varphi})|\,d\rho \\
&=\int_0^\infty |g'(te^{\pm i\gamma\varphi})|\,dt
=\int_0^\infty \frac{dt}{|1+t^2e^{\pm 2i\gamma\varphi}|} \notag
\end{align}
and
\begin{align*}
\int_0^\infty \frac{d\rho}{|1+\rho^2e^{2i\psi}|} 
&= 2 \int_1^\infty \frac{d\rho}{((\rho^2-1)^2 + 4\rho^2\cos^2\psi)^{1/2}} \\
&\le 2\sqrt{2}\int_1^\infty \frac{d\rho}{(\rho+\cos\psi)^2-(1+\cos^2\psi)}\\
&=\frac{\sqrt{2}}{(1+\cos^2\psi)^{1/2}}
\log \left(\frac{1+\cos\psi+\sqrt{1+\cos^2\psi}}{1+\cos\psi-\sqrt{1+\cos^2\psi}}\right) \\
&\le \sqrt{2}
\log\left(\frac{6}{\cos\psi}\right).
\end{align*}
Let $q_t(z) = \arccot(z^\gamma/t^\gamma), \, t>0$.   By scale-invariance (Lemma \ref{simple}), $\|q_t\|_{\H_\varphi} = \|g_\gamma\|_{\H_\varphi}$. It follows from Lemma \ref{duren21} that $t \mapsto q_t$ is continuous from $(0,\infty)$ to $\H_\varphi$.
Hence, in view of Proposition \ref{L7.4}, we have
\begin{equation} \label{hbochner}
f=f(\infty)-\frac{2}{\pi}\int_0^\infty
f_\psi'(t)\, q_t \,dt,
\end{equation}
where the integral is understood as a Bochner integral in $\H_\varphi$.
By applying the bounded operator $\Upsilon_A$ to both sides of \eqref{hbochner},
we obtain (\ref{formula_a}).   The estimate \eqref{NewEs1} follows from \eqref{formula_a}, \eqref{NewEs2}, and Lemmas \ref{arc_bound} and \ref{Ss2}. 
\end{proof}

\begin{rems}\label{rem_h}
1.  If $M_\psi(A)=1$ in \eqref{NewEs1}, that is, if $-A$ generates a holomorphic $C_0$-semigroup which is contractive on $\Sigma_{(\pi/2)-\psi}$,  then the $\mathcal H$-calculus is contractive. This seems to be a new feature which has not been present in constructions of other calculi in the literature. 

\noindent 2. An alternative to the estimate \eqref{NewEs1} is 
\begin{equation}\label{NewEs}
\|f_{\mathcal H}(A)\|\le |f(\infty)|+ 3 M_{A^\gamma}\|f_\psi'\|_{L^1(\R_+)}.
\end{equation}
This is obtained from \eqref{formula_a}, using the estimates \eqref{estimateD} (with $s=0$) and \eqref{arccot_D} to obtain the estimate $\|\arccot(A^\gamma)\| \le 3\pi M_{A^\gamma}$.  The constant $3$ is not optimal.
It is possible to provide explicit bounds for $M_{A^\gamma}$ in terms of $M_A$. However we refrain from doing so in this paper, and we refer the interested reader to \cite[Propositions 5.1 and 5.2]{BGTad}. 
\end{rems}

Finally in this section, we discuss the relations between \cite{Boy} and the present work.
For $\psi \in (0,\pi)$, as in \cite{Boy}, let
\begin{align*}
k_\psi(t)
=\frac{1}{\pi^2}\log\left|\coth\left(\frac{\pi t }{4 \psi}\right)
\coth\left(\frac{\pi t}{4(\pi-\psi)}\right)\right|.
\end{align*}
Note that $k_\psi$ is an even function on $\mathbb R\setminus \{0\}$, and $\|k_\psi\|_{L^1(\mathbb R)}=1.$

For any $f \in L^\infty (\mathbb R_+),$ let 
\[
(f \circ k_\psi) (t):=\int_{0}^{\infty}f (s)k_\psi \left(\log (t/s) \right)\, \frac{ds}{s}, \qquad t > 0,
\]
and for $A \in \Sect(\theta)$, $\theta \in [0,\psi)$,  define
\begin{equation}\label{boya}
W_\psi(A,t)=-\frac{A}{2}\left(e^{-i\psi} (A-e^{-i\psi}t)^{-2}
+ e^{i\psi}(A-e^{i\psi} t)^{-2}\right).
\end{equation}
It was proved in \cite[Theorem 3.1]{Boy} that if $A \in \Sect(\theta)$,  
$A$ has dense range, and
\begin{equation}\label{absolute}
\int_{0}^{\infty}|\langle W_{\psi}(A,t)x, x^*\rangle |\, dt <\infty, \qquad  x \in X,\, x^* \in X^*,
\end{equation}
then $A$ admits
a bounded  $H^\infty(\Sigma_\psi)$-calculus given by
\begin{equation}\label{formula}
\langle f(A)x, x^*\rangle=\int_{0}^{\infty}\langle W_\psi(A,t)x, x^*\rangle (f_\psi \circ k_\psi)(t)\, dt, \quad x \in X,\, x^* \in X^*,
\end{equation}
where the integral converges absolutely (in the weak sense).
 Conversely, if $\psi \in (\theta,\pi)$ and $\varphi \in (\theta, \psi)$ are such that $A$ has a bounded  $H^\infty(\Sigma_{\varphi})$-calculus,
then \eqref{absolute} holds.   
(Note that in this situation $A$ has a bounded $H^\infty(\Sigma_{\psi})$-calculus given by \eqref{formula}, by the uniqueness of the calculus.) 
The formula \eqref {formula} is obtained in \cite{Boy} by rather involved Fourier analysis, and some technical details are omitted in \cite{Boy}.

In \cite[Proposition 5.1]{Boy} it is observed that, if $f \in H^\infty(\Sigma_\psi)$ and is holomorphic in a larger sector, and $f'_\psi \in L^1(\mathbb R_+)$, $f'_\psi(\infty)=0$ (this  assumption is not relevant), then one can integrate by parts and rewrite \eqref{formula} as 
\begin{equation}\label{boyad_f}
f(A)=\int_0^\infty V_\psi(A,t)
(f_\psi'\circ k_\psi)(t)\,\frac{dt}{t},
\end{equation}
where
\[
V_\psi(A,t)=-\frac{t}{2}
\left(e^{-i\psi}
(A-t e^{-i\psi})^{-1}
+e^{i\psi}(A-t e^{i\psi})^{-1}\right),
\]
and the integral converges absolutely. This formally leads to the estimate \eqref{NewEs1}.
While our reproducing $\arccot$-formula \eqref{formula_a} was inspired by \eqref{boyad_f}, it is not easy to put formal considerations in \cite{Boy} into the theory of functional calculi considered in this paper. One can relate \eqref{boyad_f} to \eqref{formula_a} and show that the formulas are essentially equivalent within the $\mathcal H$-calculus. This requires a number of technicalities and we intend to communicate them elsewhere.  Here we note only that  
$f'_\psi \in L^1(\mathbb R_+)$ and $f \in H^\infty(\Sigma_\psi)$ (for $\psi=\pi/2$)  do not imply that $f \in \H_\psi$ in general, 
as shown by an intricate example kindly communicated to us by A. Borichev.

\section{Convergence Lemmas and Spectral Mapping Theorems}\label{SMT_sec}

\subsection{Convergence Lemmas}

Given a negative semigroup generator $A$, a Convergence Lemma for the holomorphic functional calculus is a useful result allowing one to deduce the convergence of $(f_k(A))_{k=1}^\infty$ to $f(A)$ from rather weak assumptions on convergence of $(f_k)_{k=1}^\infty$ to $f$; see \cite[Lemma 2.1]{Cowling}, \cite[Proposition 5.1.4]{HaaseB} and \cite[Theorem 3.1]{BHM}, for example.

The following result is similar to a Convergence Lemma for the $\mathcal B$-calculus in \cite[Theorem 4.13 and Corollary 4.14]{BGT} (see also \cite[Section 8.1]{BGT2}).   However, the different Convergence Lemmas deal with different classes of functions.   To adjust the Convergence Lemma from \cite{BGT} to the current setting, we apply the change of variables method used in previous sections, and derive a variant of the Convergence Lemma for the $\mathcal D$-calculus.

In the following result, $f(A)$ refers to the $\D$-calculus.

\begin{thm}\label{convop_lemma}
Let $A \in \Sect(\pi/2-)$.  
Let $s >-1$ and let $(f_k)_{k\ge 1} \subset \mathcal{D}_s$ be such that
\begin{equation}
\sup_{k \ge 1}\|f_k\|_{\mathcal{D}_s} <\infty,
\end{equation}
and there exists
\[
f(z):=\lim_{k\to\infty}\,f_k(z),\qquad z\in \C_{+}.
\]
Let $g\in \mathcal{D}_s$  satisfy
\[
g(0)=g(\infty)=0.
\]
Then
\begin{equation}\label{conv_uni}
\lim_{k\to\infty}\,\|(f(A)-f_k(A))g(A)\|=0.
\end{equation}
In particular, if $A$ has dense range, then
\begin{equation}\label{conv_st}
\lim_{k\to\infty}\,\|f(A)x-f_k(A)x\|=0,
\end{equation}
for all $x\in X$.
\end{thm}

\begin{proof}
By assumption there exists $\theta \in (0,\pi/2)$ such that $A \in \Sect(\theta)$.
By Corollary \ref{Fatou} we have $f \in \mathcal D_s$.
Thus, without loss of generality, we can assume that $f\equiv 0$.

Let $\gamma\in (1, \pi/(2\theta))$. 
Then by Corollary \ref{fpD}, using the notation of \eqref{fgamma},
\[
f_k(A)=f_{k,1/\gamma}(A^{\gamma}) \qquad \text{and} \qquad g(A)=g_{1/\gamma}(A^{\gamma}).
\]
Since $g_{1/\gamma} \in \mathcal{D}_0^\infty$, $f_{k,1/\gamma}\in \mathcal{D}_0^\infty$ and $f_{k,1/\gamma} g_{1/\gamma}\in \mathcal{D}_0^\infty$ (see Corollary \ref{gggC}), and the $\D$-calculus is an algebra homomorphism, we have
\[
f_k(A)g(A)=f_{k,1/\gamma}(A^{\gamma})g_\gamma(A^{\gamma})=(f_{k,1/\gamma} g_{1/\gamma})(A^{\gamma}).
\]
Now \eqref{conv_uni} follows from Lemma \ref{48St2} and the continuity of the $\mathcal D$-calculus given by Proposition \ref{bounded_s}(1).

Let $g(z)=z(1+z)^{-2}$ and note that $g \in \mathcal D_s$, and $g$ vanishes at zero and at infinity. If $A$ has dense range, 
then the range of 
$g(A)=A(1+A)^{-2}$ is dense as well (see \cite[Proposition 9.4]{KW}, for example).  Since $\sup_{k \ge 1}\|f_k(A)\|<\infty$,  \eqref{conv_uni} implies  \eqref{conv_st}.
\end{proof}

In the following result, $f(A)$ refers to the $\H$-calculus.

\begin{thm}\label{convop_h}
Let $A \in \Sect(\theta),$ and let $\psi \in (\theta,\pi)$. 
Let $(f_k)_{k\ge 1} \subset \H_\psi$ be such that
\begin{equation*}
\sup_{k \ge 1}\|f_k\|_{\H_\psi} <\infty,
\end{equation*}
and there exists
\[
f(z):=\lim_{k\to\infty}\,f_k(z),\qquad z\in \C_{+}.
\]
Let $g\in\H_\psi$  satisfy
\[
g(0)=g(\infty)=0.
\]
Then
\begin{equation*}
\lim_{k\to\infty}\,\|(f(A)-f_k(A))g(A)\|=0.
\end{equation*}
In particular, if $A$ has dense range, then
\begin{equation*}
\lim_{k\to\infty}\,\|f(A)x-f_k(A)x\|=0,
\end{equation*}
for all $x\in X$.
\end{thm}

\begin{proof}
The proof is very similar to Theorem \ref{convop_lemma}.  Corollary \ref{Fatou} is replaced by Lemma \ref{FatouH}, the compatibility with fractional powers follows from the definitions and Proposition \ref{Hprops}(i), Corollary \ref{gggC} is replaced by Lemma \ref{simple}, Lemma \ref{48St2} is replaced by Lemma \ref{H11cl}, and \eqref{epsd11} is replaced by Corollary \ref{fpD}.
\end{proof}

\subsection{Spectral mapping theorems}

Given a semigroup generator $-A$, a spectral mapping theorem for a functional calculus $\Xi_A$  signifies informally
that $\Xi_A$ is associated to $A$ in a ``natural'' way.
However, in general the spectral ``mapping'' theorem
 states only the inclusion $f(\sigma(A))\subset \sigma(\Xi_A(f))$.
Equality may fail here even for functions such as  $e^{-tz}$
and for rather simple operators $A$; see \cite[Section IV.3]{EN}, for example.
While one may expect only the spectral inclusion as above,
 the equality $f(\sigma(A))\cup \{f(\infty)\} = \sigma(\Xi_A(f))\cup \{f(\infty)\}$ sometimes holds if $A$ inherits some properties of bounded operators such as strong resolvent estimates.
Note that the spectral mapping theorem may not hold
even for bounded operators if the functional calculus possesses only  weak
continuity properties, as discussed in \cite{Bercovici}.

The statement below shows that the $\mathcal D$-calculus
possesses the standard spectral mapping properties.
It is similar to \cite[Theorem 4.17]{BGT}, with the addition of a statement about approximate eigenvalues.  Recall that for $f \in \D_\infty$, its values $f(\infty)$ at infnity and $f(0)$ at $0$ are defined by \eqref{finf} and \eqref{fzero}.
This convention is used below.

\begin{thm}  \label{SMT}
 Let $A \in \Sect(\pi/2-)$, $f \in\mathcal D_\infty$, and $\l \in \C$.
\begin{enumerate} [\rm1.]
\item  If $x \in D(A)$ and $Ax = \l x$, then $f_{\mathcal D}(A)x = f(\l)x$.
\item  If $x^* \in D(A^*)$ and $A^*x^* = \l x^*$, then $f_{\mathcal D}(A)^*x^* = f(\l)x^*$.
\item  If $(x_n)_{n\ge1}$ are unit vectors in $D(A)$ and $\lim_{n\to\infty}\|Ax_n - \l x_n\| = 0$, then $\lim_{n\to\infty}\|f_{\D}(A)x_n - f(\l)x_n\| = 0$.
\item  One has $\sigma(f_{\mathcal D}(A)) \cup \{f(\infty)\} = f(\sigma(A)) \cup \{f(\infty)\}$.
\end{enumerate}
\end{thm}

\begin{proof}
The statements (1) and (2) are direct corollaries of \eqref{formulaD} and the reproducing formula for the $\mathcal D_s$-spaces given in Corollary \ref{Repr}. 

For (3), we use the $F$-product of the semigroup $(e^{-tA})_{t\ge0}$, as introduced in \cite{Dern}.  Let $Y$ be the Banach space of all bounded sequences $\mathbf y := (y_n)_{n\ge1}$ in $X$ such that $\lim_{t\to0+} \|e^{-tA}y_n - y_n\| = 0$ uniformly in $n$, where $(e^{-tA})_{t\ge0}$ is the bounded holomorphic $C_0$-semigroup generated by $-A$.   Let $Z$ be the closed subspace of $Y$ consisting of the sequences $\mathbf y$ such that $\lim_{n\to \infty} \|y_n\|=0$, and let  $\wt Y = Y/Z$ and $Q: Y \to \wt Y$ be the quotient map.   Then $(e^{-tA})_{t\ge0}$ induces a bounded holomorphic $C_0$-semigroup $(e^{-t\tilde A})_{t\ge0}$ on $\wt Y$, whose negative generator $\wt A$ is given by 
\[
D(\wt A) = \{Q(\mathbf y) :  y_n \in D(A), (Ay_n)  \in Y\},  \qquad \wt A(Q\mathbf y) = Q((Ay_n)).
\]
Then $\mathbf x := (x_n) \in Y$, $Q\mathbf x \in D(\wt A)$ and $\wt A Q(\mathbf x) = \l Q(\mathbf x)$.  It follows from  (1) that $f_{\D}(\wt A) Q\mathbf x = \l Q \mathbf x$.   However it is very easy to see that $f_{\D}(\wt A) Q\mathbf x = Q((f_{\D}(A)x_n))$ (see \cite[Theorem 1.7(i)]{Dern}), and this establishes (3).   

To prove the spectral mapping theorem in (4),
 we follow the Banach algebra method used in \cite{BGT} for similar purposes and inspired by \cite[Section 16.5]{HP} and \cite[Section 2.2]{Dav80}.   We may assume without loss of generality that $f(\infty)=0$.  Let $\mathcal A$ be the bicommutant of $\{(z+A)^{-1} : -z \in \rho(A)\}$ in $\B(X)$, so $\mathcal A$ is a commutative Banach algebra and the spectrum of $f_\D(A)$ in $\mathcal{A}$ coincides with the spectrum in $\B(X)$.
Observe that $\sigma (A) \subset \mathbb C_+\cup \{0\}$.

 Let $\chi$ be any character of $\mathcal{A}$.  
If $\chi((1+A)^{-1})=0$, then $\chi((z+A)^{-1}) = 0$ for all $z \in \C_+$, and hence $\chi(f(A)) = 0 = f(\infty)$.  Otherwise,
 by the resolvent identity, $\chi((z+A)^{-1})=(z+\l)^{-1}$ for some $\l \in \sigma(A)$ and all $z \in \mathbb C_+$. Let $s>-1$ be such that $f \in \D_s$.  Noting that the Stieltjes representation \eqref{Sti} converges in the uniform operator topology, we infer that  $\chi((z+A)^{-(s+1)}) = (z+\l)^{-(s+1)}, z \in \mathbb C_+$.
Applying $\chi$ to \eqref{formulaD} gives
\[
\chi (f_{\mathcal{D}}(A))=-
\frac{2^s}{\pi}\int_0^\infty \alpha^s\int_{-\infty}^\infty f'(\alpha+i\beta)(\l+\alpha-i\beta)^{-(s+1)}\,d\beta\,d\alpha,
\]
 and then, by the reproducing formula \eqref{qnrep} for $\mathcal D_s$-functions (valid on $\mathbb C_+\cup\{0\}$), we obtain
\[
\chi(f_{\mathcal{D}}(A)) = f(\l) \in f(\sigma(A)).
\]
Hence $\sigma(f_{\mathcal D}(A)) \cup \{0\} \subset f(\sigma(A)) \cup \{0\}$.
To prove the opposite inclusion, note that if $\l \in \sigma (A)$ is fixed, then there is a character $\chi$ such that $\chi ((z+A)^{-1})=(z+\l)^{-1}$, so the above argument can be reversed yielding
$\chi(f_{\mathcal{D}}(A))=f(\l)$, and thus finishing the proof.
\end{proof}

Next, using the same approach via Banach algebras,  we prove the analogous spectral result for the $\mathcal H$-calculus.

\begin{thm}
Let $A \in \Sect(\theta)$ and $f \in \mathcal \H_\psi$ for some $\theta <\psi < \pi$, and let $\l \in \C$.
\begin{enumerate} [\rm1.]
\item  If $x \in D(A)$ and $Ax = \l x$, then $f_{\mathcal H}(A)x = f(\l)x$.
\item  If $x^* \in D(A^*)$ and $A^*x^* = \l x^*$, then $f_{\mathcal H}(A)^*x^* = f(\l)x^*$.
\item  If $(x_n)_{n\ge1}$ are unit vectors in $D(A)$ and $\lim_{n\to\infty}\|Ax_n - \l x_n\| = 0$, then $\lim_{n\to\infty}\|f_{\H}(A)x_n - f(\l)x_n\| = 0$.
\item  One has $\sigma(f_{\mathcal H}(A)) \cup \{f(\infty)\} = f(\sigma(A)) \cup \{f(\infty)\}$.
\end{enumerate}
\end{thm}

\begin{proof}
The proofs of (1) and (2) are straightforward consequences of the reproducing formula \eqref{RepA1}.  Moreover, (3) is deduced from (1) in the same way as in Theorem \ref{SMT}.

The proof of (4) is similar to the corresponding proof in Theorem \ref{SMT}, based on the formula \eqref{sigma_def} which converges in the uniform operator topology.
Let $\gamma=\pi/(2\psi)$.   By the spectral mapping theorem for the holomorphic functional calculus \cite[Theorem 2.7.8]{HaaseB}, or \cite[Theorem 5.3.1]{MS}, one has
\begin{equation}\label{alpha_spec}
\sigma ((A^{\gamma}+t-i\b)^{-1})=\{(\l^\gamma+t-i\b)^{-1}:\l \in \sigma(A) \}\cup \{0\}, \quad t>0, \, \b \in \mathbb R.
\end{equation}
As in the proof of Theorem \ref{SMT}, let $\mathcal A$ be the bicommutant of $\{(z+A)^{-1} : -z \in \rho(A)\}$ in $\B(X)$.   Then $f_{\mathcal H}(A) \in \mathcal{A}$ and the spectrum of $f_{\mathcal H}(A)$ in $\mathcal A$  coincides with the spectrum in $\B(X)$.
Let $\chi$ be any character of $\mathcal{A}$, and let $f \in \mathcal \H_\psi$ be such that $f(\infty)=0$.   If $\chi((1+A^\gamma)^{-1})=0$, then, as above, $\chi((z+A^\gamma)^{-1}) = 0$ for all $z \in \C_+$, hence
$\chi(f_\H(A)) = 0 = f(\infty)$.  Otherwise $\chi ((z+A^{\gamma})^{-1})=(\l^{\gamma}+z)^{-1}$ for some $\l \in \sigma(A)$ and all $z \in \mathbb C_+$.  Applying $\chi$ to \eqref{sigma_def} 
and using the representation \eqref{RepA1} for $\H_\psi$-functions, one gets
\[
\chi (f_{\mathcal H}(A))= -\frac{1}{\pi}\int_0^\infty\int_{-\infty}^\infty f'_{1/\gamma}(t+i\b)
(\l^{\gamma}+t-i\b)^{-1}\,d\b\, dt=f(\l).
\]
Hence $\sigma (f_{\mathcal H}(A))\cup \{0\} \subset f(\sigma(A)) \cup\{0\}$.

On the other hand, if $\l \in \sigma(A)$, then by \eqref{alpha_spec} there is a character $\chi$ such that $\chi ((z+A^{\gamma})^{-1})
=(z+\l^\gamma)^{-1}$. So using \eqref{RepA1} again, we infer that  $f(\l) \in \sigma (f_{\mathcal H}(A))$.

Combining the two paragraphs above yields (4).
\end{proof}

Our spectral mapping theorems differ from known spectral mapping theorems for the holomorphic functional calculus (see \cite{Haase_Sp} or \cite[Section 2.7]{HaaseB}) in at least three respects.   We do not assume that $A$ is injective, we cover a wider class of functions including some with a mild singularity at zero (for example, $e^{-1/z}$), and our proofs are completely different.

\section{Some applications to norm-estimates}\label{norm_estimates}

In this section we apply directly the $\mathcal D$- and $\mathcal H$-calculi that we have constructed
to obtain some operator norm-estimates. In particular, 
we obtain uniform bounds on the powers of Cayley transforms
and on the semigroup generated by the inverse of a semigroup generator.
We then compare the results to known estimates in the literature.
We also revisit the theory of holomorphic $C_0$-semigroups and obtain several basic estimates
along with some slight generalizations.

\subsection{Norm-estimates via the $\mathcal D$-calculus}

Let $A \in \Sect(\pi/2 -)$, and $V(A)$ be the Cayley transform $(A-I)(A+I)^{-1}$ of $A$.
We now review several important estimates from the literature in the framework of the constructed $\D$- and $\H$-calculi.

Recall that $-A$ is the generator of a bounded holomorphic semigroup $(e^{-tA})_{t\ge0}$.  Let $e_t(z) = e^{-tz}, \, t \ge 0,\,z\in \C_+$.  Then $e_t \in \LT \subset \D_\infty$ and 
\[
e^{-tA} = (e_t)_{\mathrm{HP}}(A) = (e_t)_{\D}(A).
\]

\begin{cor}\label{cayley_l}
Let $A \in \Sect(\pi/2 -)$, so that \eqref{Aaa1} holds.
\begin{enumerate}[{\rm(i)}]
\item
 One has
\[
\|V(A)^n\|\le 1+32(1+(\sqrt2\pi)^{-1})M_A^2, \qquad n \in \mathbb N.
\]
\item
One has
\[
\|e^{-tA}\|\le 2 M_A^2, \qquad t \ge 0.
\]
\item For every $\nu>0$ one has
\[
\|A^\nu e^{-tA}\| \le 2^{\nu+2} t^{-\nu} \Gamma(\nu+1) M_A^{\lceil\nu\rceil+2}, \qquad t\ge0. 
\]
\end{enumerate}
If, in addition, the inverse $A^{-1}$ exists and is densely defined, then $A^{-1}$ generates
a bounded holomorphic $C_0$-semigroup $(e^{-tA^{-1}})_{t \ge 0}$ satisfying
\[
\|e^{-tA^{-1}}\|\le 1 + 2 M_A^2, \qquad t \ge 0,
\]
and, for every $\nu>0$,
\[
\|A^{-\nu} e^{-tA^{-1}}\| \le 2^{\nu+2} t^{-\nu} \Gamma(\nu+1) M_A^{\lceil\nu\rceil+2}, \qquad t\ge0.
\]
\end{cor}

\begin{proof}
By Lemma \ref{cayley} and Theorem \ref{dcalculus}, for every $s >0$,
\[
\|V(A)^n\|\le 1+  2^{s+4} \pi^{-1} (B(s/2, 1/2)+2^{-s/2}) M_A^{\lceil{s}\rceil+1}, \qquad n \in \mathbb N.
\]
Setting $s=1$ we get the assertion (i).

By Proposition \ref{BDs} or by Example \ref{dm} and Lemma \ref{leminv}(iii), the function $e^{-tz} \in \D_s$ for $s > 0$ and $t>0$, and by \eqref{frac} and \eqref{estimateD},
\[
\|e^{-tA}\|\le 2^s \pi^{-1}B(s/2, 1/2)M_A^{\lceil{s}\rceil+1}, \qquad t \ge 0.
\]
So the estimate (ii) follows by setting $s=1$ above. 

If $f_\nu(z):=z^\nu e^{-tz}$, $\nu>0$, then $f_\nu \in\D_s$ if and only if $s >\nu$, and in that case $\|f_\nu\|_{\D_s} = 2t^{-\nu} B((s-\nu)/2,1/2) \Gamma(\nu+1)$ (see Example \ref{dm}).  
Since  $f_\nu$ has zero polynomial limits at zero and at infinity, $(f_{\nu})_{\mathcal D}(A)$ coincides with $A^{\nu}e^{-tA}$
as defined by the holomorphic functional calculus (see Remark \ref{coincid}).
Using \eqref{frac}, it follows that, for every $s >\nu$,
\[
\|A^\nu e^{-tA}\|\le \frac{2^{s+1}t^{-\nu}}{\pi} B\left(\frac{s-\nu}{2},\frac{1}{2}\right)\Gamma(\nu+1) M_A^{\lceil s\rceil+1}.
\]
Setting $s=\nu+1$, the first assertion in (iii) follows.  The other two estimates are consequences of Lemma \ref{leminv}(i) and the estimates for $e^{-tA}$ and $A^\nu e^{-tA}$ obtained above.  
\end{proof}

The results in Corollary \ref{cayley_l} are not new, and it serves as an ilustration of the utility of the $\D$-calculus.  We have not aimed at finding the best possible estimates, but it seems that the $\H$-calculus provides bounds that are fairly precise whenever it is applicable.    
The power-boundedness of $V(A)$ was shown in \cite{Cr93} and \cite{Palencia}, using different methods.   In \cite[Corollary 5.9]{BGT}, a weaker result was shown using the $\Bes$-calculus (so all operators satisfying \eqref{8.1}).  Corollary \ref{cayley_l} shows how the $\D$- and $\H$-calculi can give a sharper estimate than the $\Bes$-calculus in the case of sectorial operators.  
Part (ii) above is one of many estimates for the bound on a bounded holomorphic semigroup in terms of its sectorial bound, and it is clearly not sharp.  
A careful estimation in \cite[Lemma 4.7]{BGT} of the bound obtained via the $\Bes$-calculus gave a bound of order $M_A \log M_A$ when $M_A$ is large. See also \cite[Theorem 5.2]{Sch1} where the result was established for the first time. Estimates of the form given in part (iii) have been known for a long time, but usually without showing the dependence on $M_A$.

Next we consider estimates similar to Corollary \ref{cayley_l}(iii).  In Lemma \ref{ders} and Theorem \ref{AnalF} $f(A)$ refers to the $\D$-calculus.

\begin{lemma}  \label{ders}
Let $A\in \Sect(\pi/2-).$ If $f \in \D_s$, $s>-1,$ and $n \in\N$, then
\begin{equation}\label{formulaDD1}
(z^nf^{(n)})(A)=  C_{s,n} \int_0^\infty \alpha^s \int_{-\infty}^\infty f'(\alpha+i\beta)A^n (A+\alpha-i\beta)^{-(s+n+1)}\,d\beta\,d\alpha,
\end{equation}
where
\[
C_{s,n} = {(-1)}^{n+1}
2^s \frac{\Gamma(s+n+1)}{\Gamma(s+1)}.
\]
\end{lemma}

\begin{proof}
By Corollary \ref{AlgDer} and the boundedness of the $\D_{s+n}$-calculus, \linebreak[4] $(z^nf^{(n)})(A)$ coincides with the derivative of order $n$ of the function $t \mapsto f(tA)$ evaluated at $t=1$.   The formula \eqref{formulaD} for $f(tA) \in \mathcal D_s$ can be differentiated repeatedly with respect to $t$ by a standard method, and putting $t=1$ then gives the formula \eqref{formulaDD1}.
\end{proof}

\begin{thm}\label{AnalF}
Let $A \in \Sect(\pi/2-)$ and $n \in \N$.  Let $f \in \mathcal{D}_\infty$, and assume that $f^{(k)} \in \D_\infty$ for $k=1,\dots,n$.  Then
\begin{equation} \label{anzn}
 (z^nf^{(n)})(A) = A^n f^{(n)}(A).
\end{equation}
Moreover, if $f \in \D_s, \, s>-1$, then
\begin{equation}\label{tFG}
\|t^n A^nf^{(n)}(tA)\|\le \frac{2^s\Gamma(s+n+1)}{\pi \Gamma(s+1)}(M_A+1)^n M_A^{\lceil{s}\rceil+1}\|f\|_{\mathcal{D}_s},\quad t>0.
\end{equation}
In particular, for $f(z)=e^{-z}\in \mathcal{D}_1$,
\begin{equation}\label{tFGA}
\|t^n A^n e^{-tA}\|\le 2(n+1)!(M_A+1)^n M_A^2,\qquad t>0,\quad n\in \N.
\end{equation}
\end{thm}

\begin{proof}
We will prove, by induction on $n$, that \eqref{anzn} holds for all functions $f \in  \D_\infty$ such that $f^{(k)} \in \D_\infty$ for $k=1,\dots,n$. .  First, assume that $f,f' \in \D_\infty$.   Then
\[
(1+A)^{-1} (zf')(A) + (1+A)^{-1}f'(A) = f'(A),
\]
This implies that $(zf')(A) = Af'(A)$. 

Now assume that, for some $n\ge1$, $(z^kg^{(k)})(A) = A^kg^{(k)}(A)$ for $k=1,\dots,n$, for all functions $g$ such that $g^{(k)} \in \D_\infty$ for $k=0,\dots,n$.  Let $f^{(k)} \in \D_\infty$ for $k=0,1,\dots,n+1$.  Applying the base case ($k=1$) to the function $z^nf^{(n)}$ (noting that this function and its first derivative are in $\D_\infty$, by Corollary \ref{D00n}), we obtain
\[
(z(z^nf^{(n)})')(A) = A (z^nf^{(n)})'(A).
\]
Then applying the inductive hypothesis with $k=n$ to the function $f$, and with  $k=n-1$ and $k=n$ to the function $f'$, we obtain
\begin{align*}
(z^{n+1}f^{(n+1)})(A) &=  A \left(n(z^{n-1}f^{(n)}) + (z^nf^{(n+1)})\right)(A) - n (z^nf^{(n)})(A) \\
&= A^{n+1}f^{(n+1)}(A).
\end{align*}
This completes the proof of the inductive hypothesis for all $n\in\N$, and hence proves \eqref{anzn}

Since $M_A = M_{tA}$ for all $t>0$, it suffices to prove (\ref{tFG}) for $t=1$.  From Lemmas \ref{fractional} and \ref{ders},
we obtain
\begin{align*}
\lefteqn{\null\hskip-15pt \|A^n\,f^{(n)}(A)\|}\\
&\le \frac{2^s\Gamma(s+n+1)}{\pi \Gamma(s+1)}(M_A+1)^n M_A^{\lceil{s}\rceil+1} \int_0^\infty\alpha^s\int_{-\infty}^\infty
\frac{|f'(\alpha+i\beta)|}{|\alpha-i\beta|^{s+1}}\,d\b\,d\alpha\\
&=\frac{2^s \Gamma(s+n+1)}{\pi\Gamma(s+1)}(M_A+1)^n M_A^{\lceil{s}\rceil+1}\|f\|_{\mathcal{D}_s}.
  \qedhere
\end{align*}
\end{proof}

\begin{rem}  \label{AFrem}
In Theorem \ref{AnalF}, the assumption that $f \in \D_\infty$ and $f^{(k)} \in \D_\infty, k=1,2,\dots,n$, can be replaced by the assumption that $f \in \D_\infty$ and $f^{(n)} \in \D_\infty$, by using a result of Lyubich \cite{Lyu}. See Corollary \ref{lyubich}.  
\end{rem}

\subsection{Norm-estimates via the $\H$-calculus}  \label{Bernstein}

Now we use the $\H$-calculus to provide a new proof that holomorphy of operator semigroups generated by $-A$ is preserved for subordinate semigroups generated by $-g(A)$ where $g$ is a Bernstein function.  This was proved for the first time in \cite{GT15}.

If $-A$ is the generator of a bounded $C_0$-semigroup $(e^{-tA})_{t \ge 0}$
on a Banach space $X$, and $g$ is a Bernstein function given by
\eqref{bsfn},
then the operator 
\begin{equation}\label{DBer}
g_0(A)x:=ax+bAx+\int_{(0, \infty)}\left(x-e^{-tA}x\right)\, d\mu(t), \quad x \in D(A),
\end{equation}
is closable, and $g(A)$ can be defined as the closure of $g_0(A)$.  Thus $D(A)$ is a core for $g(A)$, and one can prove that $-g(A)$ generates a contraction $C_0$-semigroup on $X$.
Several equivalent definitions of $g(A)$ are possible, and we refer the reader to \cite{Schill}, \cite{GHT} and \cite{GT15}.
If $A$ is injective, then $g(A)$ is well-defined within the (extended) holomorphic functional calculus
and is given by \eqref{DBer} as above; see Proposition \ref{Hprops} and \cite[Propositions 3.3 and 3.6]{GT15}.

The next statement shows that the so-called semigroup subordination preserves the holomorphy of $C_0$-semigroups along with the holomorphy angles. It was one of the main results of \cite{GT15}, settling a question raised by Kishimoto and Robinson \cite{KR}.
See also \cite{BGTad} and \cite{BGT} for generalizations and other proofs.

\begin{thm}\label{BF1}
Let $A \in \Sect(\theta)$, where $\theta\in [0,\pi/2)$,
and let $g$ be a Bernstein function as in
\eqref{bsfn}.
Then $g(A) \in \Sect(\theta)$.   More precisely, for all  $\psi\in (\theta,\pi/2)$, $\varphi\in (\psi,\pi)$, and
$\lambda\in{\Sigma}_{\pi-\varphi}$,
\begin{equation}\label{bound_berns}
\|\lambda (\lambda + g(A))^{-1}\| \le  2M_\psi (A)\left(\frac{1}{\sin(\min(\varphi,\pi/2))}+\frac{2}{\cos\psi \sin^2((\varphi-\psi)/2)}\right).
\end{equation}
\end{thm}

\begin{proof}
Let $\psi\in (\theta,\pi/2)$, $\varphi\in (\psi,\pi)$, and
$\lambda\in{\Sigma}_{\pi-\varphi}$. If
\[
f(z)=f(z;\lambda):=(\lambda+g(z))^{-1},\qquad z\in \Sigma_\psi,
\]
then by Corollary \ref{BF} we have $f \in \mathcal \H_\psi$.  
We will show that $f_\H(A) = (\lambda+g(A))^{-1}$.  It then follows from  \eqref{NewEs1} and \eqref{bernst}  that \eqref{bound_berns} holds.  Since the choice of $\psi\in (\theta,\pi)$ and
$\varphi\in (\psi,\pi)$ is arbitrary, this shows that the operator $g(A)$ is  sectorial
of angle $\theta$.

If $A$ is injective, then $f(A)$ and $\lambda + g(A)$ are consistently defined in the holomorphic functional calculus,  
and therefore $f_{\mathcal H}(A) = f(A) = (\lambda+g(A))^{-1}$ (see \cite[Theorem 1.3.2f)]{HaaseB}).

When $A$ is not injective, we follow the approach proposed in the proof of \cite[Theorem 4.8]{BGTad}.

Since $A+\epsilon$ is invertible, we have
\[
(\lambda+g(A+\epsilon))^{-1}=f_\H(A+\epsilon).
\]
By Proposition \ref{Hprops}(v), 
\begin{equation} \label{epsD}
\lim_{\ep\to0+} \|f_\H(A+\ep) - f_\H(A)\|=0.
\end{equation}

Let $x \in D(A)$.  Since $f_\H(A+\epsilon)$ commutes with $(1+A)^{-1}$, 
we have $f_\H(A+\epsilon)x\in D(A)$, and 
by \eqref{DBer},
\begin{gather} \label{first}
x-(\lambda+g(A))f_\H(A+\epsilon)x
=[g(A+\epsilon)-g(A)]f_\H(A+\epsilon)x\\
 \null{\hskip30pt} =\epsilon bf_\H(A+\epsilon)x
-\int_{(0,\infty)} (1-e^{-\epsilon t})e^{-tA} f_\H(A+\epsilon)x\,d\mu(t).\notag
\end{gather}
It follows from \eqref{epsD} that 
\[
C_\lambda := \sup_{\ep\in(0,1]} \|f_\H(A+\ep)\| < \infty,
\] 
and hence 
\begin{align}\label{second}
\|x - (\lambda+g(A))f_\H(A+\epsilon)x\| &\le
\epsilon b C_\lambda\|x\|
+C_\lambda K_A \int_{(0,\infty)} (1-e^{-\epsilon t})\,d\mu(t)\,\|x\| \\
&\to 0,\qquad \epsilon\to 0,  \notag
\end{align}
where $K_A := \sup_{t>0} \|e^{-tA}\|$.  Since $\lambda+g(A)$ is closed, it follows firstly that
\[
(\lambda+g(A))f_\H(A)x=x,\quad x\in D(A).
\]
Since $D(A)$ is dense in $X$ and $f_\H(A)$ is bounded, it follows secondly that
\[
(\lambda+g(A))f_\H(A)x=x,\qquad x\in X.
\]
Since $\l + g(A)$ and $f_\H(A)$ commute on $D(A)$,
\[
f_\H(A)(\lambda+g(A))x=x,\qquad x\in D(A).
\]
Since $D(A)$ is a core for $g(A)$, it follows that this holds for all $x \in D(g(A))$.   Thus, $f_\H(A)=(\lambda+g(A))^{-1}$, as required.
\end{proof}

\begin{rems}
1.  A new feature of Theorem \ref{BF1} is an explicit sectoriality constant for $g(A)$, given by the right hand-side of \eqref{bound_berns}. 
This could be valuable when applying the result to families of sectorial operators.  
Thus \eqref{bound_berns} offers an improvement over similar estimates in \cite{BGTad}, \cite{BGT} and \cite{GT15}, where the sectoriality constants for $g(A)$ are rather implicit.

\noindent 2.  
We take this opportunity to correct a parsing misprint in the proof \cite[Theorem 4.9]{BGTad}.
One should replace $f(A)$ with $f(A)+z$ in the third and fourth displays on \cite[p.932]{BGTad} 
(see \eqref{first} and \eqref{second} above for similar formulas).
\end{rems}

Finally, as an illustration, we show how the holomorphy of $C_0$-semigroups generated by operators $-A^\gamma$ fits within the $\H$-calculus, and how estimates of similar type to Corollary \ref{cayley_l}(ii) can be obtained from the representation of the $\H$-calculus and the function $\arccot$, as in Theorem \ref{Sovp}.   
The following result is similar to \cite[Corollary 5.2]{Boy}, 
and a generalization of the main result in \cite{DeLa87} to non-integer $\gamma$. See also \cite[Remark 2, p.83]{Cr93}.

\begin{cor} \label{better}
Let $A \in \Sect(\theta)$,  $\theta\in (0,\pi)$, and $\gamma  \in (0,\pi/(2\theta))$.  Then $(e^{-tA^\gamma})_{t\ge0}$ is a bounded holomorphic $C_0$-semigroup of angle $(\pi/2) - \gamma\theta$.  More precisely, if $\psi \in (\theta,\pi/(2\gamma))$ and $\lambda = |\l| e^{i\varphi} \in \Sigma_{(\pi/2)-\gamma\psi}$, then
\begin{equation} \label{est_gamma}
\|e^{-\lambda A^{\gamma}}\| \le \frac{1}{2}\left(\frac{1}{\cos(\gamma\psi+\varphi)}+\frac{1}{\cos(\gamma\psi-\varphi)}\right) M_{\gamma\psi}(A).
\end{equation}
\end{cor}

\begin{proof}
Let $\psi \in(\theta,\pi)$.  Since the $\H$-calculus $\Upsilon_A$ is a homomorphism and $\Upsilon_A(e^{-\l z^\gamma}) = e^{- \lambda A^\gamma}$ for every $\l \in \Sigma_{(\pi/2)-\gamma\psi}$, the family $(e^{-\l A^\gamma})_{\l \in \Sigma_{(\pi/2)-\gamma\psi}}$ is an operator semigroup.  
By Theorem \ref{hardy2} and the argument in Section \ref{prelims}, the map $\l \mapsto e^{-\l z^\gamma}$ is holomorphic from $\Sigma_{(\pi/2)-\gamma\psi}$ to $\H_\psi$, so $\lambda \mapsto e^{-\l A^\gamma}$ is also holomorphic. 
The estimate \eqref{est_gamma} follows from Example \ref{hexs}(2) and \eqref{NewEs2}, and it shows boundedness of the semigroup on each relevant sector.
\end{proof}

Note that, for $\gamma=1$, Corollary \ref{better} provides a sharper bound than Corollary \ref{cayley_l}(ii).

\section{Appendix: Shifts on  $\mathcal{D}_s$ and $\H_\psi$} \label{shifts} 

The shift semigroups on the space $\Bes$ had an important role in the study of the $\Bes$-calculus in \cite{BGT} and \cite{BGT2}.   While the semigroups are not essential in this paper, we think they will be important for further research, and so we describe their properties on the spaces $\D_s$ and $\H_\psi$.  In this appendix  we prove that the shifts $(T(\tau))_{\tau\in\C_+}$ given by
\[
(T(\tau)f)(z):=f(z+\tau),\quad z\in \C_{+},\quad \tau\in \C_{+},
\]
form a holomorphic $C_0$-semigroup on $\D_s$ for each $s>-1$. We also show that a similar statement holds for shifts on $\H_\psi$ for each $\psi \in(0,\pi)$.

We consider first the space $\D_s$, and we begin by proving that the semigroup $(T(\tau))_{\tau \in \Sigma_\psi}$ of operators is uniformly bounded on $\D_s$, for each $s>-1$ and $\psi \in (0,\pi/2)$.

\begin{thm}\label{21T}
Let $s>-1$, $\psi \in (0,\pi/2),$ and $a=\tan\psi$.   For all $\tau\in \Sigma_\psi$, we have
\begin{equation}\label{Taa}
\|T(\tau)f\|_{\mathcal{D}_s}\le C_{a,s}\|f\|_{\mathcal{D}_s},\quad f\in \mathcal{D}_s,
\end{equation}
where
\begin{align*}
C_{a,s} &:=
\frac{(s+1)2^{s}B((s+1)/2,1/2)}{\pi \cos\psi \cos^{s+2}\psi_a}
+2^{s+1}, \label{const_a}\\
 \psi_a &:=\arctan \big(a+\sqrt{1+a^2}\big).
\end{align*}
\end{thm}

\begin{proof}
Let $\tau \in \Sigma_\psi$ and $f \in \D_s$.    We have
\begin{align*}
\|T(\tau)f\|_{\mathcal{D}_s}&\le |f(\infty)|+
\int_{\C_+}  \frac{(\Re z)^s}{|z|^{s+1}} |f'(z+\tau)| \, dS(z) = |f(\infty)|+J(\tau),
\end{align*}
where $dS$ denotes area measure on $\C_+$ and
\[
J(\tau):=\int_{\Re z \ge \Re\tau} \frac{(\Re z-\Re\tau)^s}
{|z-\tau|^{s+1}} |f'(z)|\,dS(z).
\]

Let
\begin{align*}
W(\tau)&:= \left\{z\in \C: \Re z\ge \Re\tau,\;|z-\tau| \le |\tau| \right\},\\
W_0(\tau):&= W(\tau)-\tau= \left\{z\in \C_{+}:\;|z| \le |\tau| \right\}.
\end{align*}
If $\Re z \ge \Re\tau$ and $z \notin W(\tau)$, then   $|z| \le |z-\tau|+|\tau| \le 2|z-\tau|$.      Hence
\begin{align} \label{JSo}
J(\tau) &\le \int_{W(\tau)}
\frac{(\Re z - \Re\tau)^s}{|z-\tau|^{s+1}}  |f'(z)|\,dS(z)\\
& \null \hskip30pt + 2^{s+1}\int_{\Re z \ge \Re\tau} \frac{(\Re z)^s}{|z|^{s+1}} |f'(z)| \,dS(z)
 \notag \\
&\le \max_{z\in W(\tau)}\,|f'(z)| \int_{W_0(\tau)}
\frac{(\Re z)^s}{|z|^{s+1}}\,dS(z) +2^{s+1}\|f'\|_{\mathcal{V}_s}.\notag
\end{align}
Moreover,
\begin{equation}\label{sts}
\int_{W_0(\tau)} \frac{(\Re z)^s}{|z|^{s+1}}\,dS(z)
=  \int_{-\pi/2}^{\pi/2}\int_0^{|\tau|}  \cos^{s}\varphi \,d\rho\,d\varphi = |\tau| B((s+1)/2,1/2).
\end{equation}
For $z\in W(\tau)$, we also have $\Re z \ge \Re\tau$ and
\[
|\Im z| \le |\Im\tau| + |\tau|\le \left(a + \sqrt{1+a^2}\right) \Re\tau.
\]
Hence $z\in \Sigma_{\psi_a}$ and $|z|\ge \Re\tau$, so by Corollary \ref{Cangle},
\begin{equation}\label{der}
\max_{z\in W(\tau)}\,|f'(z)| \le
\max_{z\in \Sigma_{\psi_a},\,|z|\ge \Re\tau}\,|f'(z)|
\le \frac{(s+1)2^s}{\pi \Re\tau\cos^{s+2}\psi_a}\|f'\|_{\mathcal{V}_s}.
\end{equation}
Inserting the estimates \eqref{sts} and \eqref{der}  into (\ref{JSo}) and using
$|\tau|\le \Re\tau/\cos\psi$,  we obtain
\[
J(\tau)\le \left(
\frac{(s+1)2^s B((s+1)/2,2)}{\pi\cos\psi \cos^{s+2}\psi_a}
+2^{s+1}\right)
\|f'\|_{\mathcal{V}_s},
\]
and (\ref{Taa}) follows.
\end{proof}

\begin{cor}\label{SemDa}
For any $s>-1$,  the family $T:=(T(\tau))_{\tau\in \C_{+}}$ is
a bounded holomorphic $C_0$-semigroup on $\mathcal{D}_s$ of angle $\pi/2$.  The generator of the semigroup is $-A_{\D_s}$, where
\[
D(A_{\D_s}) = \{f \in \D_s : f' \in \D_s\}, \qquad A_{\D_s}f = - f'.
\]
\end{cor}

\begin{proof}
By Theorem \ref{21T}, $T$ is bounded on $\Sigma_\psi$ for each $\psi \in (0,\pi/2)$, and as noted in Remark \ref{remdr} the function $\l \mapsto r_\l$ is a holomorphic function from $\C_+$ to $\D_s$, so $\tau \mapsto T(\tau)r_\l$ is holomorphic.  Since $\widetilde{\mathcal{R}}(\C_+)$ is dense in $\D_s$ (Theorem \ref{D00}), it follows that $T$ is strongly continuous on $\Sigma_\psi \cup \{0\}$, and moreover, for $f \in \D_s$, the map $\tau \mapsto T(\tau)f$  is holomorphic on $\C_+$.

The proof of the statement about the generator is almost identical to the proof for the space $\Bes$ in \cite[Lemma 2.6]{BGT}.
\end{proof}

The following corollary justifies Remark \ref{AFrem} about the assumptions in Theorem \ref{AnalF}.

\begin{cor} \label{lyubich}
Let $f \in \D_s$, $s>-1,$ and assume that $f^{(n)} \in \D_s$ for some $n \in \N.$   Then $f^{(k)} \in \D_s$ for $k=1,2,\dots,n-1$.
\end{cor}

\begin{proof}
Consider the operator $Ag = -g'$ on $\Hol(\C_+)$, and its part $A_{\D_s}$ in the subspace $\D_s$.  The operators $A_{\D_s}+m, \, m=1,2,\dots,n$, are surjective on $\D_s$, and $g=0$ is the only solution in $\D_s$ to $\prod_{m=1}^n(A+m)g=0$.  The statement follows
from \cite[Theorem 1]{Lyu}.
\end{proof}

\begin{rem} 
The space $\Bes$ is invariant under vertical shifts: $f(z) \mapsto f(z+i\sigma)$ for $\sigma \in \R$.   However the spaces $\D_s$ and $\D_s^\infty$ are not invariant under vertical shifts. See Example \ref{eac}.
\end{rem}

Now we will show that the family of shifts $T$ also forms a bounded holomorphic $C_0$-semigroup on $\H_\psi$ for every $\psi\in(0,\pi)$.   If $\psi>\pi/2$, then $T(\tau)$ are defined for $\tau \in \Sigma_{\pi-\psi}$.
 For this aim, we will recall the Gabriel inequality for holomorphic functions.

Let $\Omega$ be a bounded convex domain in $\C$ and let $\Gamma \subset \overline{\Omega}$ be a convex curve.
Then there exists a universal constant $K>0$ (not depending on $f, \Omega$ and $\Gamma$) such that, for all $f \in \Hol(\Omega) \cap C(\overline\Omega)$,
\begin{equation}\label{gabriel}
\int_{\Gamma} |f(z)|\, |dz| \le K \int_{\partial \Omega} |f(z)|\, |dz|.
\end{equation}
Clearly $K \ge 1.$ Moreover, it can be shown that if $\Gamma$ is closed, then $2 <  K < 3.7$ (see \cite[p.457]{Beurling}, for example). The inequality was conjectured by J. Littlewood and first proved by Gabriel in \cite[Theorem I]{Gabriel_35}.  It is thoroughly discussed in \cite[Selected Seminars, 2, 4 and 5]{Beurling} and \cite[Section 5]{Granados} providing simpler proofs, more general versions and additional insights.   

\begin{thm}\label{Hsg}
Let $\psi \in (0,\pi)$ and $\psi_0=\min\{\psi,\pi-\psi\}$. The 
 family $T=(T(\tau))_{\tau \in \Sigma_{\psi_0}}$ is
a bounded holomorphic $C_0$-semigroup of angle $\psi_0$ on each of the spaces $\mathcal H_\psi$. 
The generator $-A_{\mathcal H_\psi}$ of $T$ on $\mathcal H_\psi$ is given by 
\[
D(A_{\mathcal H_{\psi}}) = \{f \in \mathcal H_\psi : f' \in \mathcal H_\psi \}, \qquad A_{\mathcal H_{\psi}}f = - f'.
\]
\end{thm}

\begin{proof}
We will show first that the family of shifts $T$ is uniformly bounded on $\mathcal H_\psi$ 
for every $\psi \in (0,\pi).$ Then the result 
follows quickly by a density argument.

Let $g \in \mathcal H_\psi$ so that for $f=g'$ one has $f \in \in H^1(\Sigma_\psi).$ 
Assume first that $\psi \in (0,\pi/2]$, and let $\psi' \in (0,\psi)$  and $\a:=\sin(\psi -\psi')$.  
By the mean value inequality, for any $r>0$ and $\varphi \in (-\psi',\psi')$ we have
\begin{align}\label{limit1}
r |f(re^{i\varphi})| &\le \frac{1}{\pi \a^2 r}\int_{|z-re^{i\varphi}|\le \a r} |f(z)| \, dS(z)\\
&\le \frac{2(1+\a)}{\pi \a} \int_{-\psi'}^{\psi'} \int_{(1-\a)r}^{(1+\a)r} |f(\rho e^{i\varphi})|\, d\rho\,d\varphi \notag\\
&\to 0,\notag
\end{align}
as $r \to 0$ or $r \to \infty$, by the dominated convergence theorem.

Now let $\tau \in \Sigma_\psi$ and $\varphi \in (0,\psi).$ 
Let 
\begin{equation}\label{gamma_t}
\Gamma_{\tau, \varphi}=\{\tau+te^{i\varphi}: t \ge 0\} \cup \{\tau+te^{-i\varphi} :  t \ge 0 \}. 
\end{equation}
Let $\psi'>\max(|\arg\tau|,\varphi)$, and take $r\in (0,1)$ such that  $0< r <|\tau| <  1/r$.   We now apply Gabriel's inequality \eqref{gabriel} with  
\[
\Omega_{r} := \{z  \in \C: \Re z > r \cos\psi',  |z| < 1/r, |\arg z| < \psi'\}
\]
and
\[
\Gamma_{\tau,\varphi,r} := \{z \in \Gamma_{\tau,\varphi} : |z| \le 1/r\} \cup \{z \in \C: |z|=1/r, \, |\arg (z-\tau)| \le \varphi\}.
\]
We obtain
\begin{align*}\label{gabr}
&\\
\lefteqn{K^{-1} \int_{\Gamma_{\tau,\varphi,r}} |f(z)|\, |dz| 
\le \int_{\partial\Omega_{r}} |f(z)|\, |dz|} \notag\\
&\le \int_{r}^{1/r}|f(\rho e^{i\psi'})|\, d\rho 
+\int_{r}^{1/r}|f(\rho e^{-i\psi'})|\, d\rho 
+ 2 \pi r^{-1} \sup_{\varphi \in (-\psi', \psi')} |f(e^{i\varphi}/r)| \notag \\
&\null\hskip20pt +2 r\sin\psi' \sup_{\varphi \in (-\psi', \psi')}
 \left|f\left(\frac{r\cos \psi'}{\cos \varphi} e^{i\varphi}\right)\right|,\notag
\end{align*}
where $K>0$ is given by \eqref{gabriel}.

 By \eqref{limit1},
the last two terms converge to zero as $r \to 0.$ 
From Theorem \ref{hardy0}(iv) it follows that 
\begin{equation*} \label{acute}
K^{-1} \int_{\Gamma_{\tau,\varphi}} |f(z)|\, |dz| \le  \|f\|_{H^1(\Sigma_{\psi'})}\le \|f\|_{H^1(\Sigma_\psi)}.
\end{equation*}
Thus, since the choice of $\varphi \in (0, \psi)$ was arbitrary, we have   
 \begin{equation}\label{hardy1}
\|T(\tau)f\|_{H^1(\Sigma_\psi)}\le K \|f\|_{H^{1}(\Sigma_\psi)}.
\end{equation}

Now we consider the case when $\psi \in (\pi/2,\pi)$ and $\tau \in \Sigma_{\pi-\psi}.$ 
Let for $\varphi \in (0,\psi)$  
the path $\Gamma_{\tau, \varphi}$ be given by \eqref{gamma_t},
and $\Gamma^{\pm}_{\tau,\varphi}:= \{\tau+te^{\pm i\varphi}: t \ge 0\}.$
Define the half-planes 
$\C_+^{\psi}:= -ie^{i\psi} \Sigma_{\pi/2}$
and $\C_-^{\psi}:=ie^{-i\psi} \Sigma_{\pi/2}$
so that $\Gamma^{\pm}_{\tau,\varphi}\subset \C_{\pm}^{\psi}.$
Letting $\widetilde f (z)=f(ie^{-i\psi}z), z \in \C_+^{\psi},$  
 and applying \eqref{hardy1} with $\psi=\pi/2,$   $\widetilde f\in H^1(\Sigma_{\pi/2})$ in place of $f,$
and $\widetilde \tau=ie^{-i\psi}\tau$ in place of  $\tau,$
we obtain that
\begin{align*}
K \int_{e^{i\psi}\R} |f(z)|\, |dz| = K \|\widetilde f\|_{H^1(\Sigma_{\pi/2})} 
\ge \|T(\widetilde \tau)\widetilde f\|_{H^{1}(\Sigma_{\pi/2})}\ge \int_{\Gamma^+_{\tau,\varphi}} |f(z)|\, |dz|.
\end{align*}
Similarly 
\[
K \int_{e^{-i\psi}\R} |f(z)|\, |dz| \ge \int_{\Gamma^-_{\tau,\varphi}} |f(z)|\, |dz|.
\]
Thus, taking  into account Theorem \ref{hardy0}(iv),
we infer that
\begin{equation*}
2K\|f\|_{H^1(\Sigma_{\psi})}\ge K \left(\int_{e^{i\psi}\R} |f(z)|\, |dz| +\int_{e^{-i\psi}\R}|f(z)|\, |dz|\right) \ge 
 \int_{\Gamma_{\tau,\varphi}} |f(z)|\, |dz|.
\end{equation*}
Since as above the choice  of $\varphi \in (0,\psi)$ is arbitrary, we then have 
\[
\|T(\tau)f\|_{H^1(\Sigma_\psi)}\le 2K \|f\|_{H^{1}(\Sigma_\psi)}
\]
if $\psi \in (\pi/2,\pi),$ and then if $\psi \in (0,\pi)$ in view of \eqref{hardy1}.

Hence for all  $\tau \in \Sigma_\psi,$ 
 \begin{equation}\label{hardy3}
\|T(\tau)g\|_{\mathcal H_\psi}\le \max(1,2K) \|g\|_{\mathcal H_\psi}\le 2K \|g\|_{\mathcal H_\psi}.
\end{equation}

A direct verification shows that $\|T(\tau)r_\l-r_\l\|_{\H_\psi} \to 0, \tau \to 0$,  for every $\l \in \C \setminus \overline{\Sigma}_\psi$.   
It now follows from \eqref{hardy3} and Theorem \ref{DM122} that $(T(\tau))_{\tau\in \Sigma_{\psi_0}}$ is a bounded $C_0$-semigroup on $\H_{\psi,0}$ and then on $\mathcal H_{\psi}$. The holomorphy of  $(T(\tau))_{\tau\in \Sigma_{\psi_0}}$ on $\mathcal H_\psi$ follows from Theorem \ref{hardy2}(iii) and the method discussed in Section \ref{prelims}.

The claim about the generator $-A_{\mathcal H_{\psi}}$ can be justified 
along the lines of the proof of a similar fact for the space $\Bes$ in \cite[Lemma 2.6]{BGT}.
\end{proof}
\begin{rem}
Following a more conventional approach, one may try to prove Theorem \ref{Hsg} by reducing the estimates to the half-plane case and applying Carleson's embedding theorem for $H^1(\mathbb C_+)$.
However, the technical details become rather cumbersome, so we prefer to use Gabriel's inequality allowing for a more transparent argument.
\end{rem}

\section{Appendix: The $\D$-calculus vs the HP- and the $\mathcal B$-calculi}\label{VS}

It is natural to compare the strength of the $\mathcal D$-calculus with some other functional calculi, such as the recently constructed $\mathcal B$-calculus and the well-known HP-calculus. 
To show the advantages of our $\mathcal D$-calculus  with respect to the $\mathcal B$-calculus and the HP-calculus, 
as an illustrative example, we consider the family of functions $\{f_n: n \ge 1\}$ given by 
\[
f_n(z)=\left(\frac{z-1}{z+1}\right)^n,\qquad z\in \C_{+},\quad n\in \N.
\]

This family is contained in $\mathcal{LM}$ (see \cite[Section 6]{BGT2}), and it arises naturally in the study of asymptotics for powers of Cayley transforms of semigroup generators.
It is shown in \cite[Lemma 3.7]{BGT} and \cite[Lemmas 5.1 and 5.2]{BGT2}  that
\[
\|f_n\|_\Bes \asymp  \log n \qquad \text{and} \qquad \|f_n\|_{{\rm HP}} \asymp  n^{1/2}, \qquad n\to\infty.
\]
We will show that the  $\mathcal D$-calculus provides sharper estimates for the corresponding operator functions. 
 To this aim we need the next lemma.

\begin{lemma}\label{cayley}
For $s>0$,
\[
\|f_n\|_{\mathcal{D}_s}\le 1+16\left(B(s/2,1/2)+2^{-s/2}\right),\qquad n\in \N.
\]
\end{lemma}

\begin{proof}
Let $s>0$ and $n \in \mathbb N$ be fixed. We have
\[
f_n(\infty)=1,\quad f'_n(z)=2n\frac{(z-1)^{n-1}}{(z+1)^{n+1}},
\]
and then
\[
\|f_n\|_{\mathcal{D}_s}=1+8n \int_0^{\pi/2} \cos^s\psi\int_1^\infty \frac{g_n(\rho,\psi)\,d\rho}{\rho^2+2\rho\cos\psi+1}\,d\psi =1+8n J_n,
\]
where
\[
 g_n(\rho,\psi)=\left(\frac{\rho^2-2\rho\cos\psi+1}{\rho^2+2\rho\cos\psi+1}\right)^{(n-1)/2}.
\]

Let $J_n= J_{1, n} + J_{2,n}+J_{3,n}$,
where
\begin{align*}
J_{1,n}&=\int_0^{\pi/2} \cos^s\psi\int_2^\infty \frac{g_n(\rho,\psi)\,d\rho}{\rho^2+2\rho\cos\psi+1}\,d\psi,\\
J_{2,n}&=\int_{\pi/4}^{\pi/2} \cos^s\psi\int_1^2\frac{g_n(\rho,\psi)\,d\rho}{\rho^2+2\rho\cos\psi+1}\,d\psi,\\
J_{3,n}&=\int_0^{\pi/4}\cos^s\psi\int_1^2\frac{g_n(\rho,\psi)\,d\rho}{\rho^2+2\rho\cos\psi+1}\,d\psi.
\end{align*}

We estimate each of the summands $J_{1, n}$, $J_{2,n}$, and $J_{3,n}$ separately.
First,
\begin{align*}
\frac{\partial}{\partial \rho}\,g_{n+2}(\rho,\psi)&=2(n+1)\frac{(\rho^2-1)\cos\psi}{(\rho^2+2\rho\cos\psi+1)^2}g_{n}(\rho,\psi)\\
&\ge \frac{2(n+1)}{3} \frac{\cos\psi\,g_n(\rho,\psi)}{(\rho^2+2\rho\cos\psi+1)}>0,\quad \rho>2.
\end{align*}
Hence
\begin{align*}
J_{1,n}&\le \frac{3}{2(n+1)}\int_0^{\pi/2}\cos^{s-1}\psi
\int_2^\infty
\frac{\partial}{\partial \rho}\,g_{n+2}(\rho,\psi)\,d\rho\,d\psi\\
&\le \frac{3}{2(n+1)}\int_0^{\pi/2}\cos^{s-1}\psi\,d\psi<\frac{B(s/2,1/2)}{n}.
\end{align*}

Next, for all $\rho\in (1,2)$ and $\psi\in(\pi/4,\pi/2)$,
\begin{align*}
\frac{\partial}{\partial \psi}\,g_{n+2}(\rho,\psi)&=2(n+1)\frac{\rho(\rho^2+1)\sin\psi}{(\rho^2+2\rho\cos\psi+1)^2}g_{n}(\rho,\psi)\\
&\ge \frac{(n+1)}{2} \frac{\,g_n(\rho,\psi)}{(\rho^2+2\rho\cos\psi+1)}>0,
\end{align*}
so
\begin{align*}
J_{2,n}&\le \frac{2}{(n+1)}
\int_1^2
\int_{\pi/4}^{\pi/2}\cos^s\psi
\frac{\partial}{\partial \rho}\,g_{n+2}(\rho,\psi)\,d\psi\,d\rho\\
&\le \frac{2}{(n+1)} \int_{\pi/4}^{\pi/2}(\cos^s\psi)'\,d\psi
=\frac{2}{2^{s/2}(n+1)}.
\end{align*}

Finally, since
\begin{align*}
\sup_{\rho\in (1,2),\,\psi\in (0,\pi/4)} g_n(\rho, \psi)
&=g_n(2,\pi/4)
=\left(\frac{5-2\sqrt{2}}{5+2\sqrt{2}}\right)^{(n-1)/2}
<\frac{2}{n},
\end{align*}
we have
\[
J_{3,n}\le
\frac{2}{n}
\int_0^{\pi/4}\cos^s\psi\,d\psi
<\frac{B(s/2,1/2)}{n}.
\]
Summing up the above estimates above, the assertion of the lemma
follows.
\end{proof}

The following corollary showing sharpness of the $\mathcal D$-calculus is immediate.
\begin{cor}
For any $s>0$,
\[
\frac{\|f_n\|_\Bes}{\|f_n\|_{\mathcal{D}_s}} \asymp \log n, \qquad \frac{\|f_n\|_{{\rm HP}}}{\|f_n\|_{\mathcal{D}_s}} \asymp n^{1/2}, \quad n\to\infty.
\]
\end{cor}

\begin{rem}
Curiously, for $s=0$ the asymptotics of $\|f_n\|_{\mathcal D_0}$ match those of $\|f_n\|_{\mathcal B}$,
and one does not get any advantages using the $\mathcal D$-calculus in this case.  Specifically, one can show that, for $n \in \N$,
\begin{equation}\label{ApLog1}
1+2 \log (n+1)\le \|f_n\|_{\mathcal{D}_0}\le 8(4+\log(n+1)).
\end{equation}
\end{rem}

\end{document}